\documentclass[10pt]{amsart}

\usepackage[leqno]{amsmath}
\usepackage{amsthm}
\usepackage{amsfonts}
\usepackage{amssymb}
\usepackage{eucal}
\usepackage[all]{xy}
\usepackage[usenames,dvipsnames]{color}
\usepackage[colorlinks=true]{hyperref}
\CompileMatrices

\DeclareFontFamily{OT1}{rsfs}{}
\DeclareFontShape{OT1}{rsfs}{n}{it}{<-> rsfs10}{}
\DeclareMathAlphabet{\mathscr}{OT1}{rsfs}{n}{it}
\renewcommand*{\setminus}{-}

\DeclareMathOperator{\Gal}{Gal}

\DeclareMathOperator{\Spec}{Spec}

\DeclareMathOperator{\Spf}{Spf}

\def\whka{{\widehat{k^a}}}
\def\tilD{{\widetilde D}}
\def\tilT{{\widetilde T}}
\def\wHx{{\wt{\calH(x)}}}
\def\wHy{{\wt{\calH(y)}}}
\def\..{{,\dots,}}
\def\:{{\colon}}
\def\trdeg{{\rm trdeg}}
\def\Frac{{\rm Frac}}
\def\Int{{\rm Int}}
\def\bir{{\rm bir}}

\def\Im{{\rm Im}}
\def\wt{\widetilde}
\def\into{\hookrightarrow}
\def\wtimes{\widehat{\otimes}}
\def\oA{{\overline{A}}}

\def\hatX{{\widehat{X}}}
\def\hatU{{\widehat{U}}}
\def\hatV{{\widehat{V}}}
\def\hatW{{\widehat{W}}}

\def\hatcalF{{\widehat{\calF}}}
\def\tilf{{\wt f}}
\def\tilg{{\wt g}}
\def\tilh{{\wt h}}
\def\bfR{{\bf R}}
\def\tilk{{\wt k}}

\def\tilF{{\wt F}}
\def\tilK{{\wt K}}
\def\tilL{{\wt L}}

\def\tilU{{\wt U}}
\def\tilV{{\wt V}}
\def\tilW{{\wt W}}
\def\tilX{{\wt X}}
\def\tilY{{\wt Y}}

\def\tilx{{\wt x}}
\def\tily{{\wt y}}
\def\tila{{\wt a}}
\def\tilpsi{{\wt\psi}}
\def\tilphi{{\wt\phi}}

\def\bfP{{\mathbf P}}
\def\calO{{\mathcal O}}

\def\tilcalA{{\wt{\calA}}}

\def\tilcalC{{\wt{\calC}}}

\def\tilbfP{{\wt{\bfP}}}

\def\varcalX{{\mathcal X}}
\def\varcalY{{\mathcal Y}}

\def\vcalX{{\mathcal X}}
\def\vcalY{{\mathcal Y}}
\def\vcalZ{{\mathcal Z}}

\newcommand*{\R}{\ensuremath{\mathbf{R}}}                        
\newcommand*{\Z}{\ensuremath{\mathbf{Z}}}                        
\newcommand*{\Q}{\ensuremath{\mathbf{Q}}}                        
\newcommand*{\A}{\ensuremath{\mathbf{A}}}                        

\renewcommand*{\P}{\ensuremath{\mathbf{P}}}                        
\newcommand*{\calA}{\mathscr{A}}
\newcommand*{\calB}{\mathscr{B}}
\newcommand*{\calF}{\mathscr{F}}                               
\newcommand*{\calH}{\mathscr{H}}                               

\newcommand*{\calX}{\mathscr{X}}

\newcommand*{\calY}{\mathscr{Y}}

\theoremstyle{plain}
  \newtheorem{theorem}[subsection]{Theorem}
  
  \newtheorem{lemma}[subsection]{Lemma}
  \newtheorem{corollary}[subsection]{Corollary}

\theoremstyle{definition}
  \newtheorem{definition}[subsection]{Definition}

\theoremstyle{remark}
  \newtheorem{example}[subsection]{Example}
  \newtheorem{remark}[subsection]{Remark}

\numberwithin{equation}{section}

   \DeclareMathOperator{\calC}{\mathscr C}

   \DeclareMathOperator{\calM}{\mathscr M}

\title{Descent for non-archimedean analytic spaces}
\author{Brian Conrad}
\address{Department of Mathematics\\
Stanford University\\
Building 380, Sloan Hall\\
Stanford, CA 94305\\
U. S. A.}
\email{conrad@math.stanford.edu}
\author{Michael Temkin}
\address{Einstein Institute of Mathematics\\
               The Hebrew University of Jerusalem\\
                Edmond J. Safra Campus, Giv'at Ram, Jerusalem, 91904, Israel}
\email{michael.temkin@mail.huji.ac.il}
\date{\today}
\thanks{The work of B.C. was partially supported by NSF grant DMS-0600919.
Also, B.C. is very grateful to Columbia University for its generous hospitality during a sabbatical visit. The first part of this work was done when M.T. was visiting the Max Planck Institute for Mathematics, and he wishes to thank the Institute for its hospitality. The work on the revised version was supported by the Israel Science Foundation (grant No. 1159/15).}
\begin{document}

\maketitle

\begin{abstract}
In this paper we study two types of descent in the category of Berkovich analytic spaces: flat descent and descent with respect to an extension of the ground field. Quite surprisingly, the deepest results in this direction seem to be of the second type, including the descent of properties of being a good analytic space and being a morphism without boundary.
\end{abstract}

\section{Introduction}

\subsection*{Motivation}
In the theory of schemes, faithfully flat descent is a very powerful tool. One wants a descent theory not only for quasi-coherent sheaves and morphisms of schemes (which is rather elementary), but also for geometric objects and properties of morphisms between them.   In rigid-analytic geometry, descent theory for coherent sheaves was worked out by Bosch and G\"{o}rtz \cite[3.1]{bosch} under some quasi-compactness hypotheses by using Raynaud's theory of formal models, and their result can be generalized \cite[4.2.8]{relamp} to avoid quasi-compactness assumptions (as is necessary to include analytifications of faithfully flat maps arising from algebraic geometry \cite[\S2.1]{ct}).  Similarly, faithfully flat descent for morphisms, admissible open sets, and standard properties of morphisms works out nicely in the rigid-analytic category \cite[\S4.2]{relamp}.

In Berkovich's theory of $k$-analytic spaces, one can ask if there are similar results. The theory of flatness in $k$-analytic geometry is more subtle than in the case of schemes or complex-analytic spaces, ultimately because morphisms of $k$-affinoid spaces generally have non-empty relative boundary (in the sense of \cite[2.5.7]{berbook}).  In the case of quasi-finite morphisms \cite[\S3.1]{berihes}, which are maps that are finite locally on the source and target, it is not difficult to set up a satisfactory theory of flatness \cite[\S3.2]{berihes}. A general theory of flatness was recently developed by Ducros in a 270-page-long manuscript \cite{flat}. (In fact, the original plan of his project was to write a short appendix to the current paper on flat morphisms without boundary, and we decided to postpone submitting our paper until the project of Ducros is completed.)

Let $\mathbf{P}$ be a property of morphisms that are preserved by base changes. In this paper we study twenty such properties listed on the next page, and a very important case is the property of a morphism $f\:X \rightarrow Y$ being {\em without boundary} in the sense that  for any $k$-affinoid $W$ and morphism $W \rightarrow Y$, the base change $X \times_Y W$ is a good $k$-analytic space (i.e., each point has a $k$-affinoid neighborhood) and the morphism of good spaces $X \times_Y W \rightarrow W$ has empty relative boundary in the sense of \cite[p.~49]{berbook}. We say that $\bfP$ is {\em local for a topology $\tau$} if for any morphism $g\:X\to Y$ and  a $\tau$-covering $Y'\to Y$ such that the base change $Y'\times_YX\to Y'$ satisfies $\bfP$, one also has that $g$ satisfies $\bfP$. A simple example with non-admissible covers shows that being without boundary is not local for the flat topology, but we will prove that one does have locality for a slightly weaker Tate-flat topology.

Similarly, if $Y$ is a $k$-analytic space and it has a cover $Y' \rightarrow Y$ without boundary such that $Y'$ is good then it is natural to expect that $Y$ is good but this does not seem to follow easily from the definitions since the target of a finite surjective morphism with affinoid source can be non-affinoid \cite{liu}.

Finally, one can also ask for analogous descent results with respect to extension of the ground field.  That is, if $f\:X \rightarrow Y$ is a map of $k$-analytic spaces and if $K/k$ is an arbitrary analytic field extension then we ask if $f$ satisfies a property $\mathbf{P}$ precisely when $f_K\:X_K \rightarrow Y_K$ satisfies this same property.   Likewise, if $Y_K$ is good then is $Y$ good (the converse being obvious)? This latter question seems to be very non-trivial, and in general the problem of descent through a field extension is much harder than descent through Tate-flat covers. The purpose of this paper is to apply the theory of reduction of germs (as developed in \cite{temkin2}) to provide affirmative answers to all of  the above descent questions.

\begin{remark}\label{goodrem}
What we are calling a morphism {\em without boundary} is called a {\em closed} morphism in \cite[1.5.3(ii)]{berihes} and \cite{temkin2}. It is a non-trivial fact \cite[5.6]{temkin2} that whether or not $f\:X \rightarrow Y$ is without boundary can be checked locally for the $G$-topology on $Y$. This locality is very useful when checking that an abstractly constructed map is without boundary, and it is the reason that for any ground field extension $K/k$ the map $f_K$ is without boundary when $f$ is without boundary. In contrast, the relative notion of  good morphism \cite[1.5.3(i)]{berihes} is not local for the $G$-topology on the base (one can construct a counterexample using \cite[Rem.~1.6,~Thm.~3.1]{temkin1}), so we do not consider it to be an interesting concept for its own sake (although morphisms that are either proper or without boundary are good by definition).
\end{remark}

\subsection*{Main results}\label{intromainsec}
{In this paper we study descent of properties of morphisms from the following list: (i) surjective, (ii) has zero-dimensional fibers, (iii) unramified, (iv) monomorphism, (v) $G$-surjective, (vi) locally separated, (vii) boundaryless, (viii) quasi-finite, (ix) topologically separated, (x) topologically proper, (xi) separated, (xii) proper, (xiii) finite, (xiv) closed immersion, (xv) flat, (xvi) quasi-smooth, (xvii) quasi-\'etale, (xviii) \'etale, (xix) open immersion, (xx) isomorphism, and we will use this numeration in the sequel. See Definitions \ref{Gsurdef}, \ref{compsurdef} and \ref{flatdef}.

Any property $\mathbf{P}$ in the above list is stable under compositions, base changes and extensions of the ground field. Let $f\:Y'\to Y$ be a morphism of $k$-analytic spaces, and let $K/k$ be a ground field extension. We say that $\mathbf{P}$ {\em descends} or {\em is local} with respect to $f$ (resp. $K/k$) if for any morphism $g\:X\to Y$ such that the base change $g'\:X'\to Y'$ (resp. the ground field extension $g_K\:X_K\to Y_K$) satisfies $\mathbf{P}$, one also has that $g$ satisfies $\mathbf{P}$. Our first main result is combined from Theorems~\ref{Tateflatth} and \ref{finited}, and it establishes descent of all twenty properties with respect to arbitrary Tate-flat covers and ground field extensions:

\begin{theorem}\label{mainth1}
The above properties (i)--(xx) are local with respect to arbitrary Tate-flat covers $Y\to Y'$ and ground field extensions $K/k$. More concretely, a morphism $g\:X\to Y$ satisfies any property $\mathbf{P}$ from the list if and only if the base change $g'\:X'\to Y'$ satisfies $\mathbf{P}$ if and only if the ground field extension $g_K\:X_K\to Y_K$ satisfies $\mathbf{P}$.
\end{theorem}

In both cases, our proof heavily uses that certain properties can be expressed in terms of others. The Tate-flat descent is (surprisingly) easier, and its main steps are as follows: (1) by essentially set-theoretic methods we prove that monomorphisms are local for the class of all surjective morphisms, (2) by $G$-set-theoretic methods (e.g., reduction of germs and birational descent) we prove that being without boundary descends for the smaller class of all $G$-surjective morphisms, (3) using topological methods we prove that topological properness descends for an even smaller class of properly surjective morphisms, (4) naturally, flatness is local for faithfully flat morphisms. Notice also that a Tate-flat covering is a $G$-surjective flat covering, and by Theorem~\ref{Tateflatth} this is the same as a properly surjective flat covering. Furthermore, in Theorems \ref{surjth}, \ref{Gsurth}, \ref{compsurth}, \ref{flatth}, and Examples \ref{surexam}, \ref{Gsurexam} and \ref{compsurexam} we completely describe properties satisfying descent of each type and here is the summary:

\begin{theorem}\label{mainth2}
Consider the properties (i)--(xx) above, then

(1) (i)--(iv) are precisely the properties local with respect to all surjective morphisms.

(2) (i)--(viii) are precisely the properties local with respect to all $G$-surjective morphisms.

(3) (i)--(xiv) are precisely the properties local with respect to all properly surjective morphisms.

(4)(i)--(iv) and (xv)--(xvii) are precisely the properties local with respect to all faithfully flat morphisms.
\end{theorem}

The proof of descent with respect to ground field extensions is very similar up to one significant exception: $G$-surjective descent, including descent of $G$-surjectivity and of the boundaryless property, is much more difficult and involves a deeper descent result for graded birational spaces. Finally, in Theorems \ref{goodcase} and \ref{strictK} we use the same results on birational spaces to also study descent of two absolute properties of analytic spaces that have no relative analogs: being good and being strictly $k$-analytic (or, more generally, $H$-strict as defined in Section~\ref{Hstrsec}). Here we only formulate the main result for goodness.

\begin{theorem}\label{mainth3}
Let $f\:Y\to X$ be morphism of $k$-analytic spaces with surjective interior and let $K/k$ be an extension of the ground field, then

(1) If $Y$ is good then $X$ is good, and the converse is true if $f$ has no boundary.

(2) $X$ is good if and only if $X_K$ is good.
\end{theorem}

\subsection*{Methods and outline of the paper}
In Sections \ref{birsec1}--\ref{app} we work out the easier instances of descent, culminating in the Tate-flat one. First, we establish in \S\ref{birsec1} simple cases of  descent of graded birational spaces, and apply this to solve a question from \cite{ct}. Then we study in \S\ref{setsec} set-theoretic and topological descents. The most subtle result is descent of being without boundary, which is based on the following facts: a difficult theorem \cite[5.2]{temkin2} reduces this property to properness of the reduction map of birational spaces (which is a purely $G$-set-theoretic condition), Lemma~\ref{birlem2} establishes descent for birational spaces and Theorem~\ref{surjlem} produces enough test points of $X\times_YY'$ over a compatible pair $x\in X,y'\in Y'$. Finally, in \S\ref{app} we recall some basic facts about flat morphisms introduced by Ducros, study flat topologies and related descent issues, and accomplish the proof of Tate-flat descent. We want to stress that there are various technicalities specific for analytic geometry. In particular, unlike algebraic geometry, most properties are not local for the flat topology, and one has to weaken it to the Tate-flat one. On the one hand, the latter is specific for analytic geometry, but on the other hand it has certain similarities with the classical fpqc topology, see Remark~\ref{Tateflatrem} and Theorem~\ref{Tateflat}.
}

Rather more difficult is descent for the absolute properties of goodness and strict-analyticity of analytic spaces, which we treat (under some separatedness hypotheses) in Theorem~\ref{goodcase}. These harder results rest on a difficult descent theorem for graded birational spaces (Theorem \ref{birdescent}) that we prove as a consequence of a lot of work in ``graded commutative algebra'' in \S\ref{gradsec}--\S\ref{birgsec}. For example, goodness for strictly $k$-analytic germs (when $|k^{\times}| \ne 1$) is closely related to affineness for ungraded birational spaces over the residue field $\tilk$, so to study descent for goodness one is led to seek an analogue of affineness in classical birational geometry.  As a warm-up, in \S\ref{ungradsec} we digress to show that affineness for a birational space over $\tilk$ is equivalent to a certain auxiliary integral $\tilk$-scheme of finite type having normalization that is proper over an affine algebraic $\tilk$-scheme. This latter property is a substitute for affineness in birational geometry, and it is a delicate fact that this property descends through modifications:  such birational invariance fails if the normalizations are omitted, as shown by an elegant example of de Jong (Example \ref{deJongexample}). De Jong's counterexample gave us new insight into Q. Liu's surprising examples of non-affinoid rigid spaces with affinoid normalization (see Example \ref{alaLiuexample}), and it also inspired the intervention of integral closures in both \S\ref{gradsec} and the proof of our main descent theorem for birational spaces (Theorem \ref{birdescent}). Only while working on the revision of this paper we discovered that varieties admitting a proper morphism to an affine variety were called semiaffine and studied in \cite{semiaffine}. In particular, descent under properness and necessity of normalization rediscovered by us are originally due to Goodman and Landman.

Though we expect our results to be of general interest in $k$-analytic geometry, perhaps of greater interest is the techniques of proof.  For example, to descend properties after a field extension $K/k$ the essential difficulty is that {if $y\in Y=X_K$ is a preimage of a point $x$ on a $k$-analytic space $X$, it may freely happen that the reduction map $\tilY_y\to\tilX_x$ is not surjective. In particular, a priori it is even not clear if the natural map $h\:X_K\to X$ is $G$-surjective. An additional technical complication is that the topological fibers $h^{-1}(x)$} are of the form $\calM(\calH(x)\wtimes_k K)$, and these can have a rather complicated structure since the $K$-Banach algebra $\calH(x) \wtimes_k K$ can fail to be $K$-affinoid. We will not try to describe such fibers in their entirety, nor their Shilov boundary (see \cite[p.~36]{berbook}) which may be infinite, but we will prove in Theorem~\ref{surjlem} that any point of the homogeneous spectrum of $\wHx\otimes_\tilk\tilK$ can be lifted to a point of $\calM(\calH(x)\wtimes_k K)$. Lifting a generic point of the homogeneous spectrum provides a generic enough point $y=x_K$ over $x$ for which the map $\tilY_y\to\tilX_x$ is surjective, see Lemma~\ref{genpointcor2}.

The latter property of $x_K$ is critical for using it as a test point: many local properties $\bfP$ hold at $x\in X$ if and only if they hold at $x_K$. In particular, $x_K$ is used to prove descent of goodness in Theorem \ref{mainth3}(2) (see Theorem~\ref{strictK}). Another absolute property for which descent through a ground field extension $K/k$ is non-trivial and interesting is the strict analyticity property (under the necessary assumption $\surd{|K^{\times}|} = \surd{|k^{\times}|}$). We treat this in Theorem \ref{strictK} subject to some separatedness hypotheses. Again, the argument essentially uses a test point $x_K$. The same strategy is also used to prove the second claim of Theorem~\ref{mainth1} (see Theorem~\ref{finited}), with the main step being descent of the property to be without boundary through ground field extensions. Though for this application existence of a single test point is not crucial, as one could also argue similarly to the proof of Theorem~\ref{Gsurth}(vii).

As an interesting application of the invariance of the closed immersion property with respect to arbitrary analytic ground field extensions we conclude the paper by carrying over to $k$-analytic geometry the basic results for relatively ample line bundles in the rigid-analytic case \cite[\S3]{relamp}. We do not know how to adapt the rigid-analytic arguments of \cite[\S3]{relamp} to work in the $k$-analytic case (especially without goodness hypotheses), so we do not obtain a new approach to relative ampleness in rigid geometry.

Since our proofs rely extensively on the theory of reduction of germs as developed in \cite{temkin2},  we assume that the reader is familiar with this work and we will use its terminology and notation, including the theory of birational spaces over the $\R^{\times}_{>0}$-graded field $\widetilde{k}$ and the ``graded commutative algebra'' in \cite[\S1]{temkin2}. As usual in the theory of $k$-analytic spaces, we permit the possibility that the absolute value on $k$ may be trivial.

\subsection*{A comparison with the preliminary version}
A preliminary version of this paper is available since 2008 and was cited in quite a few papers, so we want to indicate the changes made in the final version. The main change is that following a request of Tony Yue Yu we study descent for the Tate-flat topology rather than for the flat boundaryless one. This resulted in complete rewriting of Sections~\ref{birsec1}--\ref{app} devoted to the flat descent. In particular, descent of few new properties was included and we analyse in detail set-theoretic and topological aspects of the descent. In addition, some material in other sections was rearranged for expository reasons, so numeration shifted substantially.

\subsection*{Terminology and notation}
For an abelian group $G$ and a $G$-graded field $\widetilde{k}$, by a {\em $G$-graded birational space over $\widetilde{k}$} we mean an object $\mathcal X$ of the category $\bir_{\widetilde{k}}$ introduced in \cite{temkin2}. Such an $\varcalX$ consists of a $G$-graded field $\widetilde{K}$ over $\widetilde{k}$ and a local homeomorphism $X\to\bfP_{\widetilde{K}/\widetilde{k}}$ where $X$ is a non-empty, connected, quasi-compact, and quasi-separated topological space and $\bfP_{\widetilde{K}/\widetilde{k}}$ is the naturally topologized set of $G$-graded valuation rings of $\widetilde{K}$ containing $\widetilde{k}$. In particular, by taking $X = \P_{\widetilde{K}/\widetilde{k}}$ we can view $\P_{\widetilde{K}/\widetilde{k}}$ as a $G$-graded birational space over $\widetilde{k}$.  When the group $G$ is understood from context, we will not mention it explicitly (and will simply speak of ``graded'' objects). Note also that $\P_{\widetilde{K}/\widetilde{k}}$ is irreducible and sober (that is, the topology is $T_0$ and irreducible closed sets have unique generic points), hence the same necessarily holds for $X$.

For any extension $\widetilde{L}/\widetilde{K}$ of graded fields over $\widetilde{k}$, restriction of graded valuation rings from $\widetilde{L}$ to $\widetilde{K}$ induces a continuous map $\psi_{\widetilde{L}/\widetilde{K}/\widetilde{k}}\: \bfP_{\widetilde{L}/\widetilde{k}}\to\bfP_{\widetilde{K}/\widetilde{k}}$. A {\em morphism} $\varcalY\to\varcalX$ from the birational space $\varcalY = (Y\to\bfP_{\widetilde{L}/\widetilde{k}})$ to the birational space $\varcalX = (X\to\bfP_{\widetilde{K}/\widetilde{k}})$ is a pair consisting of a graded embedding $\widetilde{K}\into \widetilde{L}$ over $\widetilde{k}$ and a continuous map $Y\to X$ compatible with $\psi_{\widetilde{L}/\widetilde{K}/\widetilde{k}}$.

An {\em analytic extension} $K/k$ is a field extension in which $K$ is endowed with an absolute value that extends the given one on $k$ and with respect to which $K$ is complete.  If $X$ is a $k$-analytic space then $X_K$ denotes the $K$-analytic space $X \widehat{\otimes}_k K$. If $x$ is a point in a $k$-analytic space $X$ then we write $(X,x)$ to denote the associated germ (denoted $X_x$ in \cite{temkin2}). A $k$-analytic space $X$ is {\em locally separated} if each $x \in X$ admits a neighborhood that is a separated $k$-analytic domain. By \cite[4.8(iii)]{temkin2}, this happens if and only if each $\R^{\times}_{>0}$-graded birational space $\widetilde{X}_x$ over the $\R^{\times}_{>0}$-graded reduction $\widetilde{k}$ is separated, that is, the map $\tilX_x\to\bfP_{\wHx/\tilk}$ is injective.

A $k$-analytic map $f\:X \rightarrow Y$ is {\em without boundary} (or {\em has no boundary}) if, for any $k$-affinoid $Y'$ and morphism $Y' \rightarrow Y$, the pullback $X' = X \times_Y Y'$ is a good space and the morphism of good spaces $X' \rightarrow Y'$ has empty relative boundary.  This concept is called {\em closed} in \cite{berbook}, \cite{berihes}, and \cite{temkin2}, but we prefer the change in terminology to avoid confusion with the unrelated topological notion of a closed map and because the open unit disc is without boundary over $\calM(k)$ whereas the closed unit disc is not.   The {\em relative interior} ${\rm{Int}}(X/Y)$ is the open locus of points $x \in X$ admitting an open neighborhood $V \subseteq X$ such that $V \rightarrow Y$ is without boundary; the complement $\partial(X/Y) = X - {\rm{Int}}(X/Y)$ is the {\em relative boundary} (so $X \rightarrow Y$ is without boundary if and only if $\partial(X/Y)$ is empty).

{\sc Acknowledgements}. We would like to thank A. Ducros for answering our questions and useful discussions. Also, the second author enjoyed discussions with T. Yue Yu that urged us to extend flat descent from the boundaryless case to general Tate-flat covers. Finally, the paper owes a lot to the anonymous referee for pointing our a mistake in our first version of Lemma~\ref{birlem2} and for valuable suggestions and simplifications. In particular, Lemma~\ref{tatered}, a simple proof of Lemma~\ref{oint}, and a few simplifications in the proof of Theorem~\ref{psilem} are due to the referee.

\section{First results on $\bir_\tilk$ and applications}\label{birsec1}

\subsection*{Birational spaces}
The following result is not surprising since maps in ${\rm{bir}}_{\widetilde{k}}$ are analogous to dominant morphisms of varieties.

\begin{lemma}\label{birlem1}
Assume that $f\:\mathcal{X} \rightarrow \mathcal{Y}$ and $g\:\mathcal{Y} \rightarrow \mathcal{Z}$ are maps in ${\rm{bir}}_{\widetilde{k}}$ with $g \circ f$ proper, and assume that either $g$ is separated or $f$ is proper. Then both $f$ and $g$ are proper.
\end{lemma}
\begin{proof}
By the definition of a birational space over $\widetilde{k}$, $\mathcal{Z}$ corresponds to a local homeomorphism $U \rightarrow \mathbf{P}_{\widetilde{K}/\widetilde{k}}$ where $U$ is a non-empty, connected, quasi-compact, and quasi-separated topological space and $\widetilde{K}/\widetilde{k}$ is an extension of $\R^{\times}_{>0}$-graded fields. We similarly have that $\mathcal{Y}$ and $\mathcal{X}$ respectively correspond to local homeomorphisms $U' \rightarrow \mathbf{P}_{\widetilde{K}'/\widetilde{k}}$ and $U'' \rightarrow \mathbf{P}_{\widetilde{K}''/\widetilde{k}}$, and the maps $f$ and $g$ respectively correspond to the left and right squares in a commutative diagram of topological spaces
$$\xymatrix{
{U''} \ar[r] \ar[d] & {U'} \ar[d] \ar[r] & {U} \ar[d] \\
{\mathbf{P}_{\widetilde{K}''/\widetilde{k}}} \ar[r]^\lambda & {\mathbf{P}_{\widetilde{K}'/\widetilde{k}}} \ar[r] &
{\mathbf{P}_{\widetilde{K}/\widetilde{k}}}}$$
in which the maps along the bottom row are the natural pullback maps induced by maps of graded $\widetilde{k}$-algebras $\widetilde{K} \rightarrow \widetilde{K}'$ and $\widetilde{K}' \rightarrow \widetilde{K}''$. Properness (resp. separatedness) of $g$ (if it holds) means that the natural map $\beta\:U' \rightarrow U \times_{\mathbf{P}_{\widetilde{K}/\widetilde{k}}} \mathbf{P}_{\widetilde{K}'/\widetilde{k}}$ is bijective (resp. injective), and the properness of $g \circ f$ says that the natural map $U'' \rightarrow U \times_{\mathbf{P}_{\widetilde{K}/\widetilde{k}}} \mathbf{P}_{\widetilde{K}''/\widetilde{k}}$ is bijective. This latter map factors as the composition of natural maps
$$
U'' \stackrel\alpha\longrightarrow U' \times_{\mathbf{P}_{\widetilde{K}'/\widetilde{k}}} {\mathbf{P}}_{\widetilde{K}''/\widetilde{k}}
\stackrel{\lambda^*\beta}\longrightarrow U \times_{\mathbf{P}_{\widetilde{K}/\widetilde{k}}} \mathbf{P}_{\widetilde{K}''/\widetilde{k}}
$$
in which the second map is the topological (or set-theoretic) base change of the map $\beta$ by the map $\lambda$.

The natural map $\lambda$ along which we form the base change of $\beta$ is surjective, due to the easy consequence of Zorn's Lemma that every graded local ring in a graded field $F$ is dominated by a graded valuation ring in $F$ having graded fraction field $F$. In particular, $\lambda^*\beta$ is bijective (resp. injective) if and only if $\beta$ bijective (resp. injective). Since $f$ is proper if and only if $\alpha$ is bijective, the lemma reduces to the following obvious set-theoretic assertion: assume that the composition $\lambda^*\beta\circ\alpha$ is bijective and either $\lambda^*\beta$ is injective or $\alpha$ is bijective, then both $\lambda^*\beta$ and $\alpha$ are bijective.
\end{proof}

The category $\bir_\tilk$ does not have all fiber products because tensor products of fields do not have to be fields. We will use the following replacement, which is also motivated by \cite[Proposition~4.6]{temkin2}.

\begin{definition}
Let $\varcalY=(Y\to\bfP_{L/\tilk})$ and $\varcalX_i=(X_i\to\bfP_{K_i/\tilk})$ with $i=0,1,2$ be four birational spaces in $\bir_\tilk$. A square diagram
$$
\xymatrix{
\varcalY \ar[r] \ar[d] & \varcalX_1 \ar[d] \\
\varcalX_2 \ar[r]  & \varcalX_0
}
$$
is called {\em quasi-Cartesian} if $Y=X'_2\times_{X'_0}X'_1$, where $X'_i=X_i\times_{\bfP_{K_i/\tilk}}\bfP_{L/\tilk}$.
\end{definition}

\begin{remark}
(1) Clearly, the quasi-Cartesian diagram above is determined by $\vcalX_i$ for $i=0,1,2$ and the graded field $L$ with the embeddings $K_i\into L$. More concretely, using that limits commute with limits it is easy to see that a diagram is quasi-Cartesian if and only if $Y$ is the limit of the diagram obtained from the cube
$$
\xymatrix{
&\bfP_{L/\tilk}\ar[dd]\ar[rr]&&\bfP_{K_1/\tilk}\ar[dd]\\
Y \ar@{-->}[rr] \ar@{-->}[dd]\ar@{-->}[ur] & & X_1 \ar[dd]\ar[ur] \\
&\bfP_{K_2/\tilk}\ar[rr]&&\bfP_{K_0/\tilk}\\
X_2 \ar[rr]\ar[ur] & & X_0\ar[ur] \\
}
$$
by removing $Y$ and the dashed arrows.

(2) Morally, the non-existing fiber product $\varcalX_1\times_{\vcalX_0}\vcalX_2$ should be the disjoint union of the quasi-Cartesian diagrams with the fields $L_t=k(t)$ parameterized by the points $t\in\Spec_G(A)$, where $A=K_1\otimes_{K_0}K_2$ (recall that the latter are the prime homogeneous ideals $t\subset A$, and $k(t)=\Frac_G(A/t)$). This can be made rigorous by considering the category of all disjoint unions of birational spaces, but we will not need this.
\end{remark}

The following descent result is the main motivation for considering quasi-Cartesian diagrams.

\begin{lemma}\label{birlem2}
Let $\vcalX=(X\to\bfP_{K/\tilk})$, $\vcalY=(Y\to\bfP_{L/\tilk})$ and $\vcalX_i=(X_i\to\bfP_{K_i/\tilk})$ with $i\in I$ be birational spaces in $\bir_\tilk$ and let $f\:\varcalY\to\varcalX$ and $\vcalX_i\to\vcalX$ be morphisms such that $X$ is covered by the images of $X_i$. In addition, for each $i\in I$ and $t\in\Spec_G(L\otimes_KK_i)$ let
$$
\xymatrix{
\varcalY_{i,t} \ar[r] \ar[d]_{f_{i,t}} & \varcalY \ar[d]^f \\
\varcalX_i \ar[r]  & \varcalX
}
$$
be a quasi-Cartesian diagram such that $t$ is the image of $\Spec_G(L_{i,t})$, where $\vcalY_{i,t}=(Y_{i,t}\to\bfP_{L_{i,t}/\tilk})$. Then $f$ is proper (resp. separated) if and only if each $f_i$ is proper (resp. separated).
\end{lemma}
\begin{proof}
Only the descent statement needs a proof, so assume that each $f_{i,t}$ is proper (resp. separated). We should prove that for each valuation $v\in\bfP_{L/\tilk}$ with restriction $u\in\bfP_{K/\tilk}$ and a point $x\in X$ over $u$ there exists a unique (resp. at most one) point $y\in Y$ mapped to $x$ and $v$.

By our assumptions, there exists $j\in I$ and $x'\in X_j$ sitting over $x$, and let $u'\in\bfP_{K_j/\tilk}$ be the image of $x'$. Now comes the main point: the pair $(u',v)$ can be lifted to a valuation $v'$ on some $L_{j,t}$ (in other words, the map $\coprod_t\bfP_{L_{j,t}/\tilk}\to\bfP_{K_j/\tilk}\times_{\bfP_{K/\tilk}}\bfP_{L/\tilk}$ is onto). The maps $\bfP_{L_{j,t}/\tilk}\to\bfP_{k(t)/\tilk}$ are surjective, hence it suffices to lift $(u',v)$ to some $k(t)$, that is, to find $t$ such that the graded valuation rings $\calO_{u'}$ and $\calO_v$, viewed as subrings of $k(t)$ via the embeddings $K_j\into k(t)$ and $L\into k(t)$, are dominated by a graded valuation ring of $k(t)$. Moreover, it suffices to show that they are dominated by a local graded subring $\calO\subset k(t)$, because any such $\calO$ is dominated by a graded valuation ring of $k(t)$.

Choose a point  $w\in W=\Spec_G(\calO_{u'}\otimes_{\calO_u}\calO_v)$ over the closed points of $\Spec_G(\calO_{u'})$ and $\Spec_G(\calO_v)$ -- the fiber over this pair coincides with $\Spec_G(k(u')\otimes_{k(u)}k(v))$ and hence is non-empty. It is easy to see that all generic points of $W$ are contained in $W_0=\Spec_G(K_j\otimes_KL)$ (for example, one can use that any torsion free module over a graded valuation ring is flat, see Lemma~\ref{flat}). So, $w$ possesses a generization $t\in W_0$ and then the image $\calO\subset k(t)$ of $\calO_w$ is a local ring dominating $\calO_{u'}$ and $\calO_v$.

Using that the diagram with $\vcalY_{j,t}$ is quasi-Cartesian we obtain that the lifts of $(x,v)$ to a point of $Y$ are in a one-to-one correspondence with the lifts of the pair $(x',v')$ to a point of $Y_{j,t}$. Since $f_{j,t}$ is proper (resp. separated) there is precisely one (resp. at most one) such a lift.
\end{proof}

\begin{remark}\label{genrem0}
In fact, it suffices to consider only finite sets $I$ since the maps $X_i\to X$ are open and $X$ is quasi-compact. Also, one can show (though we will not use this) that it suffices to consider only generic points $t$. In the second part of the paper, in such a case we will say that $K_i$ and $L$ are in general position in $k(t)$ over $K$, see Remark~\ref{genrem}(2).
\end{remark}

\subsection*{Fiber spaces}
Let $k[T]$ be the polynomial algebra in a set of indeterminates $T=(T_i)_{i\in I}$. Any family $r=(r_i)_{i\in I}$ of positive real numbers induces a generalized Gauss norm $|\ |_r$ whose restriction to each $k[T_{i_1}\..T_{i_n}]$ is the usual generalized Gauss norm with radii $r_{i_1}\..r_{i_n}$. In particular, $|\ |_r$ is multiplicative and hence extends to $k(T)$. The completions of $k[T]$ and $k(T)$ with respect to $|\ |_r$ will be denoted $k\{r^{-1}T\}$ and $k_r$. For a finite $I$ these are just the generalized Tate algebra and the completed residue field of the maximal point of a $k$-polydisc of radii $r_1\..r_n$. The following result and its proof were suggested by the referee, it allows to avoid finiteness assumptions in Theorem~\ref{surjlem} below.

\begin{lemma}\label{tatered}
The graded reduction of $\calA=k\{r^{-1}T\}$ is $\tilcalA=\tilk[r^{-1}\tilT]$ and the reduction map $\pi\:\calM(\calA)\to\Spec_G(\tilcalA)$ is surjective, where $G=\bfR^\times_{>0}$.
\end{lemma}
\begin{proof}
By definition, $\calA$ is the completed filtered union of its isometric subalgebras $\calA_J=k\{r_i^{-1}T_i\}_{i\in J}$ with finite subsets $J\subseteq I$. Therefore, $\tilcalA$ is the filtered union of the graded reductions $\tilcalA_J$, which coincide with $\tilk[r_i^{-1}T_i]_{i\in J}$ by \cite[Proposition~3.1(i)]{temkin2}. The first claim follows.

Let us show that a point $\tilx\in\tilX:=\Spec_G(\tilcalA)$ can be lifted to a point $x\in X:=\calM(\calA)$. Choose an extension $k'/k$ such that there is an embedding of graded $\tilk$-fields $k(\tilx)\into\tilk'$, and set $\calA'=k'\{r^{-1}T\}$. Then $\tilcalA'=\tilk'[r^{-1}T]$ by the first claim, hence $\tilX':=\Spec_G(\tilcalA')=\tilX\otimes_\tilk\tilk'$ and the fiber of $\tilX'\to\tilX$ over $\tilx$ contains a $k'$-point $\tilx'$. This point is given by equalities $\tilT_i=\tila_i$ for some homogeneous $\tila_i\in\tilk'$ of grading $r_i$ and lifting them to elements of $k'$ one obtains a $k$-point $x'\in X':=\calM(\calA')$ mapped by the reduction $\pi'\:X'\to\tilX'$ to $\tilx'$. It follows that the image $x\in X$ of $x'$ is a lift of $\tilx$.
\end{proof}

In order to apply Lemma \ref{birlem2} to non-archimedean spaces we will need some control on the spaces $\calM(K\wtimes_kL)$ and their reductions, where $K,L$ are non-archimedean extensions of $k$. Such spaces, for example, appear as fibers of the maps $X_L=X\wtimes_kL\to X$ and they can be fairly complicated, see Examples~\ref{fiberexam1} and \ref{fiberexam2} (these examples also discuss some properties of birational spaces, so we postpone them to \S\ref{lastsec}). Nevertheless, we have:

\begin{theorem}\label{surjlem}
Let $k$ be a non-archimedean field with non-archimedean extensions $K/k$ and $L/k$. Then reduction induces a surjective map $\psi\:\calM(K\wtimes_kL)\to\Spec_G(\tilK\otimes_\tilk\tilL)$, where $G=\bfR^\times_{>0}$.
\end{theorem}
\begin{proof}
The proof runs by establishing a series of particular cases. As a rule, we fix $K$ and play with $L$, so set $X=\calM(K)$, $X_L=X\wtimes_kL=\calM(K\wtimes_kL)$ and $S_\tilL=\Spec_G(\tilK\otimes_\tilk\tilL)$. We should prove that the map $\psi_L\:X_L\to S_\tilL$ composed from the reduction map $X_L\to\tilX_L=\Spec_G(\widetilde{K\otimes_kL})$ and the morphism $\tilX_L\to S_\tilL$ is surjective.

Step 1. {\it If suffices to prove that $\psi_{L'}$ is surjective, where $L'/L$ is an extension of non-archimedean fields.} The maps $\psi_L$ and $\psi_{L'}$ are compatible with the projections $X_{L'}=X_L\wtimes_LL'\to X_L$ and $S_{\tilL'}=S_\tilL\otimes_\tilL\tilL'\to S_\tilL$, which are clearly surjective. In particular, the surjectivity of $\psi_{L'}$ implies the surjectivity of $\psi_L$.

Step 2. {\it The assertion of the theorem holds when $L=\whka$.} This is the only step, in which we play with $K$: by Step 1 and the symmetry between $K$ and $L$ we can assume that $\whka\subseteq K$. Now, the points of $X_L$ correspond to $k$-embeddings $\whka\into K$, hence $X_L$ is a torsor under $G=\Gal(k^s/k)$. Similarly, $S_\tilL$ is a torsor under $H=\Gal(\tilk^s/\tilk)$. It is easy to see that for any graded field $F$ the separable closure $F^s$ is obtained by taking the extension $F\otimes_{F_1}F_1^s$ and extracting all roots of orders invertible in $\tilK$ from its homogeneous elements. It follows that the homomorphism $G\to H$ identifies $H$ with the tame Galois group of $k$, in particular, it is surjective. Therefore, the map $X_L\to S_\tilL$ is surjective.

Step 3. {\it The assertion holds when $L=k_r$ is the completed fraction field of $\calA=k\{r^{-1}T\}$, where $r=(r_i)_{i\in I}$.}
Note that $\calM(L)$ is the point of $D=\calM(\calA)$ cut off by the conditions $|f(x)|=\rho_{\calA}(f)$ for any $f\in\calA\setminus\{0\}$. Similarly, $\Spec_G(\tilL)$ is the point of $\tilD=\Spec_G(\tilk[r^{-1}T])$ cutoff by the conditions $\tilf(x)\neq 0$. Therefore, the same conditions cut off $X_L$ from $D\wtimes_kK$ and $S_\tilL$ from $\tilD\otimes_\tilk\tilK$. It follows that $\psi_L$ is the base change of the reduction map $\calM(K\{r^{-1}T\})\to\Spec_G(\tilK[r^{-1}T])$, which is surjective by Lemma~\ref{tatered}.



Step 4. {\it The assertion holds when $L$ is the completed algebraic closure of some $k_r$.}
Since $\phi_{k_r}$ is surjective by Step 3 it suffices to show that for any $x\in X_{k_r}$ with image $w\in S_{\tilk_r}$ the fiber over $x$ in $X_L$ is mapped surjectively onto the fiber over $w$ in $S_\tilL$. The map of this fibers is the composition $$\calM(\calH(x)\wtimes_{k_r}L)\to\Spec_G(\wHx\otimes_{\tilk_r}\tilL)\to\Spec_G(k(w)\otimes_{\tilk_r}\tilL).$$ The first map is surjective by Step 2 and the second map is surjective in the obvious way.

Step 5. {\it The assertion holds true when $L$ is arbitrary.} By Step 1 we can assume that $L$ is algebraically closed. We claim that $L$ contains a subfield isomorphic to $k_r$ whose completed algebraic closure $F$ satisfies $\tilF=\tilL$. Choose any homogeneous transcendence basis $\tilT=(\tilT_i)_{i\in I}$ of $\tilL/\tilk$ and choose any lifting $T=(T_i)_{i\in I}$ in $L$. Using the algebraic independence of $\tilT_i$ over $\tilk$ it is easy to see that the induced norm on $k[T]$ is the generalized Gauss norm, and hence the closure of $k(T)$ in $L$ is $k_r$. Since $L$ is algebraically closed, $F=\widehat{k_r^a}$ also embeds in $L$ and clearly both $\tilF$ and $\tilL$ coincide with the algebraic closure of $\tilk_r$. Now we have that $S_\tilF=S_\tilL$ and the map $\psi_L$ is composed from the surjective map $X_L\to X_F$ and the map $\psi_F$, which is surjective by Step 4.
\end{proof}

\begin{remark}
It seems plausible that the theorem can be strengthened to the claim that both the reduction map $X_L\to\tilX_L$ and the morphism $f\:\tilX_L\to S_\tilL$ are surjective. However, unlike the situation with products of $k$-analytic spaces, $f$ does not have to be integral, as follows from Example \ref{fiberexam2}.
\end{remark}

\subsection*{Properties of relative interior}
The following result was recorded (under a mild restriction) in \cite[5.7]{temkin2}. Since the proof was omitted there and we now need a more general version, we give the proof in its entirety.

\begin{lemma}\label{57lemma} Let $\psi\:X \rightarrow Y$ and $\varphi\:Y \rightarrow Z$ be morphisms of $k$-analytic spaces.  The relative interiors satisfy $${\rm{Int}}(X/Y) \cap \psi^{-1}({\rm{Int}}(Y/Z)) \subseteq {\rm{Int}}(X/Z),$$ and a point $x \in {\rm{Int}}(X/Z)$ lies in ${\rm{Int}}(X/Y) \cap \psi^{-1}({\rm{Int}}(Y/Z))$ if either $x \in {\rm{Int}}(X/Y)$ or if $\varphi$ is separated on an open neighborhood of $\psi(x)$. In particular, the inclusion is an equality whenever ${\rm{Int}}(X/Y) = X$ or $\varphi$ is locally separated.
\end{lemma}
\begin{proof}
The inclusion is \cite[1.5.5({\rm{ii}})]{berihes}, and for the reverse statement at a point $x \in {\rm{Int}}(X/Z)$ such that $\varphi$ is separated near $\psi(x)$ we can replace our spaces with suitable open subspaces so that $\varphi$ is separated.  In this case the equality is \cite[5.7]{temkin2}, but since  we also want to treat the case $x \in {\rm{Int}}(X/Y)$ without separatedness conditions on $\varphi$ we give the argument here for the convenience of the reader.    We choose $x \in {\rm{Int}}(X/Z)$ and let $y = \psi(x) \in Y$ and $z = \varphi(y) \in Z$.  Consider the induced maps $\widetilde{\psi}\:\widetilde{X}_x \rightarrow \widetilde{Y}_y$ and $\widetilde{\varphi}\:\widetilde{Y}_y \rightarrow \widetilde{Z}_z$ of reductions of germs in the category ${\rm{bir}}_{\widetilde{k}}$ of birational spaces over the $\R^{\times}_{>0}$-graded field $\widetilde{k}$.   By \cite[5.2]{temkin2}, a morphism of $k$-analytic spaces has empty relative boundary near a point of the source if and only if the induced map of reductions of germs at that point and its image is a proper map in the category ${\rm{bir}}_{\widetilde{k}}$ (in the sense of properness defined in \cite[\S2]{temkin2}). The condition $x \in {\rm{Int}}(X/Z)$ therefore says exactly that $\widetilde{\varphi} \circ \widetilde{\psi} = \widetilde{\varphi \circ \psi}$ is a proper map in ${\rm{bir}}_{\widetilde{k}}$, and the condition $x \in {\rm{Int}}(X/Y)$ says exactly that $\widetilde{\varphi}$ is proper. By \cite[4.8(iii)]{temkin2}, separatedness for $\varphi$ near $\psi(x)$ says exactly that $\widetilde{\varphi}$ is a separated map in the category ${\rm{bir}}_{\widetilde{k}}$ (in the sense defined in \cite[\S2]{temkin2}). It remains to use Lemma~\ref{birlem1}.
\end{proof}

As an immediate consequence of this lemma we obtain a result on \'etale equivalence relations that answers a question that naturally arose in \cite{ct} (and was referenced to there as 2.2).

\begin{theorem}\label{deltaclosed}
Let $p_1, p_2\:R \rightrightarrows U$ be a pair of quasi-finite maps of $k$-analytic spaces such that $$\delta = (p_1, p_2)\:R \rightarrow U \times U$$ is a quasi-compact monomorphism. If $U$ is locally separated in the sense that each $u \in U$ has a separated open neighborhood then $\delta$ is a closed immersion.
\end{theorem}

The quasi-compactness of $\delta$ means that the inverse image of a quasi-compact domain is quasi-compact (see also Lemma~\ref{qcmap} below). The case of most interest is when the $p_i$ are \'etale. Recall that \'etale maps are quasi-finite; cf. \cite[3.1.1,~3.3.4]{berihes}.
{
\begin{proof}
Let us first show that $\delta$ has no boundary. By definition of quasi-finiteness in $k$-analytic geometry \cite[3.1.1]{berihes}, the quasi-finite maps $p_1$ and $p_2$ are without boundary.  The first projection $U \times U \rightarrow U$ is locally separated since $U$ is locally separated, so by Lemma \ref{57lemma} we have that ${\rm{Int}}(R/U \times U)$ contains ${\rm{Int}}(R/U)$ (taken with respect to $p_1\:R \rightarrow U$), and this latter interior is $R$.

Next, we claim that $\delta$ is a finite morphism of $k$-analytic spaces. Being a monomorphism $\delta$ is separated, hence topologically separated and hence $\delta$ is actually compact. Recall that a proper morphism is nothing else but a compact morphism without boundary (\cite[1.5.3,~1.5.4]{berihes}), so $\delta$ is proper. The fibers of $\delta$ are discrete, hence it is finite by \cite[1.5.3(iii)]{berihes}.

Finally, we claim that a finite monomorphism of $k$-analytic spaces $X \rightarrow Y$ is a closed immersion.  Indeed, by \cite[1.3.7]{berihes} we can assume $X = \mathscr{M}(\mathscr{A})$ and $Y = \mathscr{M}(\mathscr{B})$ are $k$-affinoid, and since completed tensor products coincide with ordinary tensor products for finite admissible morphisms of $k$-affinoid algebras it follows that the corresponding map $\Spec(\mathscr{A}) \rightarrow \Spec(\mathscr{B})$ is a finite monomorphism of schemes. Hence, it is a closed immersion of schemes, so the finite admissible map of $k$-affinoid algebras $\mathscr{B} \rightarrow \mathscr{A}$ is surjective, as required.
\end{proof}

\section{Set-theoretic and topological descent}\label{setsec}
As another application of the theory of birational spaces, we shall now establish some easier descent results.  In particular, results in graded commutative algebra in \cite{temkin2} cover our needs in this section.

\subsection*{Counterexamples}
In this subsection we consider three examples where certain properties do not descend along a morphism $f$. This will be later used to show that our descent theorems are tight. All three examples are constructed along the same line: one takes $f$ to be a bijective monomorphism with some nasty properties. The base change of $f$ along itself is an isomorphism, and this suffices to obtain almost all no-descent results.

Our first example is specific for the analytic category. The idea is to take a non-admissible covering by subdomains, and here is the simplest but typical example.

\begin{example}\label{surexam}
Let $Y=\calM(k\{t\})$ be the closed unit disc, $Y_0=Y\{|t|\ge 1\}$ the unit annulus in $Y$ and $Y_1=Y-Y_0$ the open unit disc. Then $f\:Y'=Y_0\coprod Y_1\to Y$ is a non-admissible covering and a bijective monomorphism. The following properties are not local with respect to $f$: (v)--(xiv), and (xviii)--(xx). Indeed, the base change of $g=f$ with respect to $f$ is an isomorphism, and this covers all above properties except various forms of separatedness. To show that separatedness and local separatedness do not descent with respect to $f$, consider the space $X$ obtained from two copies of $Y$ by gluing along $Y_0$ (this is $Y$ with doubled open disc). Then $g:X\to Y$ is not locally separated at the maximal point of $X$, but $Y'=X\times_YY'$ is a disjoint union of an annulus and two open discs, hence $g'\:Y'\to X'$ is even separated. Finally, the case of topological separatedness is dealt with similarly, but this time we double the annulus $Y_0$, that is, we define $X$ by gluing two copies of $Y$ along $Y_1$.
\end{example}

The other two examples have algebraic analogs and can arise from analytifications of algebraic $k$-schemes as well. Our second example is obtained by ``decomposing'' a space into a closed subspace and its complement.

\begin{example}\label{Gsurexam}
Let $D$ be the closed unit disk, $D^{\ast} = D - \{0\}$, and $D'$ the disjoint union $D^{\ast} \coprod \{0\}$.  The canonical bijective monomorphism $f\:D' \rightarrow D$ is a disjoint union of an open immersion and a closed immersion. In particular, $f$ is boundaryless. Properties (ix)--(xx) are not local with respect to $f$. For all properties except separatedness and topological separatedness, it suffices to observe that $f$ does not satisfy them, while its base change along itself is an isomorphism. For the latter two, let $X$ denote the non-separated gluing of $D$ to itself along the identity on $D^{\ast}$ then the canonical map $g\:X \rightarrow D$ is not topologically separated but its base change along $f$ is separated because $g$ has separated pullback over $D^{\ast}$ and over the origin.
\end{example}

Our last example is a classical nilpotent thickening, and, again, we consider the simplest possible case.

\begin{example}\label{compsurexam}
Consider the bijective closed immersion $f\:X=\calM(k)\into Y=\calM(k[\varepsilon]/(\varepsilon^2))$. The base change of $f$ along itself is an isomorphism, hence properties (xv)--(xx) are not local with respect to $f$.
\end{example}

\subsection*{Surjective descent}
First, we study which properties of morphisms satisfy purely set-theoretic descent.

\begin{theorem}\label{surjth}
The following properties of morphisms $g\:X\to Y$ are local with respect to the class of surjective base change morphisms $Y'\to Y$: (i) surjective, (ii) has zero-dimensional fibers, (iii) unramified, (iv) monomorphism.
\end{theorem}
\begin{proof}
Let $g'\:X'\to Y'$ be the base change of $g$. Cases (i) and (ii) are obvious: if $g'$ is has non-empty (resp. zero-dimensional) fibers then so does $g$.

(iii) If $g$ is not unramified then the coherent $\calO_{X_G}$-module $\Omega_{X_G/Y_G}$ does not vanish, and we can find a point $x\in X$ with a non-zero fiber $\Omega_{X_G/Y_G}(x):=\Omega_{X_G/Y_G}\wtimes\calH(x)$. Since differentials are compatible with base changes, this implies that for any point $x'\in X'$ above $x$ the fiber $\Omega_{X'_G/Y'_G}(x')$ does not vanish too. In particular, $g'$ is not unramified.

(iv) Assume that $g'$ is a monomorphism. Set $Z=X\times_YX$ and $Z'=X'\times_{Y'}X'$ and note that $Z'=Z\times_XX'$, that is, the diagonal $\delta_{g'}$ of $g'$ is the base change of the diagonal $\delta_g$ of $g$ with respect to the surjective morphism $Z'\to Z$. A morphism is a monomorphism if and only if its diagonal is an isomorphism. In particular, $\delta_{g'}$ is an isomorphism and $I'=\ker(\calO_{Z'_G}\to\calO_{X'_G})$ vanishes. Thus $\Omega_{X'_G/Y'_G}=I'/I'^2=0$, and using (iii) we obtain that $g$ is unramified. In addition, $\delta_g$ is surjective by (ii). Being a diagonal of a morphism, $\delta_g$ is a $G$-locally closed immersion, that is, it is a composition of a closed immersion and embedding of a subdomain. Therefore, $\delta_g$ is a surjective closed immersion, and hence the corresponding ideal $I=\ker(\calO_{Z_G}\to\calO_{X_G})$ is locally nilpotent. We already  know that $I/I^2=\Omega_{X_G/Y_G}=0$, and since $I$ is locally nilpotent this implies that $I=0$. So, $\delta_g$ is an isomorphism and we are done.
\end{proof}

The list of properties in Theorem \ref{surjth} cannot be increased due to Examples \ref{surexam} and \ref{Gsurexam}.

\subsection*{$G$-surjectivity}
Example \ref{surexam} indicates that in order to descent more properties we should better take the $G$-topology into account, so we will now impose a surjectivity assumption for the $G$-topology. In the case of covers by subdomains, this is precisely the condition that makes a covering admissible, ruling out Example~\ref{surexam}. In general, this condition can be defined purely topos-theoretically, but we prefer to stick to a down-to-earth approach with points because not only $G$-topological spaces in analytic geometry have enough points, but these points also possess a simple informative description.

\begin{definition}\label{Gsurdef}
For an analytic space $X$ let $X_G$ denote the corresponding $G$-topological space considered as a site and let $|X_G|$ denote the set of (isomorphism classes of) points of $X_G$. We say that a morphism $f\:X\to Y$ is {\em $G$-surjective} if the map $|X_G|\to|Y_G|$ is surjective.
\end{definition}

\begin{remark}
We will prove in Theorem~\ref{Htop} that $|X_G|=\coprod_{x\in X}\tilX_x$. That proof is self-contained and we are going to use this fact below without causing to a circular reasoning. The reason for postponing our study of spaces $|X_G|$ to the end of \S\ref{Hstrsec} is expository: we will introduce $H$-strict $G$-topologies $X_H$ in \S\ref{Hstrsec}, and Theorem~\ref{Htop} is formulated in that generality. The case we need here is obtained by taking $H=\bfR^\times_{>0}$, it corresponds to Kedlaya's reified adic spaces, see Remark~\ref{adicrem}.
\end{remark}

Recall that for any morphism $\phi\:(X\to\bfP_{L/\tilk})\to(Y\to\bfP_{K/\tilk})$ of graded birational spaces over $\tilk$ the underlying topological map $f\:X\to Y$ is an open map by \cite[Theorem~7.2.5]{flat}. In particular, $(f(X)\to\bfP_{K/\tilk})$ is a birational subspace of $(Y\to\bfP_{K/\tilk})$ that we call the {\em image of $\phi$}.

\begin{lemma}\label{Gsurlem}
Let $f\:X\to Y$ be a morphism of $k$-analytic spaces, then the following conditions are equivalent:

(1) $f$ is $G$-surjective,

(2) $\cup_{x\in f^{-1}(y)}\Im(\tilf_x\:\tilX_x\to\tilY_y)=\tilY_y$ for any point $y\in Y$,

(3) for any point $y\in Y$, there exist finitely many points $x_i\in f^{-1}(y)$ such that $\tilY_y=\cup_i\Im(\tilf_{x_i})$.
\end{lemma}
\begin{proof}
(1)$\Longleftrightarrow$(2) since $|X_G|=\coprod_{x\in X}\tilX_x$, and (2)$\Longleftrightarrow$(3) since each $\tilY_y$ is quasi-compact.
\end{proof}

\subsection*{$G$-surjective descent}
We next wish to discuss descent of properties of morphisms with respect to base change along $G$-surjective morphisms $f\:Y' \rightarrow Y$. We only list properties not covered by the surjective descent. The main case is being without boundary, which is a sort of local properness (for the same reason the analogous notion is called partial properness in adic geometry). We will see that descent of being without boundary can be reduced to set-theoretic issues by germ reductions (in other words, it can be dealt with by set-theoretic methods in the $G$-topology).

\begin{theorem}\label{Gsurth}
The following properties of morphisms $g\:X\to Y$ are local with respect to $G$-surjective morphisms $f\:Y'\to Y$: (v) $G$-surjective, (vi) locally separated, (vii) boundaryless, (viii) quasi-finite.
\end{theorem}
\begin{proof}
Let $g'\:X'\to Y'$ be the base change of $g$. Case (v) is a general property of Grothendieck topologies: if $g'$ is $G$-surjective then the composition $X'\to Y'\to Y$ is $G$-surjective and hence $X\to Y$ is $G$-surjective.

(vii) Assume that $g'$ is without boundary. It suffices to prove that $g$ has no boundary at a point $x\in X$. Set $y=g(x)$ and choose $y_1\..y_n\in f^{-1}(y)$ such that $\tilY_y$ is covered by the images of $\tilY'_{y_i}$, see Lemma~\ref{Gsurlem}. Furthermore, by Theorem~\ref{surjlem} for any $t$ in the homogeneous spectrum $W=\Spec_G(\wHx\otimes_\wHy\widetilde{\calH(y_i)})$, where $G=\bfR_{>0}^\times$, the fiber $f'^{-1}(x)\cap g'^{-1}(y_i)=\calM(\calH(x)\otimes_{\calH(y)}\calH(y_i))$ contains a point $x_{i,t}$ such that $t$ is the image of $\Spec_G(\widetilde{\calH(x_{i,t})})$. On the level of reductions we obtain the following diagrams
$$
\xymatrix{
\tilX'_{x_{i,t}} \ar[r] \ar[d]_{\tilg'_{x_{i,t}}} & \tilX_x \ar[d]^{\tilg_x} \\
\tilY'_{y_i} \ar[r]  & \tilY_y
}
$$
which are quasi-Cartesian by \cite[4.6]{temkin2}. By \cite[5.2]{temkin2} the maps of birational spaces $\tilg'_{x_{i,t}}$ are proper, and using Lemma~\ref{birlem2} we obtain that $\tilg_x$ is proper. It remains to apply \cite[5.2]{temkin2} once again.

(vi) One way is to copy the argument for (vii) and replace the references to \cite[5.2]{temkin2} by references to \cite[4.8(iii)]{temkin2}. Instead of this, we prefer a formal argument with diagonals which also applies to other contexts (e.g. topological properness and separatedness). Assume that $g'$ is locally separated. Note that a morphism is locally separated if and only if its diagonal has no boundary. In particular, $\delta_{g'}$ has no boundary. Recall that $\delta_{g'}$ is the base change of $\delta_g$ along $X'\times_{Y'}X'\to X\times_YX$. The latter morphism is $G$-surjective because it is the base change of the $G$-surjective morphism $Y'\to Y$ along $X\times_YX\to Y$. Thus, $\delta_g$ has no boundary by (vii), and hence $g$ is locally separated.

(viii) By \cite[3.1.10]{berihes}, quasi-finiteness is equivalent to being without boundary and having zero-dimensional fibers. So, this case follows from (vii) and Theorem~\ref{surjth}(ii).
\end{proof}

The list of properties in Theorem \ref{Gsurth} cannot be increased due to Example \ref{Gsurexam}.


\subsection*{Properly surjective descent}
In analytic geometry, properness is a combination of two properties: a morphism is without boundary and the underlying topological map is compact. In addition, in the non-separated case it also makes sense to look at quasi-compact morphisms. These notions turn out to be a bit subtle and the terminology is not fixed in the literature, so let us discuss them briefly. We say that a morphism $f\:X\to Y$ of $k$-analytic spaces is {\em compact} (resp. {\em quasi-compact}) if for any (quasi-)compact analytic domain $V\subseteq Y$ the inverse image $f^{-1}(V)$ is (quasi-)compact. In particular, a compact morphism is topologically separated. If $Y$ possesses an admissible covering by affinoid domains $Y_i$, then any quasi-compact domain $Y'\subseteq Y$ is a finite union of affinoid domains contained in one of the $Y_i$, hence $f$ is (quasi-)compact if and only if each $f^{-1}(Y_i)$ is (quasi-)compact. In other words, these notions are local on the base.

Now let us discuss purely topological notions. We say that $f$ is {\em topologically (quasi-)compact} if for each (quasi-)compact subset $S\subseteq Y$, the preimage $f^{-1}(S)$ is (quasi-)compact. Recall also that a continuous map is called {\em topologically proper} (or proper if no ambiguity is possible) if it is universally closed. Finally, we say that a topological space $X$ is locally compact if any point of $X$ has a compact neighborhood. In particular $X$ is locally Hausdorff, but not necessarily Hausdorff, and any $k$-analytic space is locally compact.

\begin{lemma}\label{qcmap}
Let $f\:X\to Y$ be a morphism of $k$-analytic spaces.

(i) The following conditions are equivalent: (a) $f$ is topologically proper, (b) $f$ is topologically quasi-compact, (c) $f$ is quasi-compact.

(ii) The following conditions are equivalent: (a) $f$ is topologically proper and separated, (b) $f$ is topologically compact, (c) $f$ is compact.
\end{lemma}
\begin{proof}
The second claim follows from the first one working locally on $Y$ and using that if $Y$ is Hausdorff, then $f$ is topologically separated if and only if $X$ is Hausdorff. To prove (i) we note that (a) implies (b) by \cite[Tag:005R]{stacks}, and (b) obviously implies (c). Assume, now that $f$ is quasi-compact. By \cite[Tag:005R]{stacks} to prove that $f$ is topologically proper it suffices to show that it is a closed map with quasi-compact fibers.

If $y\in Y$ is a point, then shrinking $Y$ we can make it compact without affecting the fiber. In this case $X$ is quasi-compact, hence a finite union of compact (even affinoid) domains $X_i$. Each compact map $X_i\to Y$ has a compact fiber over $y$, hence the union $f^{-1}(y)$ of these fibers is quasi-compact.

Let us prove that the image of a closed set $V\subseteq X$ is closed. It suffices to show that if $Y'\subseteq Y$ is a compact analytic domain, then $f(V)\cap Y'$ is closed in $Y'$. Being a quasi-compact domain, $X'=f^{-1}(Y')$ is a union of finitely many compact domains $Y'_i$. Each intersection $V_i=V\cap Y'_i$ is closed in $Y'_i$, hence compact. Thus, each $f(V_i)$ is compact and we obtain that their union $f(V)\cap Y'$ is a compact, hence closed, subset of $Y'$.
\end{proof}

\begin{remark}
In the definition of proper $k$-analytic morphisms one requires $f$ to be topologically compact. Perhaps, it is not immediate that this definition is local on the base when the base is not Hausdorff, but this is so due to Lemma~\ref{qcmap}. In fact, we will now establish a much stronger descent claim.
\end{remark}

Our next aim is to also ensure descent of topological properties, and as a main application we will achieve descent of properness and finiteness. This time we want to also rule out Example \ref{Gsurexam}, and naturally we will have to require that a base change morphism $f\:Y'\to Y$ is ``topologically surjective'' in an appropriate sense. The most straightforward choice is to require that $f$ is quasi-compact, but it is natural to consider, the topology generated by such maps. This leads to the following definition:

\begin{definition}\label{compsurdef}
A morphism of $k$-analytic spaces $f\:X \rightarrow Y$ is {\em properly surjective} if every quasi-compact domain $Y'\subseteq Y$ is the image of a quasi-compact domain $X'\subseteq X$.
\end{definition}

For example, any surjective quasi-compact map is properly surjective. In addition, if $Y$ possesses an admissible covering by affinoids $Y_i$ it suffices to check that each $Y_i$ is the image of a quasi-compact domain. In particular, it follows easily that if $f$ is surjective and topologically open then it is properly surjective. Properly surjective maps are closed under composition, and if $f\circ g$ is properly surjective then $f$ is properly surjective. It is easy to see that any admissible covering $\coprod_i X_i\to X$ of a $k$-analytic space $X$ by $k$-analytic domains is properly surjective.

Furthermore, if $f\:X\to Y$ is a properly surjective $k$-analytic morphism, then for any morphism $g\:Y'\to Y$ the base change $f\:X'\to Y'$ is properly surjective. Indeed, it suffices to check that for a sufficiently fine admissible affinoid covering $Y'=\cup_iY'_i$ each $Y'_i$ is the image of a quasi-compact domain. Thus, we can assume that $g(Y'_i)$ is contained in an affinoid domain $Y_i\subseteq Y$. By assumption, $Y_i$ is the image of a quasi-compact domain, hence $Y_i=\cup_{j=1}^nf(X_{ij})$ for affinoid domains $X_{i1}\..X_{in}$ in $X$, and we obtain that $Y'_i$ is the image of the quasi-compact domain $\cup_{j=1}^nX_{ij}\times_{Y_i}Y'_i$. By a similar argument, any ground field extension functor $Z \rightsquigarrow Z_K$ carries properly surjective $k$-analytic maps to properly surjective $K$-analytic maps.

Finally, despite an analogy between Examples \ref{surexam} and \ref{Gsurexam}, the relation between $G$-surjectivity and compact surjectivity is completely asymmetric: on the one hand, the cover in Example~\ref{Gsurexam} is $G$-surjective but not properly surjective, but on the other hand we have the following result.

\begin{lemma}\label{surlem}
Any properly surjective morphism $f\:X\to Y$ is $G$-surjective.
\end{lemma}
\begin{proof}
By Lemma \ref{Gsurlem} it suffices to prove that for any point $y\in Y$ the reduction $\tilY_y$ is covered by the images of the maps $\tilf_x\:\tilX_x\to\tilY_y$ with $x\in f^{-1}(y)$. Shrinking $Y$ around $y$ we can assume that $Y$ is compact. Then we can also replace $X$ by a quasi-compact subdomain that maps surjectively onto $Y$. For any point $x\in f^{-1}(y)$ the image of $\tilf_x$ is open, hence we can fix a subdomain $U=U(x)$ in $Y$ such that $y\in U$ and $\tilU_y$ is the image of $\tilf_x$. Set $V=V(x)=f^{-1}(U(x))$. Then $\tilV_x=\tilf_x^{-1}(\tilf_x(\tilX_x))=\tilX_x$, and we obtain that $V(x)$ is a neighborhood of $x$. Since $f^{-1}(y)$ is quasi-compact, there exist finitely many points $x_1\..x_n$ such that $V(x_i)$ cover $f^{-1}(y)$.

Using that $f$ is quasi-compact and hence closed, it follows that for a small enough neighborhood $Y'$ of $y$ one has that $f^{-1}(Y')\subseteq\cup_{i=1}^n V(x_i)$, and hence we can replace $Y$ and the spaces $X$, $V(x_i)$ by $Y'$ and its preimages. Thus, we are reduced to the case when $X=\cup_{i=1}^n V(x_i)$, and then the assertion is clear. Indeed, $Y=f(X)=\cup_i f(V(x_i))=\cup_i U(x_i)$, hence the reductions of $U(x_i)$ at $y$ cover $\tilY_y$, and it remains to recall that the latter are the images of $\tilf_{x_i}$.
\end{proof}

Now, we can work out descent with respect to properly surjective morphisms. We only list properties that do not satisfy $G$-surjective descent.

\begin{theorem}\label{compsurth}
The following properties of morphisms $g\:X\to Y$ are local with respect to properly surjective base change morphisms $Y'\to Y$: (ix) topologically separated, (x) topologically proper, (xi) separated, (xii) proper, (xiii) finite, (xiv) closed immersion.
\end{theorem}
\begin{proof}
We start with (x), which is the critical new ingredient. By Lemma~\ref{qcmap} we should prove that $g$ is quasi-compact whenever the base change $g'\:X'\to Y'$ is. For any quasi-compact domain $V\subseteq Y$ find a quasi-compact $V'\subseteq Y'$ which maps onto $V$. Then $g'^{-1}(V')$ is quasi-compact and maps surjectively onto $g^{-1}(V)$. Hence the latter is quasi-compact too.

(ix) Similarly to the proof of Theorem \ref{Gsurth}(vi), stability of the properly surjective property under $k$-analytic base change implies that the morphism of the diagonals $X' \times_{Y'} X'\rightarrow X \times_Y X$ is properly surjective.  Thus, by working with the diagonal maps and descent for topological properness we see that $g$ is topologically separated when $g'$ is topologically separated.

(xi)  A morphism is separated if and only if it is topologically and locally separated. Therefore this case follows from (ix) and Theorem \ref{Gsurth}(vi).

(xii) This case follows from (ix), (x) and Theorem \ref{Gsurth}(vii).

(xiii) A finite morphism is the same as a proper quasi-finite morphism. So, this case follows from (xi) and Theorem \ref{Gsurth}(viii).

(xiv) If $g'$ is a closed immersion then $g$ is at least finite. By Nakayama's Lemma, a finite map is a closed immersion if and only if its non-empty fibers are reduced 1-point sets corresponding to a trivial field extension. Thus, any finite $k$-analytic map that becomes a closed immersion after a surjective $k$-analytic base change is clearly a closed immersion.  Hence, $g$ is a closed immersion when $g'$ is.
\end{proof}

In view of Example \ref{compsurexam}, the list of properties in Theorem \ref{compsurth} cannot be increased.

\begin{remark}
To some extent, the properly surjective topology on the category of analytic spaces is analogous to the $h$-topology on the category of schemes, and the lists of properties local or non-local for these topologies are somewhat similar. Also, it is the analogue of Sholze's $v$-topology on perfectoid spaces.
\end{remark}

\section{Flat descent}\label{app}
Example \ref{compsurexam} shows that in order to descend more properties we have to impose a flatness assumption. Flatness in the analytic category is a rather subtle notion recently developed in full generality by Ducros in \cite{flat}, so we first recall some needed material.

\subsection*{Flatness in Berkovich geometry}
We refer the reader to \cite{flat}, but note for the sake of completeness that for quasi-finite morphisms this theory was developed by Berkovich already in \cite{berihes}, and there was an unpublished work by Berkovich where some results about flat morphisms between good spaces were obtained, including the theorem that boundaryless naively flat morphisms are flat (i.e., they are preserved under base changes).

Let $f\:X\to Y$ be a morphism between $k$-analytic spaces, $\calF$ be a coherent $\calO_X$-module, $x\in X$ be a point with $y=f(x)$. If $X$ and $Y$ are good then we say that $\calF$ (resp. $f$) is {\em naively $Y$-flat} at $x$ if $\calF_x$ (resp. $\calO_{X,x}$) is a flat $\calO_{Y,y}$-module. We say that $\calF$ or $f$ is {\em naively $Y$-flat} if it is so at all points of $X$. Unfortunately, this definition does not make too much sense in general since the naive flatness can be destroyed after a base change $Y'\to Y$ with a good source $Y'$, and, even worse, one can built such examples with $Y'\to Y$ being an embedding of an analytic subdomain, see \cite[Section~4.4]{flat}. For this reason, the only reasonable notion is that of universal naive flatness which we will simply call flatness: we say that an $f$ as above is flat if each its good base change is naively flat (and similarly for coherent sheaves). By \cite[4.1.7]{flat} (see also the online erratum) flatness is of $G$-local nature, namely $\calF$ is $Y$-flat at $x$ if and only if there exists a pair of good analytic domains $U\into X$ and $V\into Y$ with $x\in U$ and $f(U)\subset V$ and such that $\calF|_U$ is naively $V$-flat at $x$, and then $\calF|_{U'}$ is naively $V'$-flat at $x$ for any choice of such a pair $U',V'$. Due to the $G$-locality, flatness globalizes to all morphisms of $k$-analytic spaces: given a morphism $f\:X\to Y$, a coherent $\calO_{Y_G}$-module $\calF_G$ is $Y$-flat at a point $x\in X$ if there exist good domains $U\into X$ and $V\into Y$ such that $x\in U$, $f(U)\subset V$ and $(\calF_G)|_U$ is $V$-flat at $x$. An important difficult theorem \cite[Theorem~8.4.3]{flat} by Ducros asserts that for overconvergent (in particular, for boundaryless) morphisms between good spaces flatness is equivalent to the naive one, and hence can be checked on stalks. Here a morphism $X\to Y$ is called {\em overconvergent} at a point $x\in X$ if it factors into a composition of embedding of an analytic domain $X\into X'$ and a morphism $X'\to Y$ with $x\in\Int(X'/Y)$. 




\subsection*{Quasi-smoothness}
Berkovich introduced in \cite{berihes} notions of \'etale and smooth morphisms. An \'etale morphism is a flat unramified quasi-finite morphism. A smooth morphism $X\to Y$ is a morphism that locally factors as an \'etale morphism $X\to\A^n_Y$ followed by the projection. Probably the latter definition was chosen because it minimizes the use of flatness, which was not developed in \cite{berihes} beyond the quasi-finite case. However, we will see that this definition also has drawbacks. Note that smooth morphisms are automatically boundaryless. Furthermore, Berkovich introduced quasi-\'etale morphisms in \cite{bervanish} by $G$-localizing the notion of \'etaleness on the source.

In the first version of this paper we used flatness to define $G$-local analogues of \'etaleness and smoothness, that were called $G$-\'etaleness and $G$-smoothness. The first one is, in fact, equivalent to the quasi-\'etaleness while the second one was a new notion. Since then Ducros introduced in \cite[Chapter~5]{flat} the notion of quasi-smooth morphisms, which is equivalent to $G$-smoothness, and checked various properties, so we prefer to switch to the terminology of \cite{flat}. Ducros defines a quasi-smooth morphism as a morphism $X\to Y$ that $G$-locally factors as a quasi-\'etale morphism $X\to\A^n_Y$ followed by the projection, and shows that this is equivalent to requiring that $f$ is flat and has geometrically regular fibers.

\begin{remark}
A morphism is \'etale if and only if it is quasi-\'etale and boundaryless. The notion of smoothness is somewhat problematic. Probably, it is not $G$-local on the base in the non-good case. In this case, there should exist quasi-smooth and boundaryless morphisms which are not smooth (see also \cite[5.4.9]{flat}). For these reasons, we will not study descent of smoothness in this paper.
\end{remark}

\subsection*{Flat topologies}
We will work with two flat topologies.

\begin{definition}\label{flatdef}
(1) As in the theory of schemes, a {\em faithfully flat} morphism means a surjective flat morphism. Such morphisms are covers of the {\em flat topology} on the category of analytic spaces.

(2) A $G$-surjective flat morphisms will be called {\em $G$-faithfully flat} or a {\em Tate-flat covering}. The corresponding topology will be called the {\em Tate-flat topology}.
\end{definition}

\begin{remark}\label{Tateflatrem}
(1) The situation is somewhat analogous to the usual theory of schemes, where one has flat topology and fpqc topology, and many descent results only hold in the fpqc topology. This similarity is strengthened by Theorem~\ref{Tateflat} below. Nevertheless, the analogy is not so close because in algebraic geometry most properties of morphisms are local for arbitrary faithfully flat morphisms.

(2) A typical example of a flat but not Tate-flat cover is a non-admissible cover by subdomains.
\end{remark}

The following result is highly non-trivial because its proof uses the theorem of Ducros on the images of flat morphisms.

\begin{theorem}\label{Tateflat}
Let $f\:X\to Y$ be a surjective flat morphism of analytic spaces. Then $f$ is a Tate-flat covering if and only if it is properly surjective. In particular, if $f$ is boundaryless then it is a Tate-flat covering.
\end{theorem}
\begin{proof}
The inverse implication is covered by Lemma~\ref{surlem}, so we should only prove that if $f$ is $G$-surjective then each compact subdomain $Y'$ in $Y$ is the image of a quasi-compact subdomain of $X$. Replacing $Y$ and $X$ by $Y'$ and $f^{-1}(Y')$, we can assume that $Y=Y'$ is compact. It suffices to show that for a fixed point $y\in Y$ there exists a quasi-compact domain $X'\subseteq X$ whose image is a neighborhood of $y$. By Lemma~\ref{Gsurlem} there exist $x_1\..x_n\in f^{-1}(y)$ such that $\tilY_y$ is covered by the images of $\tilX_{x_i}$. For each $i$ choose a compact domain $X_i$, which is a neighborhood of $x_i$, and set $X'=\cup_i X_i$. By \cite[Theorem~9.2.1]{flat}, $Y'=f(X')$ is an analytic domain in $Y$. Since $Y'_y$ contains the images of $\tilX_{x_i}$, the inclusion $\tilY'_y\subseteq\tilY_y$ is an equality, and hence $Y'$ is a neighborhood of $y$ by \cite[Theorem~4.5]{temkin2}.
\end{proof}

We will not use the following remark, so we skip its simple justification.

\begin{remark}\label{tateflatrem}
(1) It is easy to see that a flat covering $f\:X\to Y$ is properly surjective if and only if $f$ locally admits a compact flat quasi-section, that is, there exists an open covering $Y=\cup_i Y_i$ and subdomains $X_i\into X$ such that $f$ restricts to surjective flat compact morphisms $X_i\to Y_i$.

(2) The definition of Tate-flat topology in \cite[Definition~4.2.1]{relamp} uses property \ref{Tateflat}(3) with Berkovich topology replaced by Tate topology. It is easy to see that both choices are equivalent.

(3) In addition, it follows easily that the Tate-flat topology is equivalent to the topology generated by Berkovich topology and the topology of flat topologically proper morphisms.
\end{remark}

\subsection*{Faithfully flat descent}
Now, let us discuss descent with respect to flat base change morphisms. We only list properties not satisfied by the properly surjective descent.

\begin{theorem}\label{flatth}
The following properties of morphisms $g\:X\to Y$ are local with respect to faithfully flat morphisms $Y'\to Y$: (xv) flat, (xvi) quasi-smooth, (xvii) quasi-\'etale.
\end{theorem}
\begin{proof}
All three properties are $G$-local on $X$ and $Y$, hence we can assume that $X$, $Y$ and $Y'$ are good (or even affinoid). For a point $x\in X$ choose a preimage $x'\in X'=X\times_YY'$ and let $y'\in Y'$ and $y\in Y$ be the images of $x'$. If the base change $g'\:X'\to Y'$ is flat then the composition $X'\to Y'\to Y$ is flat, and in particular $\calO_{X',x'}$ is flat over $\calO_{Y,y}$. Since $\calO_{X',x'}$ is flat over $\calO_{X,x}$ because flatness is preserved under base changes, we obtain that $\calO_{X,x}$ is flat over $\calO_{Y,y}$. This proves that $g$ is naively flat at $x$. To prove that $g$ is actually flat we must prove the same for any good base change of $g$, but that can be done precisely in the same way. Descent of quasi-smoothness follows since a morphism is quasi-smooth if and only if it is flat and has geometrically regular fibers, see \cite[5.3.4(1)]{flat}. By Theorem \ref{surjth}(ii) this also implies descent of quasi-\'etaleness, since a morphism is quasi-\'etale if and only if it is quasi-smooth and has zero-dimensional fibers.
\end{proof}

In view of Example \ref{surexam}, the list of properties in Theorem \ref{flatth} cannot be increased.

\subsection*{Tate-flat descent}
Finally, we prove that all twenty properties from \S\ref{intromainsec} satisfy Tate-flat descent. We only list properties that do not satisfy a stronger descent -- the properly surjective or the faithfully flat one.

\begin{theorem}\label{Tateflatth}
The following properties of morphisms $g\:X\to Y$ are local with respect to faithfully flat morphisms $Y'\to Y$: (xviii) \'etale, (xix) open immersion, (xx) isomorphism.
\end{theorem}
\begin{proof}
Recall that $g$ is properly surjective by Theorem \ref{Tateflat}. In particular, descent holds for all properties of Theorems \ref{Gsurth} and \ref{compsurth}, including finiteness and quasi-finiteness. Since descent of flatness holds by Theorem~\ref{flatth}, it remains to make the following observations: isomorphism is a finite flat morphism whose fibers are isomorphisms, \'etaleness is the same as quasi-\'etaleness and quasi-finiteness, and open immersions are the same as \'etale monomorphisms.
\end{proof}

It is more difficult to descend properties with respect to a base field extension and to descend absolute properties (e.g., to determine if, for a surjective $k$-analytic morphism $Y'\to Y$ without boundary, $Y$ is good or strictly $k$-analytic if and only if $Y'$ is). To prove these descent statements we need certain facts about graded birational spaces and graded reductions that are not covered by \cite{temkin2} and will be proved in \S\ref{birgsec}-\ref{descsec}, building on some additional graded commutative algebra that we develop in \S\ref{gradsec}. To motivate some of these later considerations, we make a digression in the next section.
}

\section{Ungraded birational descent}\label{ungradsec}

\subsection*{Descent of affineness}
By \cite[5.1]{temkin2}, a $k$-analytic germ $(X,x)$ is good if and only if the corresponding
($\R^{\times}_{>0}$-graded) birational space $\widetilde{X}_x$ in ${\rm{bir}}_{\widetilde{k}}$ is {\em affine}, i.e.,
corresponds to an open subset $\P_{\widetilde{\calH(x)}/\widetilde{k}}\{A\} \subseteq
\P_{\widetilde{\calH(x)}/\widetilde{k}}$ for a finitely generated $\R^{\times}_{>0}$-graded
$\widetilde{k}$-subalgebra $A \subseteq \widetilde{\calH(x)}$. Thus, descent of goodness is related to descent
of affineness of (graded) birational spaces over a graded field. In this section we describe an elementary
approach to this descent problem for birational spaces in the special case of the trivial grading group, so all
fields, rings, birational spaces, etc., in this section are assumed to be {\em ungraded} (i.e., we take the
grading group $G$ to be trivial). Recall from \cite[\S1]{temkin1} that any graded birational space $\varcalX =
(X\to\bfP_{K/k})$ in the category $\bir_k$ for the trivial grading group naturally ``arises'' from a $k$-map
$\eta\:\Spec(K)\to\calX$, where $\calX$ is an integral scheme of finite type over $k$ and $\eta$ is generic over
$\calX$. (Explicitly, if $\calX$ is separated then $X \subseteq \P_{K/k}$ is the open subset of valuations rings
of $K$ containing $k$ and dominating the local ring of a point on $\calX$.  The general case proceeds by gluing
over separated opens in $\calX$.) We call such $\eta\:\Spec(K) \rightarrow \calX$ a (pointed) {\em integral
scheme model} of $\varcalX$. For any two pointed integral scheme models $\eta_j\:\Spec(K) \rightarrow \calX_j$ of
$\varcalX = (X \rightarrow \P_{K/k})$ there exists a third such $\Spec(K)\to\calX$ with $\calX$ proper over
each $\calX_j$. (Actually, the collection of all such pointed $\calX$ with a fixed $K/k$ is an inverse system
and $X$ is naturally homeomorphic to the inverse limit of all such $\calX$.)

Given $\varcalX$ and $\calX$ as above and a pair of field extensions $K/k$ and $L/l$ equipped
with a map $K
\rightarrow L$ over a map $k \rightarrow l$, there is a naturally induced map
$\psi_{L/l,K/k}\:\bfP_{L/l}\to\bfP_{K/k}$ and one easily checks that a pointed integral $l$-scheme model
$\Spec(L) \rightarrow \calY$ of the birational space $\varcalY=\varcalX \times_{\bfP_{K/k}}\bfP_{L/l}$ in the
category $\bir_l$ is given by taking $\calY$ to be the Zariski closure (with reduced structure) of the image of
the natural composite map $$\eta\:\Spec(L) \rightarrow \Spec(K \otimes_k l) \rightarrow \calX_l=\calX \otimes_k l.$$ (Note
that $\calX_l$ does not have to be either irreducible or reduced, and $\eta$ does not have to
hit a generic point of $\calX_l$.)
If $\varcalX$ is affine then $\varcalY$ is affine, since $X = \bfP_{K/k}\{f_1,\dots, f_n\}$
implies that $\varcalY = \bfP_{L/l}\{f_1,\dots, f_n\}$. It turns out that the converse is true under the
additional assumption that
any algebraically independent set over $k$ in $l$ is algebraically independent over
$K$. (We will prove in Theorem \ref{psilem}(1) that this condition on
$L/l$ and $K/k$ is equivalent to the
surjectivity of $\psi_{L/l,K/k}$.)

Theorem \ref{birdescent} generalizes this converse statement for an arbitrary grading group $G$, but in the ungraded case it can be proved much more easily: it is a consequence of the following theorem that gives a criterion for affineness of birational spaces in the ungraded case in terms of pointed integral scheme models. To explain this implication, we first note that by Theorem~\ref{affcriterion} below affineness descends in three special cases: (1) $l=k(T)$ and $L=K(T)$ are purely transcendental with a transcendence basis $T= \{T_i\}_{i \in I}$, (2) $l=k$, (3) $L=K$ and $l/k$ is algebraic. Indeed, (2) is obvious because one can take $\calY=\calX_l=\calX$, and (3) is obvious because $L=K$ and $\psi_{L/l,K/k}$ is bijective. In (1), choosing a pointed integral scheme model $\calX$ we have that $\Spec(L)\to\calY=\calX\otimes_kk(T)$ is a pointed integral model, and it is easy to see that if $\calX\otimes_kk(T)$ is proper over an affine $k(T)$-scheme then $\calX$ is proper over an affine $k$-scheme. It remains to note that choosing a transcendence basis $T$ of $l$ over $k$ we obtain a factorization $$\psi_{L/l,K/k}=\psi_{K(T)/k(T),K/k}\circ\psi_{L/k(T),K(T)/k(T)}\circ\psi_{L/l,L/k(T)}$$ with factors of types (1), (2) and (3).

Finally, we note that Theorem \ref{affcriterion} should generalize to graded birational spaces if one uses graded integral schemes in the role of $\calX$, but a theory of graded schemes has not been developed.

\begin{theorem}\label{affcriterion}
Let $\varcalX = (X \rightarrow \P_{K/k})$ be a birational space over a field $k$,
and let $\Spec(K) \rightarrow \calX$ be a pointed integral scheme model for $\varcalX$. The
birational space $\varcalX$ is affine if
and only if the normalization of $\calX$ is proper over an affine $k$-scheme of finite type.
\end{theorem}

\begin{proof}
The ``if'' direction easily follows from the definition of the functor
$(\Spec(K) \rightarrow \calX) \rightsquigarrow (X \rightarrow \P_{K/k})$
in
\cite[\S1]{temkin1}. Indeed, the dominant
point $\Spec(K)\to\calX$ lifts to the normalization $\widetilde{\calX}$ of
$\calX$, and since $\widetilde{\calX}$ is proper over $\calX$
the pointed integral schemes $\Spec(K)\to\widetilde{\calX}$ and $\Spec(K)\to\calX$
correspond to isomorphic objects in $\bir_k$. We are given that $\widetilde{\calX}$ is proper over an affine
$k$-scheme $\calX'$ of finite type,
so $\Spec(K)\to\widetilde{\calX}$ and the induced morphism $\Spec(K)\to\calX'$ have isomorphic images in
$\bir_k$. Thus, if $A \subseteq K$ is the finitely generated coordinate ring of $\calX'$
over $k$ then $X = \P_{K/k}\{A\} \subseteq \P_{K/k}$
is affine in ${\rm{bir}}_k$. The converse implication follows from the next lemma and the
fact that any two pointed integral scheme models for an object
in ${\rm{bir}}_k$ admit a common refinement that is proper over both of them.
\end{proof}

The non-trivial direction in the following result is \cite[Theorem 2.1]{semiaffine}, but we provide our proof for completeness. Our extensive study of graded integral closure in \S\ref{gradsec} is inspired by the role of normalizations in the lemma.  The results in \S\ref{gradsec}
are the main ingredients in the proof of the key descent theorem for graded birational spaces given in Theorem \ref{birdescent}.

\begin{lemma}\label{propersurj}
Let $S$ and $S'$ be irreducible and reduced schemes of finite type over a field $k$, and let $\pi\:S' \rightarrow
S$ be a proper surjection.  The normalization $\widetilde{S'}$ of $S'$ is proper over an affine $k$-scheme of
finite type if and only if the normalization $\widetilde{S}$ of $S$ satisfies the same property.
\end{lemma}
\begin{proof}
We can and do replace $S$ and $S'$ with their normalizations. The condition that an integral $k$-scheme $X$ be proper over
an affine algebraic $k$-scheme is equivalent to the simultaneous conditions: (1) the domain $A =
\mathscr{O}_X(X)$ is of finite type over $k$, (2) the canonical $k$-map $X \rightarrow \Spec(A)$ is proper. The nontrivial direction is descent from $S'$ to $S$, so assume that $S'$ is proper over an affine
algebraic $k$-scheme. In particular, $S'$ is separated.  Since $\pi$ is a proper surjection it follows that
$\Delta_S$ has closed image, so $S$ is separated over $k$. Let $A = \mathscr{O}_S(S)$ and $A' =
\mathscr{O}_{S'}(S')$, so $A'$ is a
$k$-algebra of finite type and $A$ is a $k$-subalgebra of $A'$.  If we
can show that $A'$ is a finite $A$-module then $A$ must be of finite type over $k$ \cite[7.8]{am}, so in the
commutative square
$$\xymatrix{{S'} \ar[r] \ar[d] & {S} \ar[d] \\ {\Spec(A')} \ar[r] & {\Spec(A)}}$$
the bottom and left sides are proper, the top is a proper surjection, and the right side is separated and finite
type.  Hence, the right side is proper, as required.

To prove that $A'$ is $A$-finite, observe that the $A$-algebra $A'$ is identified with the global sections of
the coherent sheaf of algebras $\mathscr{A}' = \pi_{\ast}(\mathscr{O}_{S'})$ on $S$.  If $\eta \in S$ is the
generic point then for any section $h$ of $\mathscr{A}'$ over a non-empty open $U$ in $S$, the characteristic
polynomial of $h_{\eta} \in \mathscr{A}'_{\eta}$ over $\mathscr{O}_{S,\eta} = k(S)$ has coefficients in
$\mathscr{O}_S(U)$ since the coherent $\mathscr{O}_S$-algebra $\mathscr{A}'$ is torsion-free and we can work
over an affine open cover of $U$ in the normal scheme $S$. Thus, $A'$ is integral over $A$.  But $A'$ is
finitely generated as an $A$-algebra (it is even of finite type over $k$), so $A'$ is finite as an $A$-module.
\end{proof}

\subsection*{Some examples}
Though the proof of descent of affineness for birational spaces with trivial grading
as explained above is much shorter than the proof required in the general graded
case in Theorem \ref{birdescent}, even in the ungraded case the argument
has hidden dangers.   One subtlety is that Lemma \ref{propersurj} is false without normalizations, as was shown in \cite[Example~2.4]{semiaffine}. Here is another counterexample along these lines that was suggested to us by de Jong.

\begin{example}\label{deJongexample}
Assume $k$ has characteristic 0 (!), and let $S$ be the
integral $k$-scheme obtained from $S' = \mathbf{P}^1 \times
\mathbf{A}^1$ by identifying the lines $\{0\} \times \mathbf{A}^1$ and $\{1\} \times \mathbf{A}^1$ via $t
\mapsto t+1$.  In other words, replace $\mathbf{A}^1 \times \mathbf{A}^1 = \Spec(k[x,t]) \subseteq S'$ by
$\Spec(A)$, where $A \subseteq k[x,t]$ is the $k$-subalgebra of $f \in k[x,t]$ such that $f(0,t+1) = f(1,t)$.
(An easy argument in the category of locally ringed spaces shows that $S$ has the expected
universal mapping property in the category of $k$-schemes.)
Since $x^2 - x \in A$ and $t + x \in A$, the extension $A \rightarrow k[x,t]$ is integral and hence
finite, so \cite[7.8]{am} ensures that $A$ is finitely generated over $k$.
Hence, $S$ is of finite type over $k$. Obviously $A_{x^2-x} =
k[x,t]_{x^2-x}$ (since $(x^2-x)k[x,t] \subseteq A$), so $S'$ is the normalization of $S$
(in particular, $S$ is separated) and $S'$ is proper over
the affine $k$-line, with $\mathscr{O}_{S'}(S') = k[t]$.
Thus, the global functions on $S$ are those
$h \in k[t]$ such that $h(t) = h(t+1)$, so since ${\rm{char}}(k) = 0$ we get $\mathscr{O}_S(S) = k$. Since $S$
is not $k$-proper it therefore
 cannot be proper over an algebraic affine $k$-scheme, though its normalization $S'$
admits such a description.
\end{example}

We will later apply descent of affineness for birational spaces (in the graded case, Theorem \ref{birdescent}) to prove that the property of being a good analytic space descends through morphisms with surjective interior and extension of the ground field. As we mentioned in the Introduction, one should be especially careful when dealing with descent of goodness because of an example (due to Q. Liu) of a 2-dimensional separated non-affinoid rigid space $X$ (over any $k$ with $|k^{\times}| \ne 1$) such that $X$ has an affinoid normalization.  This phenomenon has no analogue for schemes of finite type over fields. While reading Liu's paper \cite{liu}, we discovered a much simpler example of the same nature which is a very close relative of de Jong's example above.

Naturally, one constructs $X$ by {\em pinching} its affinoid normalization $X'$. Let us recall first, how pinching works for schemes. Assume that $\vcalX'=\Spec(A')$ is an affine scheme with a closed subscheme $\vcalY'=\Spec(B')$ and $h\:\vcalY'\to\vcalY=\Spec(B)$ is a finite morphism, then there exists a pushout $\vcalX=\vcalX'\coprod_{\vcalY'}\vcalY$, called {\em pinching} of $\calX'$ along $h$, and it satisfies the following list of nice properties:

(1) $\vcalX$ is affine and hence coincides with the pushout in the category of affine schemes, that is, $\vcalX=\Spec(A)$ for $A=A'\times_{B'}B$,

(2) The morphism $f\:\vcalX'\to\vcalX$ restricts to $h$ over $\vcalY$ and induces an isomorphism $\vcalX'-\vcalY'\simeq\vcalX-\vcalY$. In particular, the following squares are both Cartesian and co-Cartesian
$$\xymatrix{
B' & A'\ar@{->>}[l] & & & \vcalY'\ar[d]^h\ar@{^(->}[r] & \vcalX'\ar[d]^f\\
B\ar[u] & A\ar@{->>}[l]\ar[u] & & & \vcalY\ar@{^{(}->}[r] & \vcalX
}
$$
and if $\vcalX'$ is normal and $\vcalY'$ is nowhere dense, then $f$ is the normalization morphism.

(3) In addition, the construction is universal: any base change with respect to $\vcalZ\to\vcalX$ yields another pinching diagram.

Using (3) one can deduce a similar theory for the case when the schemes are not affine, but it is a bit pathological and a pinching of a projective variety can be not a scheme but only an algebraic space.

\begin{example}\label{alaLiuexample}
Let $k$ be a non-archimedean field with a non-trivial valuation and residue characteristic zero.
We will work with reductions of $k$-Banach algebras in the traditional
(rather than graded) sense, so now $\widetilde{k}$ denotes
the ordinary residue field of $k$ (rather than an $\R^{\times}_{>0}$-graded
field as in \S\ref{birsec1} and \cite{temkin2}).

Surprisingly it turns out that pinching in the affinoid category is not always possible, and the general idea
of Q. Liu is to construct a Cartesian diagram of rings with strictly $k$-affinoid $\calA',\calB',\calB$, but not affinoid $k$-Banach algebra $\calA=\calA'\times_{\calB'}\calB$. Nevertheless, in Liu's example one constructs the pushout $X$ of $X'=\calM(\calA')$ and $Y=\calM(\calB)$ along $Y'=\calM(\calB')$ by gluing affinoid pushouts of subdomains, similarly to non-affine pinchings of schemes. Then $X$ is a non-affinoid space (already $\calA=\Gamma(\calO_X)$ is not affinoid), but its normalization $X'$ is affinoid.
$$\xymatrix{
\calB' & \calA'\ar@{->>}[l] & & & Y'\ar[d]^h\ar@{^(->}[r] & X'\ar[d]^f\\
\calB\ar[u]_\phi & \calA\ar@{->>}[l]\ar[u]_\psi & & & Y\ar@{^{(}->}[r] & X
}
$$
In addition, in Liu's and our examples the finite homomorphisms $\phi,\psi$ are isometric embeddings, but the reduction homomorphism $\tilpsi\:\tilcalA\to\tilcalA'$ is not finite. Then $\calA$ is not affinoid, as otherwise the latter homomorphism has to be finite by \cite[6.3.5/1]{bgr}.
The desired property of $\tilpsi$ is achieved by a careful choice of $\phi$ with a non-surjective $\tilphi$; although the reduction functor preserves finiteness, it does not preserve surjectivity, see \cite[Example~6.3.1/2]{bgr}. In particular, for the associated affine formal models the pushout datum is already not a pinching datum.

To make our example, let $\calA'=k\{x,y\}$ and $\calB'=\calA'/(x^2-\pi x)$, where $\pi\in k$ is a non-zero element
with $r=|\pi|<1$. Thus, $\calB'\simeq k\{y_0\} \times k\{y_\pi\}$, where $(y_0,y_\pi)$ is the image of $y$. In
particular, $X'$ is a 2-dimensional closed unit polydisc and $Y'$ is a
disjoint union of two one-dimensional closed unit subdiscs $D_0 = \{x = 0\}$ and
$D_\pi = \{x = \pi\}$.  Also we take $\calB=k\{y\}$ and define $\phi$ by $\phi(y)=(y_0,y_\pi+1)$. In other words, we want to pinch $X$ along the morphism $h$, which identifies $D_0$ and $D_\pi$ via $y_0=y_{\pi}+1$.

First, let us define the pinching $X$. For $s\in(0,1]$ set $\calA'_s=\calA'\{s^{-2}(x^2-\pi x)\}$ and $X'_s=\calM(\calA'_s)$. Now, fix $s<r$ and notice that in this case $X'_s=X'_0\coprod X'_\pi$ consists of two disjoint polydiscs of radii $(s,1)$ with a choice of coordinates $(x_0=x,y_0)$ and $(x_\pi=x-\pi,y_\pi+1)$ convenient for our purposes. Identifying the closed subspaces $D_0\into X'_0$ and $D_\pi\into X'_\pi$ one obtains a pinching $X_s=\calM(\calA_s)$ of $X'_s$ along $h$, where $\calA_s=\calM(k\{s^{-1}x_0,s^{-1}x_\pi,y\}/(x_0x_\pi))$. Clearly, $X'_s$ is the normalization of $X_s$. Since $X'_s\to X_s$ is an isomorphism on the complements to $Y'$ and $Y$, we can past to both complements the space $X'\setminus Y'$ obtaining a morphism $f\:X'\to X$, which restricts to $h$ and induces an isomorphism $X'\setminus Y'\simeq X\setminus Y$. Thus, $f$ is a required pinching and $X'$ is the normalization of $X$. We will now prove that $X$ is not affinoid, and other properties will be outlined later.

Assume to the contrary that $X=\calM(\calC)$. Then the morphism of reductions $\tilh\:\tilX'=\Spec(\tilcalA')\to\tilX=\Spec(\tilcalC)$ is finite by \cite[6.3.5]{bgr}. Let us compare the morphisms of $\tilY'=\tilD_0\coprod\tilD_\pi$ to $\tilX'$ and $\tilY$. The first one maps $\tilD_0$ and $\tilD_\pi$ isomorphically onto the closed subscheme $Z=\Spec(\tilk[\tily])\into\Spec(\tilk[\tilx,\tily])$ by sending $\tily$ to $\tily_0$ and $\tily_\pi$. The other one is the disjoint union of isomorphisms but this time sending $\tily$ to $\tily_0$ and $\tily_\pi+1$. Since the compositions $\tilY'\to\tilX'\to\tilX$ and $\tilY'\to\tilY\to\tilX$ coincide, we obtain that the points of $Z$ with $\tily\in\Z$ are mapped to the same point in $\tilX$. Since $\mathrm{char}(\tilk)=0$ this implies that $\tilh$ has an infinite fiber and hence cannot be finite. This completes our construction of a non-affinoid $X$ with an affinoid normalization.

For the sake of comparison, we list a few other properties satisfied by this data, but skip detailed arguments for shortness. First, $X$ is indeed the pushout $X'\coprod_{Y'}Y$. To prove this it suffices to show that $X'_s\coprod_{Y'}Y=X_s$. By a direct inspection $\calA_s=\calA'_s\times_{\calB'}\calB$, hence the latter is true in the affinoid category. Moreover, this pushout is compatible with affinoid base changes $T\to X$ and it follows easily that $X'_s\coprod_{Y'}Y=X_s$ also in the category of all $k$-analytic spaces.

In fact, the same arguments shows that even $X_r$ is still affinoid with normalization $X'_r$ -- the unit polydisc of radii $(r,1)$. On the other hand, we proved earlier that the pushout $X=X_1$ is not affinoid and the same argument shows that this is so for any $s>r$. The conceptual explanation is very simple: the affine formal model $\Spf(\calB'^\circ)\to\Spf((\calA'_s)^\circ)$ of the closed immersion $Y'\to X_s$ is a closed immersion if and only if $s\le r$.

Finally, let $\calA = \calA' \times_{\calB'} \calB$ be the preimage of $\calB$ in $\calA'$ provided with the $k$-Banach norm induced from the
Gauss norm on $\calA'$. Working on the level of abstract rings, it is immediate that the inclusion homomorphism $\psi\:\calA\to\calA'$ is finite and $\calA'$ is the normalization of $\calA$. However, comparing the maps $\tilY'\to\tilX'\to\Spec(\tilcalA)$ and $\tilY'\to\tilY\to\Spec(\tilcalA)$ in the same way as above, one obtains that the morphism $\tilX'\to\Spec(\tilcalA)$ is not finite. Thus, $\tilpsi$ is not finite, and hence $\calA$ is not $k$-affinoid by \cite[6.3.5/1]{bgr}.
\end{example}

Example \ref{alaLiuexample} is very close to de Jong's example. In both cases a global
pushout loses a good property (being proper over an affine or being affinoid), yet it can be
constructed by restricting to a subspace ($\A^1\times\A^1\into\P^1\times\A^1$ or
$X'_r=\calM(k\{\frac{x}{\pi},y\})\into X=\calM(k\{x,y\})$) where it behaves nicely and is described by simple
explicit formulas that are the same in both examples.

\section{Graded commutative algebra}\label{gradsec}
Throughout \S\ref{gradsec}--\S\ref{descsec}, $G$ is an arbitrary commutative multiplicative group (that will be $\bfR_{>0}^\times$ in the applications) and $k$ is a $G$-graded field. We consider only $G$-gradings in the sequel, so $G$ will often be omitted from the notation. The $G$-grading on any graded ring $A = \oplus_{g \in G} A_g$ will be denoted $\rho\:\coprod (A_g - \{0\})  \to G$. For any nonzero $A$, by $A^\times$ we denote the group of invertible homogeneous elements, so there is a homomorphism $\rho\:A^{\times} \rightarrow G$ whose image in case $A = K$ is a graded field is the subgroup of $G$ that consists of all $g \in G$ such that $K_g \ne 0$.  We are going to prove some results about extensions of graded fields. We will see that the theory of graded fields is similar to the classical ungraded case, and many proofs are just mild variants of the classical proofs. Some results on graded fields were proved in \cite[\S5.3]{ct}, and the notions of a finite extension and its degree were introduced there; recall from \cite[1.2]{temkin2} that any graded module over a graded field $K$ is necessarily a free module (with a homogeneous basis), and if $K \rightarrow L$ is a map of graded fields then its {\em degree} is defined to be the $K$-rank of $L$.

A {\em graded domain} is a non-zero graded ring $A$ such that all nonzero homogeneous elements of $A$ are not zero-divisors in $A$, and a key example of a graded domain is the ring $K[g_0^{-1}T]$ for a graded field $K$ and any $g_0 \in G$; this is the ring $K[T]$ in which $K$ is endowed with its given grading and $T$ is declared to be homogeneous with grading $g_0$.   For example, if $c \in K^{\times}$ with $\rho(c) = g \in G$ then evaluation at $c$ defines a graded homomorphism of graded rings $K[g^{-1}T] \rightarrow K$.

A graded ring $A$ is {\em graded noetherian} if every homogeneous ideal in $A$ is finitely generated, and by the classical argument it is equivalent to say that the homogeneous ideals of $A$ satisfy the ascending chain condition, in which case every graded $A$-submodule of a finitely generated graded $A$-module is finitely generated.   The proof of the Hilbert basis theorem carries over, so $A[g^{-1}X]$ is graded noetherian for any $g \in G$ when $A$ is graded noetherian.  In particular, if $k$ is a graded field then any finitely generated graded $k$-algebra is graded noetherian.

\subsection*{Some graded field theory}
\begin{lemma}\label{charpol}
Let $K$ be a graded field.  For any $g \in G$, every nonzero homogeneous ideal $I \subseteq K[g^{-1}T]$ is principal with a unique monic homogeneous generator.   Moreover, $K[g^{-1}T]^{\times} = K^{\times}$ and $K[g^{-1}T]$ is graded-factorial in the sense that every monic homogeneous element in $K[g^{-1}T]$ with positive degree is uniquely $($up to rearrangement$)$ a product of monic irreducible homogeneous elements in $K[g^{-1}T]$.
\end{lemma}

\begin{proof}
If $f$ and $g$ are nonzero homogeneous elements of $A$ having respective leading terms $aT^m$ and $bT^n$ with  $m\le n$ then $a, b \in K^{\times}$ and $g-\frac ab T^{n-m}f$ is homogeneous and either vanishes or is of degree smaller than $n$, as in the usual Euclidean algorithm. In particular, it follows that $I$ is generated by a monic homogeneous polynomial, and such a generator is obviously unique.  It follows that the maximal graded ideals of $K[g^{-1}T]$ are precisely the ideals $(f)$ for a monic irreducible homogeneous $f \in K[g^{-1}T]$. Hence, $K[g^{-1}T]/(f)$ is a graded field for such $f$. In particular, if $h \in K[g^{-1}T]$ is homogeneous and a given monic homogeneous $f$ does not divide $h$ then $\alpha f + \beta h = 1$ for some homogeneous $\alpha, \beta \in K[g^{-1}T]$. We can therefore show that $K[g^{-1}T]$ is graded-factorial by copying the classical argument (working just with nonzero homogeneous elements).
\end{proof}

An important instance of this lemma occurs for an extension $L/K$ of graded fields:  if $x \in L^{\times}$ with $\rho(x) = g \in G$ then the graded $K$-algebra evaluation map $K[g^{-1}T] \rightarrow L$ at $x$ has homogeneous kernel denoted $I_x$, and $K[g^{-1}T]/I_x \rightarrow L$ is an isomorphism onto the graded domain $K[x] \subseteq L$, so $I_x$ is a graded-prime ideal of $K[g^{-1}T]$.  If $I_x = 0$ then we say $x$ is {\em transcendental} over $K$, and otherwise $I_x$ is a graded-maximal ideal with $K[x]$ therefore a graded field of finite degree over $K$. The case $I_x \ne 0$ always occurs when $L$ has finite degree over $K$ (due to $K$-freeness of graded $K$-modules), and whenever $I_x \ne 0$ the unique monic generator of $I_x$ is denoted $f_x(T)$ and is called the {\em minimal homogeneous polynomial} of $x$ over $K$.

\begin{corollary}\label{splitcor}
For any graded field $K$ and nonzero homogeneous polynomial $f(T)\in K[g^{-1}T]$, there exists a finite extension $L/K$ such that $f$ splits completely into a product of homogeneous degree-$1$ polynomials in $L[g^{-1}T]$.
\end{corollary}

\begin{proof}
We may assume $f$ is monic and we induct on $\deg_T(f)$.  We may assume $\deg_T(f) > 0$.  Factoring $f(T)$ as a product of irreducible monic homogeneous polynomials in $K[g^{-1}T]$, we may assume $f$ is a monic irreducible and $f(0) \in K^{\times}$. Then $K' = K[g^{-1}T]/(f)$ is a graded field of finite degree over $K$ such that $f(x) = 0$ for a homogeneous $x \in {K'}^{\times}$ with $\rho(x) = g$.  Hence, $T - x \in K'[g^{-1}T]$ is homogeneous and by the Euclidean algorithm in such graded polynomial rings we have $f = (T-x)h(T)$ in $K'[g^{-1}T]$ for some monic homogeneous $h \in K'[g^{-1}T]$. Since $\deg_T(h) < \deg_T(f)$, we are done.
\end{proof}

Next, we discuss transcendental extensions of graded fields. For any extension of graded fields $K/k$, a subset $S \subseteq K^{\times}$ (with $\rho(s) = g_s \in G$ for each $s \in S$) is said to be {\em algebraically independent over $k$} if the graded $k$-algebra map $k[g_s^{-1}T_s]_{s \in S} \rightarrow K$ defined by $T_s \mapsto s$ is injective.  The following two conditions on a subset $S \subseteq K^{\times}$ are equivalent: (1) $S$ is a maximal algebraically independent set over $k$, (2) $S$ is minimal for the property that $K$ is integral over the graded $k$-subfield generated by $S$ (i.e., $K$ is integral over the graded fraction field of the graded $k$-subalgebra generated by $S$). In condition (2) it clearly suffices to check integrality for elements of $K^{\times}$.  A subset $S \subseteq K^{\times}$ satisfying (1) and (2) is called a {\em transcendence basis} for $K/k$ (and such subsets clearly exist, via condition (1)). Analogously to the classical arguments, one proves that all transcendence bases have equal cardinality, which is called the {\em transcendence degree} of $K$ over $k$ and is denoted $\trdeg_k(K)$. Also, one shows akin to the classical case that transcendence degree is additive in towers of graded fields. An extension $K/k$ of graded fields has transcendence degree 0 if and only if each $x \in K$ is integral over $k$, in which case we say that $K/k$ is {\em algebraic}.

As with degree for finite extensions of graded fields (studied in \cite[5.3.1]{ct}), the value of the transcendence degree ``splits" into a contribution from the extension of $1$-graded parts $K_1/k_1$ (ordinary fields) and a contribution from the extension of grading groups $\rho(k^{\times}) \subseteq \rho(K^{\times})$ in $G$. Namely, the following lemma holds.

\begin{lemma}\label{trdeglem}
Let $K/k$ be an extension of graded fields.  We have $$\trdeg_k(K)=\trdeg_{k_1}(K_1)+{\rm dim}_\Q\left((\rho(K^\times)/\rho(k^\times))\otimes_\Z\Q\right)$$ in the sense of cardinalities.
\end{lemma}
\begin{proof}
Let $\{x_i\}$ be a transcendence basis for $K_1/k_1$ in the usual sense, and choose elements $y_j \in K^{\times}$ such that the gradings $\rho(y_j)$ form a $\Q$-basis of $(\rho(K^{\times})/\rho(k^{\times})) \otimes_{\Z} \Q$. It suffices to show that $S = \{x_i, y_j\} \subseteq K^{\times}$ is a transcendence basis of $K/k$.

First we check that $S$ is a transcendence set. If there is a nontrivial polynomial relation over $k$ satisfied by these elements then by their homogeneity we may take the relation to have homogeneous coefficients. Any monomial $c y_{j_1}^{e_1} \cdots y_{j_r}^{e_r}$ with $r > 0$, $e_1, \dots, e_r > 0$, and $c \in k^{\times}$ has grading $\rho(c) \cdot \prod \rho(y_{j_m})^{e_m}$ whose image in $(\rho(K^{\times})/\rho(k^{\times})) \otimes_{\Z} \Q$ is nonzero and uniquely determines $r$ and the ordered $r$-tuple $(e_1, \dots, e_r)$. Thus, we can decompose a hypothetical nontrivial homogeneous polynomial relation over $k$ according to the $y$-monomial gradings to get such a relation in which the $y$-contribution to each monomial in the $x$'s and $y$'s is a common term $y_{j_1}^{e_1} \cdots y_{j_r}^{e_r}$ with $r \ge 0$ and $e_m > 0$ for $1 \le m \le r$.   This can then be cancelled, so we get a nontrivial relation $\sum c_I x_I^{e_I} = 0$ with all $c_I \in k^{\times}$ having a common grading. We can then scale by $k^{\times}$ to get to the case when all $c_I \in k_1^{\times}$, contradicting that the $x_i$'s are algebraically independent in $K_1$ over $k_1$ in the usual sense.

Finally, we check that $K$ is algebraic over the graded subfield generated over $k$ by $S$.  Any $t \in K^{\times}$ satisfies $\rho(t^e) = \rho(c) \cdot \prod \rho(y_{j_m})^{e_m}$ for some $c \in k^{\times}$ and $j_1, \dots, j_r$ with $e_1, \dots, e_r > 0$ (for some $r \ge 0$). Replacing $t$ with $t^e/(c \prod y_{j_m}^{e_m})$ then allows us to assume $t \in K_1$, so we are done by the transcendence basis property for $\{x_i\}$.
\end{proof}

\subsection*{Integral closure}
It was noted in \cite[\S1]{temkin2} that a non-zero graded ring $A$ does not contain non-zero homogeneous divisors of zero (i.e., $A$ is a graded domain) if and only if it can be embedded into a graded field: the unique minimal such graded field $\Frac_G(A)$ is obtained by localizing at the multiplicative set of all non-zero homogeneous elements (and it has the expected universal mapping property). We call $\Frac_G(A)$ the {\em graded field of fractions} of $A$.

For any injective graded map $R \rightarrow R'$ between graded rings, the {\em graded integral closure} of $R$ in $R'$ is the graded $R$-subalgebra $\widetilde{R} \subseteq R'$ consisting of elements $r' \in R'$ whose homogeneous parts are integral over $R$. Note also that if $r' \in R'_g$ satisfies $f(r') = 0$ for some monic $f \in R[X]$ with degree $n > 0$ then we may replace nonzero coefficients of $f$ with suitable nonzero homogeneous parts (depending on $g \in G$) to find such an $f$ that is homogeneous in $R[g^{-1}X]$. We say that $R$ is {\em integrally closed} (in the graded sense) in $R'$ if $\widetilde{R} = R$. Also, we say that a graded domain $R$ is {\em integrally closed} if it is integrally closed in the graded sense in $\Frac_G(R)$. Now, we will prove few results about integrally closed rings that will be used later.

\begin{lemma}\label{intpolcor}
If $L/K$ is an extension of graded fields, $A$ is a graded subring of $K$ integrally closed in $K$, and $x\in L^{\times}$ is integral over $A$ with $\rho(x) = g \in G$, then the minimal homogeneous polynomial $f_x$  of $x$ over $K$ is defined over $A$; i.e. $f_x(T)\in A[g^{-1}T]\subseteq K[g^{-1}T]$.
\end{lemma}
\begin{proof}
Since $x \in L^{\times}$, minimality of $d = \deg_T(f_x)$ forces the homogeneous $f_x(0) \in K$ to be nonzero with $\rho(f_x(0)) = g^d$. Increase $L$ so that there is a homogeneous factorization $$f_x(T)=(T-x_1)\dots(T-x_d)$$ in $L[g^{-1}T]$.  Each $x_i$ is homogeneous in $L$ with $\prod x_i = \pm f_x(0) \in K^{\times}$, so $x_i \in L^{\times}$ for all $i$. By homogeneity of the factorization, $\rho(x_i) = \rho(T) = g$ for all $i$. Obviously each $x_i \in L$ has minimal homogeneous polynomial $f_x$ over $K$, so each graded subring $A[x_i]\subseteq L$ is isomorphic to $A[g^{-1}T]/(f_x) \simeq A[x]$ and hence each $A[x_i]$ is finite as an $A$-module.  Thus, the subring $A[x_1,\dots,x_d] \subseteq L$ is finite over $A$. Since the coefficients of $f_x(T)$ are contained in this latter ring, they are integral over $A$. But $A$ is integrally closed in $K$ in the graded sense and the coefficients of $f_x$ are homogeneous (or zero), so $f_x(T)\in A[g^{-1}T]$.
\end{proof}

\begin{corollary}\label{intercor}
Let $L/K$ be an extension of graded fields and let $\{A_i\}$ be a non-empty collection of graded subrings of $K$ that are integrally closed in $K$ in the graded sense. If $B_i$ denotes the graded integral closure of $A_i$ in $L$ then $\cap B_i$ coincides with the graded integral closure of $\cap A_i$ in $L$.
\end{corollary}

Note that the assumption on integral closedness of $A_i$'s is critical for the corollary to hold.

\begin{proof}
Obviously, $\cap B_i$ contains the integral closure of $\cap A_i$. Conversely, let $x$ be a nonzero homogeneous element of $\cap B_i$, with $\rho(x) = g$. Clearly, $K[x]$ has finite $K$-rank (since $x$ satisfies a monic homogeneous relation over any $A_i \subseteq K$), and the coefficients of $f_x \in K[g^{-1}T]$ lie in each $A_i$ by Lemma~\ref{intpolcor}. Hence $x$ is integral over $\cap A_i$, as claimed.
\end{proof}

\begin{lemma}\label{finiteint}
If $k$ is a graded field and $A$ is a finitely generated $k$-algebra that is a graded domain then the graded integral closure $\overline{A}$ of $A$ in $K = {\rm{Frac}}_G(A)$ is finite over $A$ $($and so it is a finitely generated $k$-algebra$)$.
\end{lemma}
\begin{proof}
Let $k' = {\rm{Frac}}_G(k[g^{-1}T_g]_{g \in G})$, so $k'/k$ is a graded extension field with $\rho({k'}^{\times}) = G$. {Since $k'$ is a graded localization of $k[g^{-1}T_g]_{g \in G}$, the finitely generated graded $k'$-algebra $A' = A \otimes_k k'$ is a graded localization of the domain $A[g^{-1}T_g]_{g \in G}$. In particular, $A'$ is a graded domain.} The graded integral closure $\overline{A'}$ of $A'$ in $K' := {\rm{Frac}}_G(A')$ contains $\overline{A} \otimes_k k'$.  Since $k'$ is a free $k$-module, $\overline{A}$ is a finite $A$-module if and only if $\overline{A} \otimes_k k'$ is a finite $A'$-module. But $A'$ is a graded noetherian ring, so such finiteness holds if $\overline{A'}$ is $A'$-finite.  Hence, we may replace $k$ with $k'$ to reduce to the case $\rho(k^{\times}) = G$.

To prove that $\overline{A}$ is $A$-finite, it suffices (by integrality of $\overline{A}$ over $A$) to show that there cannot be an infinite strictly increasing sequence of $A$-finite graded subalgebras in $\overline{A}$. Since $\rho(k^{\times}) = G$, \cite[1.1]{temkin2} gives that $R_1 \mapsto k \otimes_{k_1} R_1$ is an equivalence of categories  between $k_1$-algebras and graded $k$-algebras, and by chasing gradings we see that a map of $k_1$-algebras is integral if and only if the corresponding map of graded $k$-algebras is integral.   It is likewise clear that this equivalence respects finiteness of morphisms in both directions, and it also respects the property of being a graded domain or being finitely generated (over $k$ or $k_1$) in both directions. Hence, if there is an infinite strictly increasing sequence of $A$-finite graded domains over $A$ with the same graded fraction field as $A$ then we get an infinite strictly increasing sequence of $A_1$-finite graded domains over $A_1$ with the same ordinary fraction field as $A_1$. But $A_1$ is a domain finitely generated over $k_1$, so we have a contradiction (as the $A_1$-finiteness of the integral closure of $A_1$ is classical).
\end{proof}

\begin{lemma}\label{interlem}
Let $L/k$ be a graded field extension, $\{A_j\}$ a collection of $k$-subalgebras of $L$,
$T=\{T_i\}_{i\in I}\in L$ a set of homogeneous elements algebraically independent over each $A_j$, and
$F=\Frac_G(k[T])$. Then $\cap (F A_j)=F(\cap A_j)$.
\end{lemma}

\begin{proof}
Since $FA_j$ is the localization of $A_j[T]$ at the multiplicative set of
nonzero homogeneous elements of $F[T]$, the natural graded map $A_j\otimes_k F \rightarrow FA_j$
is an isomorphism of rings (and hence is a graded isomorphism). By the same reason,
$(\cap A_j)\otimes_k F\simeq F(\cap A_j)$, and since $F$ is a free $k$-module by \cite[1.2]{temkin2}, the lemma
now follows from the following lemma.
\end{proof}

\begin{lemma}\label{submodlem}
Let $M$ and $N$ be modules over a commutative ring $R$, with $M$ projective over $R$. Let $\{N_i\}$ be a set of
$R$-submodules of $N$. The inclusion $(\cap N_i) \otimes_R M \subseteq \cap (N_i \otimes_R M)$ inside of $N
\otimes_R M$ is an equality.
\end{lemma}
\begin{proof}
Since $M$ is a direct summand of a free module, we can assume $M$ is free. We may replace $N$ with $N/(\cap
N_i)$, so $\cap N_i = 0$.  In this case we want $\cap (N_i \otimes_R M) = 0$ inside of $N \otimes_R M$.   We
have $M \simeq \oplus_j Re_j$, so using linear projection to the factors reduces us to the trivial case $M = R$.
\end{proof}

\subsection*{Graded valuation rings}
Recall that a graded domain $\calO$ is called a {\em graded valuation ring} if for any non-zero homogeneous element $f\in\Frac_G(A)$, at least one of the elements $f,f^{-1}$ lies in $A$. As in the classical case, a graded valuation ring $\calO_v$ of a graded field $K$ gives rise to a graded valuation $|\ |_v\:K^\times\to H=K^\times/\calO^\times$ augmented by $|0|_v=0$, the divisibility relation induces an ordered group structure on $H$ and $|\ |_v$ satisfies the strong triangle inequality on elements of the same grading. Our next aim is to extend some results about valuation rings to the graded setting.

\begin{lemma}\label{flat} Let $M$ be a graded module over a graded valuation ring $\calO$, and assume $M$ is torsion-free in the graded sense, which is to say that for each nonzero homogeneous $a \in \calO$, the self-map $m \mapsto am$  of the $\calO$-module $M$ is injective. Then $M$ is $\calO$-flat, and if $M$ is finitely generated it admits a homogeneous basis.
\end{lemma}
\begin{proof}
By consideration of direct limits we may assume that $M$ is finitely generated. Let $K = {\rm{Frac}}_G(\calO)$, and let $V = K \otimes_{\calO} M$. The natural map $M \rightarrow V$ is injective, and $V$ admits a finite homogeneous $K$-basis. We prove the existence of a homogeneous basis of $M$ by induction on the $K$-rank of $V$ (which we may assume to be positive).  If the $K$-rank is 1 then by shifting the grading on $M$ we can assume that $V = K$ as graded $K$-modules and that $M$ is a finitely generated graded ideal in $\calO \subseteq K$.  Hence, this case is settled since $\calO$ is a graded valuation ring.  If the $K$-rank $n$ is larger than 1, let $L \subseteq V$ be the graded $K$-submodule spanned by a member of a homogeneous $K$-basis of $V$, so the image $M''$ of $M$ in $V/L$ admits a homogeneous $\calO$-basis. In particular, $M''$ splits off as a graded direct summand of $M$, so $M \cap L$ is identified with a complement and thus is also finitely generated over $\calO$.  By the settled rank-1 case we are done.
\end{proof}

\begin{theorem}\label{valint}
Let $K'/K$ be an integral extension of graded fields and let $\calO$ be a graded valuation ring of $K$. Let $\calO'$ denote the integral closure of $\calO$ in $K'$. Each graded prime ideal $m'$ of $\calO'$ over $m_{\calO}$ is a graded maximal ideal and the graded localization $\calO'_{m'}$ is a graded valuation ring, with $\calO'_{m'} \cap \calO' = m'$. Every graded valuation ring of $K'$ dominating $\calO$ arises in this way.
\end{theorem}
\begin{proof}
The maximality of $m'$ reduces to the fact that a graded domain that is integral over a graded field $F$ must be a graded field, the proof of which goes almost exactly as in the classical ungraded case (by using integrality to reduce to considering a graded domain that is finitely generated over $F$ as a graded $F$-module, for which there is a homogeneous $F$-basis). As for the description of the graded valuation rings extending $\calO$, in the ungraded case this is \cite[Thm.~12,~\S7]{zs} and its corollaries.  The method of proof there (including the proof of \cite[\S5,~Lemma]{zs}) adapts nearly verbatim to the graded case, due to the fact \cite[1.4(i)]{temkin2} (where the ground field $k$ plays no role) that the integral closure of $\calO$ in $K'$ is the intersection of all graded valuation rings of $K'$ containing $\calO$.
\end{proof}

Finally, we later need the following result.

\begin{lemma}\label{oint}
Let $B$ be an integrally closed graded domain, $T=\{T_i\}_{i\in I}$ a set of variables and $g=\{g_i\}_{i\in I}$ a set of elements of $G$ $($where some $g_i$'s may be equal$)$. The $G$-graded Laurent polynomial ring $B[g^{-1}T,gT^{-1}]$ is an integrally closed graded domain.
\end{lemma}
\begin{proof}
Since $B[g^{-1}T,gT^{-1}]$ is a filtered union of its Laurent subalgebras in finitely many variables, we may assume that $I=\{1,\dots,n\}$. By induction on
the number of variables we can furthermore assume $n = 1$, so in the sequel $T$ is a single variable.

Consider the graded fraction field $K=\Frac_G(B)$. First, let us establish the case when $B=K$. By Lemma~\ref{charpol} $K[g^{-1}T]$ is graded-factorial, hence the same is true for its homogeneous localization $K[g^{-1}T,gT^{-1}]$. It remains to note that the usual argument shows that any graded-factorial graded domain $A$ is integrally closed: if $a,b$ are homogeneous in $A$ and $b\nmid a$, then dividing by the gcd we can assume that $(a,b)=1$, and in the latter case there exists a prime $p$ such that $p|b$ and $p\nmid a$. Using such a $p$ one easily sees that $a/b$ does not satisfy a monic polynomial equation over $A$, and hence $A$ is integrally closed.

The above case reduces our task to proving the following claim: if a Laurent polynomial $f=\sum_i a_iT^i\in K[g^{-1}T,gT^{-1}]$ is integral over $B[g^{-1}T,gT^{-1}]$, then each $a_i$ lies in $B$. Furthermore, $B$ is the intersection of the graded valuation rings $\calO_v\subseteq K$ containing it, hence $B[g^{-1}T,gT^{-1}]=\cap_v\calO_v[g^{-1}T,gT^{-1}]$ and it suffices to prove that $a_i\in\calO_v$ for each $v$. This reduces us to the case when $B=\calO_v$ is  a graded valuation ring.

Let $|\ |_v$ be the graded valuation on $K$ corresponding to $\calO_v$ and let $|\ |_w$ be its Gauss extension to $K[g^{-1}T,gT^{-1}]$, that is, $|\sum_ia_iT^i|_w=\max_i|a_i|_v$. Obviously, it is multiplicative and hence extends to a valuation on $\Frac_G(K[g^{-1}T,gT^{-1}])$ with a graded valuation ring $\calO_w$. Note that $f$ is integral over $\calO_w$ because $B[g^{-1}T,gT^{-1}]\subseteq\calO_w$. But $\calO_w$ is integrally closed, hence $f\in\calO_w$ and by the definition of $|\ |_w$ we obtain that $|a_i|_v\le 1$ for any $i$. The latter means that $a_i\in B$, as required.
\end{proof}

\begin{corollary} \label{ointcor}
Let $L/k$ be a graded field extension, $A\subseteq  L$  an integrally closed
graded $k$-subalgebra, $T=\{T_i\}_{i\in
I}\in L$ a set of
homogeneous elements algebraically independent over $A$, and $F=\Frac_G(k[T])$. Then the graded
subring $FA\subseteq L$ generated by $F$ and $A$ is integrally closed.
\end{corollary}

\begin{proof}
Obviously $F$ is the graded fraction field of its graded subring $B$ generated by $k$ and the elements $T_i^{\pm
1}$, so $F$ coincides with the graded localization ring $B_R$, where $R$ is the set of non-zero homogeneous
elements of $B$. It follows that $FA=C_R$, where $C$ is the graded algebra generated by $A$ and the elements
$T_i^{\pm 1}$. Since graded integral closedness is preserved by graded localization (similarly to the ungraded
case), it suffices to prove that $C$ is integrally closed. Now, it remains to notice that by our assumption on
$T_i$'s, $C \simeq A[g^{-1}S,gS^{-1}]$, where $g=\{\rho(T_i)\}_{i\in I}$ and $S=\{S_i\}_{i\in I}$ is a
corresponding set of graded indeterminates. Hence, the corollary follows from Lemma \ref{oint}.
\end{proof}

\section{The category $bir_G$}\label{birgsec}

\subsection*{The absolute birational category}
In the sequel, we will have to simultaneously consider several $G$-graded birational spaces with different $G$-graded ground fields. {Here is a key example of a natural map between birational spaces with different fields of definition. Assume that $\chi\:l\otimes_kK\to L$ is a homogeneous homomorphism, where $k,l,K,L$ are $G$-graded fields. Then a commutative diagrams of $G$-graded fields arises
$$\xymatrix{
{l} \ar[r] & {L} \\
{k} \ar[u] \ar[r] & {K} \ar[u]}$$
and restriction of graded valuation rings induces a continuous map $\psi_{L/l,K/k}\:\bfP_{L/l}\to\bfP_{K/k}$. We do not include the diagram or the homomorphism $\chi$ in the notation $\psi_{L/l,K/k}$ because it will always be clear from the context.}  Note that $\psi_{L/l,K/k}$ is the composition of the obvious topological embedding $\psi_{L/l,L/k}\:\bfP_{L/l}\into\bfP_{L/k}$ and the natural map $\psi_{L/K/k}\:\bfP_{L/k}\to\bfP_{K/k}$ of birational spaces over $k$. The maps $\psi_{L/l,K/k}$ were used in \cite[5.3]{temkin2} to establish a connection between reduction of germs of the analytic spaces $X$ over $k$ and $X_K = X\wtimes_k K$ over $K$. A deeper study of these maps in this and the next sections will be used later to prove that certain properties of analytic spaces (e.g. goodness) descend from $X_K$ to $X$.

We now introduce the category $\bir=\bir_G$ of all $G$-graded birational spaces. On the level of objects, $\bir$ is just the disjoint union of all categories $\bir_l$. A morphism $f\:\varcalY\to\varcalX$ of birational spaces corresponding to respective local homeomorphisms $Y\to\bfP_{L/l}$ and $X\to\bfP_{K/k}$ is a pair of compatible graded embeddings $k\into l$ and $K\into L$ and a continuous map $Y\to X$ compatible with $\psi_{L/l,K/k}$. We naturally extend the properties of objects and morphisms that were defined in \cite[\S2]{temkin2}:

\begin{definition}\label{defbir}
In the category ${\rm{bir}}_G$ of birational spaces over $G$-graded fields, $\varcalY = (Y \rightarrow \P_{L/l})$ is {\em affine} if $Y$ maps bijectively onto an affine subset of $\bfP_{L/l}$, a morphism $(Y \rightarrow \P_{L/l}) \rightarrow (X \rightarrow \P_{K/k})$ is {\em separated} if the natural map $\phi\:Y\to X\times_{\bfP_{K/k}}\bfP_{L/l}$ is injective, and such a morphism is {\em proper} if $\phi$ is bijective and $\psi_{L/l,K/k}$ is onto.

Furthermore, we say that $\varcalX = (X \rightarrow \P_{K/k})$ is  {\em separated} (resp. {\em proper}) if $X\to\bfP_{K/k}$ is injective (resp. bijective), which is to say that the canonical morphism $\varcalX \rightarrow \bfP_{K/k}$ in ${\rm{bir}}$ is separated (resp. proper).
\end{definition}

\begin{lemma}\label{composlem}
For birational spaces $\mathcal{Z} = (Z\to\bfP_{M/m})$, $\mathcal{Y} = (Y\to\bfP_{L/l})$, and $\mathcal{X} = (X\to\bfP_{K/k})$ and morphisms $h\:\mathcal{Z}\to \mathcal{Y}$, $g\:\mathcal{Y}\to \mathcal{X}$, and $f=g\circ h$, the following properties hold.

$(1)$ If $g$ and $h$ are separated, then $f$ is separated.

$(2)$ If $f$ is separated then $h$ is separated.

$(3)$ If $g$ and $h$ are proper, then $f$ is proper.

$(4)$ If $f$ and $h$ are proper, then $g$ is proper.

$(5)$ If $f$ is proper, $g$ is separated, and $\psi_{M/m,L/l}$ is surjective $($e.g., $m=l$$)$ then $h$ is proper.

$(6)$ If $f$ is separated and $h$ is proper then $g$ is separated.
\end{lemma}
\begin{proof}
The proof is based on the same set-theoretic argument as we used in the proof of Lemma \ref{57lemma}. Consider the following diagram $$Z {\stackrel{\gamma}{\rightarrow}} Y \times_{\mathbf{P}_{L/l}} {\mathbf{P}}_{M/m} {\stackrel{\beta'}{\rightarrow}} X\times_{\mathbf{P}_{K/k}} \mathbf{P}_{M/m}$$ where $\beta'$ is the base change by $\psi_{M/m,L/l}$ of the natural map $\beta\:Y\to X\times_{\mathbf{P}_{K/k}} \mathbf{P}_{L/l}$. Separatedness/properness of $f$, $g$, and $h$ are related to injectivity/bijectivity of the maps $\alpha:=\beta'\circ\gamma$, $\beta$, and $\gamma$ respectively. In (1), we are given that $\beta$ and $\gamma$ are injective. Hence the base change $\beta'$ is injective, and so $\alpha$ is also injective, proving (1). Obviously, injectivity of $\alpha$ implies injectivity of $\gamma$, thus proving (2). In (3), we are given that the maps $\psi_{M/m,L/l}$ and $\psi_{L/l,K/k}$ are surjective, hence their composition $\psi_{M/m,K/k}$ is also surjective. Also, since $\beta$ and $\gamma$ are bijective, so are $\beta'$ and $\alpha$ in this case. In particular, $f$ is proper. In (4), we deduce bijectivity of $\beta'$ from the bijectivity of $\alpha$ and $\gamma$. Since $\psi_{M/m,L/l}$ is surjective by the properness of $h$, the bijectivity of $\beta'$ implies bijectivity of $\beta$. Thus, $g$ is proper. It remains to establish (5) and (6).  In either case the map $\psi_{M/m,L/l}$ is surjective, so $g$ is separated if and only if $\beta'$ is injective. Thus, (5) is the obvious claim that $\gamma$ is bijective when $\alpha$ is bijective and $\beta'$ is injective, and (6) is the obvious claim that $\beta'$ is injective when $\alpha$ is injective and $\gamma$ is bijective.
\end{proof}

{
\subsection*{Surjectivity of $\psi$}
We, now, turn to studying the maps $\psi_{L/l,K/k}$. Our main goal is to describe when such a map is surjective, since such surjectivity is an ingredient of our definition of properness in $bir_G$. For this we will need the following notion: assume that $L/k$ is a graded field extension and $l,K\subseteq L$ are two graded subfields containing $k$. We say that $l$ and $K$ are in {\em general position} with respect to $k$ if for any pair of graded subfields $l'\subseteq l$ and $K'\subseteq K$ containing $k$ and finitely generated over it, the equality $\trdeg_k(l'K')=\trdeg_k(l')+\trdeg_k(K')$ holds. It will be convenient to express this condition in the following asymmetric form: $l$ and $K$ are in general position if and only if any algebraically independent set over $k$ in $l^{\times}$ is algebraically independent over $K$. Moreover, it suffices to check this condition for a single transcendence basis for $l$ over $k$, and if $\trdeg_k(l)<\infty$ this happens if and only if the inequality $\trdeg_K(lK)\le\trdeg_l(k)$ is an equality.

\begin{remark}\label{genrem}
(1) For $l$ and $K$ to be in general position it is insufficient to assume the weaker condition that $l \cap K$ is algebraic over $k$. For example, in the ungraded case there are infinite subgroups $\Gamma \subseteq {\rm{PGL}}_2(k)$ generated by a finite collection of non-trivial elements $\gamma_1, \dots, \gamma_n$ with finite order (using $n = 2$ if ${\rm{char}}(k) = 0$ and $n = 3$ otherwise), so there exists a pair of finite subgroups $H_1, H_2 \subseteq {\rm{PGL}}_2(k)$ that generate an infinite group. Thus, $L = k(t)$ is a finite Galois extension of rational subfields $l = k(t)^{H_1} = k(x_1)$ and $K = k(t)^{H_2} = k(x_2)$, yet $l \cap K = k$.

(2) It is easy to see that for any extensions $l/k$, $K/k$ there exist $k$-embeddings of $l$ and $K$ into a graded $k$-field $L$ such that $l$ and $K$ are in general position in $l$ over $k$. For example can take $L$ to be the algebraic closure of a field whose transcendence basis over $k$ is the union of transcendence bases of $l$ and $K$. Also, decreasing $L$ we can achieve that $L=lK$.

(3) In fact, one can show that $l$ and $K$ are in general position if and only if the kernel of the homomorphism $l\otimes_kK\to L$ is a minimal homogeneous prime ideal, or, that is equivalent, the image of $\Spec_G(L)$ in $\Spec_G(l\otimes_kL)$ is a generic point. In the ungraded case, this is a simple exercise on the dimension theory of varieties, and the graded case is dealt with similarly.
\end{remark}
}

\begin{theorem}\label{psilem}
Assume that $L$ is a graded field with graded subfields $l$, $K$ and $k\subseteq l\cap K$, and consider the map associated map $\psi_{L/l,K/k}$. Then,

$(1)$ $\psi_{L/l,K/k}$ is surjective if and only if $l$ and $K$ are in general position with respect to $k$.

$(2)$ $\psi_{L/l,K/k}$ has finite fibers $($resp. is injective$)$ if the extension $L/K$ is finite $($resp. trivial$)$;

$(3)$ $\psi_{L/l,K/k}$ is bijective if $l/k$ is algebraic, $L_1=K_1$, and $\rho(L^\times)/\rho(K^\times)$ is a torsion group.
\end{theorem}
\begin{proof}
We start with the third claim. Note that any element of $L^\times$ is a power of an element of $K^\times$. Hence any graded valuation ring $\calO$ in $\bfP_{K/k}$ uniquely extends to an element $\calO'$ of $\bfP_{L/k}$ via the following rule: $f$ is a homogeneous element of $\calO'$ if and only if $f^n\in\calO$ for some $n>0$. This proves that $\bfP_{L/k}=\bfP_{K/k}$. Choose an element $\calO'\in\bfP_{L/k}$. Since $l$ is algebraic over $k$ and $\calO$ is integrally closed in $L$ in the graded sense, we obtain that $l\subseteq\calO$, and hence $\bfP_{L/l}=\bfP_{L/k}$.

The second claim follows from the fact that any graded valuation ring $\calO$ on $K$ admits at most $[L:K]$ extensions to a graded valuation ring on $L$. To prove this fact, we adapt the classical ungraded argument. By Theorem \ref{valint}, it is equivalent to check that the graded integral closure $\widetilde{\calO}$ of $\calO$ in $L$ has at most $[L:K]$ graded prime ideals over the graded maximal ideal of $\calO$. For this it suffices to show that each graded $\calO$-subalgebra $R \subseteq L$ that is $\calO$-finite and spans $L$ over $K$ has at most $[L:K]$ graded prime ideals dominating the graded maximal ideal of $\calO$. The $\calO$-module $R$ is free by Lemma \ref{flat} so the graded algebra $R/m_{\calO}R$ over the graded residue field $\calO/m_{\calO}$ is a free module of rank $[L:K]$.  We are therefore reduced to checking that if $F$ is a graded field and $A$ is a graded finite $F$-algebra whose underlying $F$-module has rank $n$ then $A$ has at most $n$ graded prime ideals.  Note that all such graded prime ideals are maximal, since a graded domain of finite rank over $F$ is necessarily a graded field.  Thus, if $F'/F$ is any graded extension field and $A' = A \otimes_F F'$ then $\Spec_G(A')$ maps onto $\Spec_G(A)$, so it suffices to treat the pair $(A', F')$.  By choosing $F'$ such that $\rho({F'}^{\times}) = G$ (e.g., the graded fraction field of the graded polynomial ring $F[g^{-1}T_{g}]_{g \in G}$) we can thereby reduce to the case that each homogeneous nonzero $a \in A$ satisfies $\rho(a) \in \rho(F^{\times})$. Consequently, by \cite[1.1(ii)]{temkin2} the natural map $F \otimes_{F_1} A_1 \rightarrow A$ is an isomorphism and the operation $I \mapsto I \cap A_1$ identifies the ideal theory of $A$ with that of $A_1$ respecting primality.  Since $A_1$ is of rank $n$ over the ordinary field $F_1$, we are reduced to the obvious classical case.

The first claim is the most difficult one. Assume the fields are not in general position. Then there exists a finite subset $t=\{t_1\..t_n\}\in K^\times$ which is algebraically independent over $k$ and is algebraically dependent over $l$. Choose linearly independent elements $r_1\..r_n\in\bfR_{>0}^\times$ and extend the trivial valuation on $k$ to $k(t)$ as the generalized Gauss valuation such that $|t_i|=r_i$ (thus a homogeneous element $\sum_{j\in J}a_jt^j$ with non-zero $a_j$'s is of valuation $\max_{j\in J}r^j$). Extending this valuation from $k(t)$ to $K$ in an arbitrary way yields an element $|\ |\in\bfP_{K/k}$, which is easily seen to be not in the image of $\psi_{L/l,K/k}$. Indeed, otherwise $|\ |$ can be extended to a valuation on $L$ trivial on $l$, which restricts to a valuation on $l(t)/l$ violating the Abhyankar inequality: $(|l(t)^\times|/|l^\times|)\otimes\Q=|l(t)^\times|\otimes\Q$ is of dimension at least $n$, while $\trdeg_l(l(t))<n$. The classical proof of Abhyankar's inequality (for example, see \cite[Lemma~2.1.2]{temst}) applies in this case without changes, though it even suffices to check the particular case that $t_1\..t_n$ are algebraically independent over $l$ because $|t_1|\..|t_n|$ are linearly independent over $|l^\times|=1$. So we skip details.

Let us prove the converse implication. Assume that {$l$ and $K$ are in general position over $k$.} Choose a transcendence basis $T=\{T_i\}$ of $l$ over $k$, so $T \subseteq L^{\times}$ is algebraically independent over $K$. For $l_0=\Frac_G(k[T]) \subseteq L$ we have {$\psi_{L/l,K/k} = \psi_{L/l_0,K/k} \circ \psi_{L/l,L/l_0}$, and $\psi_{L/l,L/l_0}$ is bijective by (3).} Thus, by replacing $l$ with $l_0$ we can assume that $l=\Frac_G(k[T])$ with $T$ an algebraically independent set over $K$.   The inclusion of graded subfields $k \subseteq K \cap l$ in $K$ must be an equality, since any $x \in (K \cap l)^{\times}$ not in $k$ is transcendental over $k$ (from the explicit description of $l$) yet is also visibly algebraic over $K$.

We can now lift a graded valuation ring $\calO\in\bfP_{K/k}$ to an element of $\bfP_{L/l}$, as follows. Consider the corresponding valuation $|\ |$ on $K$. By our assumption $K[g^{-1}T]$ with $g=\rho(T)$ is a graded polynomial subalgebra of $L$, hence the Gauss valuation provides an extension of $|\ |$ to $K(T):=\Frac(K[g^{-1}T])$ such that $|k^\times|=1$ and $|T_i|=1$. It follows that this valuation on $K(T)$ restricts to the trivial valuation on $l$, hence extending it from $K(T)$ to $L$ in an arbitrary way, we obtain a lift of $|\ |$ to an element of $\bfP_{L/l}$.
\end{proof}

\subsection*{$H$-strict birational spaces}
We conclude this section with a brief discussion of $H$-strictness of birational spaces. Let $\varcalX$
correspond to $X\to\bfP_{K/k}$ and let $H\supseteq\rho(k^\times)$ be a subgroup of $G$. Then $\varcalX$ is
called {\em $H$-strict} if it admits a proper morphism to a birational space $\varcalX_H =
(X_H\to\bfP_{K_H/k})$,
where $K_H$ denotes the $G$-graded subfield $K_H := \oplus_{g\in H}K_g \subseteq K$ over $k$.
Thus, a separated $\varcalX$ is $H$-strict if and only if the corresponding open subset of $\bfP_{K/k}$ is the
exact preimage of its image in $\bfP_{K_H/k}$. It is proved in \cite[2.6,~2.7]{temkin2} that for any $H$-strict
$\varcalX$, the space $\varcalX_H$ and the proper morphism $\varcalX\to\varcalX_H$ are unique up to unique
isomorphism. A given $\varcalX$ is $H$-strict if and only if it admits an {\em $H$-strict structure}, which is
an open covering of $\varcalX$ by $H$-strict separated subspaces whose pairwise intersections are also
$H$-strict.  (This corresponds to choosing an open covering of $\varcalX_H$ and forming its preimage in
$\varcalX$.) By the uniqueness up to unique isomorphism, any two $H$-strict structures on $\varcalX$ are
equivalent in the sense that the pairwise intersections among their members are $H$-strict.

\begin{remark}\label{Hrem}
Unfortunately, a simple criterion of $H$-strictness is missing. Given a birational space $\varcalX$ corresponding to $X\to\bfP_{K/k}$ let $S_n$ denote the set of all points of $\bfP_{K/k}$ that have precisely $n$ preimages in $X$. Clearly, $S_n$ form a stratification of $\bfP_{K/k}$, and it is easy to see that each $S_n$ is constructible, i.e. it is an element of the boolean algebra generated by open quasi-compact subsets of $\bfP_{K/k}$. If $\varcalX$ is $H$-strict then each $S_n$ is $H$-strict, i.e. it is the preimage of a constructible subset of $\bfP_{K_H/k}$. We do not know the answer to the natural question if the converse is also true, but the positive answer would provide a nice criterion of $H$-strictness. Currently, it is only available in the trivial case when $\varcalX$ is separated and so $S_n=\emptyset$ for $n>1$.
\end{remark}

\section{Descent for birational spaces}\label{descsec}

We study descent on $G$-graded birational spaces in this section, and later our results will
be applied via the
reduction functor to study descent on analytic spaces.

\begin{theorem}\label{birdescent}
Let $\varcalY=(h\:Y\to\bfP_{L/l})$ and $\varcalX=(g\:X\to\bfP_{K/k})$ be two birational spaces
equipped with a proper
morphism $f\:\varcalY\to\varcalX$ and let $\bfP$ be any of the following properties
of objects in ${\rm{bir}}$: separated, affine,
separated and $H$-strict for $\rho(l^\times)\subseteq H\subseteq G$. Then $\varcalY$ satisfies $\bfP$ if and
only if $\varcalX$ satisfies $\bfP$.
\end{theorem}

It seems probable that $H$-strictness descends without the separatedness assumption, but it is not clear how to attack this problem. If the answer to the question in Remark~\ref{Hrem} is positive then our proof below extends to the non-separated case straightforwardly.

\begin{proof}
Only the descent implications (from $\varcalY$ to $\varcalX$) require a proof.
Descent of separatedness is
purely set-theoretic:  given that $Y=X\times_{\bfP_{K/k}}\bfP_{L/l}$,
so $h$ is the base change of
$g$ under the surjective map $\bfP_{L/l}\to\bfP_{K/k}$, clearly $g$ is injective if (and only if) $h$
is.

We switch now to descent of $H$-strictness in the separated case. Consider the commutative diagram
$$
\xymatrix{
\bfP_{L/l} \ar[r] \ar[d]  & \bfP_{K/k} \ar[d] \\
\bfP_{L_H/l} \ar[r] & \bfP_{K_H/k}}
$$
in which the bottom side is surjective since the top and right sides are surjective.
We claim that the natural map $\psi\:\bfP_{L/l}\to\bfP_{K/k}\times_{\bfP_{K_H/k}}\bfP_{L_H/l}$ is
surjective.   Choose graded valuation rings
$\calO'$, $\calO$, and $\calO''$ corresponding to compatible points
in $\bfP_{K/k}$, $\bfP_{K_H/k}$, and $\bfP_{L_H/l}$ respectively.
The natural graded map $K \otimes_{K_H} L_H \rightarrow L$
is clearly injective (by consideration of graded parts).  Moreover,
the natural map $\calO' \otimes_{\calO} \calO'' \rightarrow K \otimes_{K_H} L_H$
is injective because $K \otimes_{K_H} L_H = K \otimes_{\calO} L_H$
and $\calO \rightarrow \calO''$ and $\calO \rightarrow K$ are flat (Lemma \ref{flat}).
The corresponding tensor product of graded residue fields
is nonzero, so by choosing a graded prime ideal of this latter tensor product
we get a graded prime ideal $P$ of $\calO' \otimes_{\calO} \calO''$
that dominates $m_{\calO'}$, $m_{\calO}$, and $m_{\calO''}$.
The graded localization $(\calO' \otimes_{\calO} \calO'')_P$
is a graded-local subring of $L$ that contains $l$, so it is dominated
by a graded valuation ring $R \in \bfP_{L/l}$.
Clearly $\psi(R) = (\calO', \calO'')$, establishing the surjectivity of $\psi$.

Since $\varcalY$ is assumed to be separated, $\varcalX$ is separated by descent of separatedness.
Thus, we can
identify $\varcalX$ with an open subspace $X$ in $\bfP_{K/k}$. If $\varcalX$ is not $H$-strict
 then we can find two points
$x\in X$ and $x'\in\bfP_{K/k}\setminus X$ sitting over a point $x_H\in\bfP_{K_H/k}$. Choose $y_H\in\bfP_{L_H/l}$
over $x_H$, so by surjectivity of $\psi$ we can find points $y,y'\in\bfP_{L/l}$ sitting over $(y_H,x)$ and
$(y_H,x')$, respectively. Then $y\in Y$ and $y'\notin Y$, but their images in $\bfP_{L_H/l}$ coincide.
Thus,
$\varcalY$ is not $H$-strict.  This establishes descent of $H$-strictness in the separated case.

The deepest and most useful property is being affine, and dealing with it makes use of much of
the
preliminary work done in \S\ref{gradsec}. Assume that $\varcalY$ is affine,
so we can identify $\varcalY$ with an
affine subset of $\bfP_{L/l}$, and we can
identify $\varcalX$ with a subset of $\bfP_{K/k}$ using the established descent of
separatedness. The first step is
to reduce to the case when the extensions
$K/k$ and $L/l$ are finitely generated. Choose a finite open covering of
the quasi-compact $\varcalX$ by open affine subsets
$\varcalX_i=\bfP_{K/k}\{A_i\}$ with each $A_i$ a finitely generated graded $k$-subalgebra
of $K$.  Let $K_0$ be the graded subfield of
$K$ generated by the $A_i$'s.
Clearly $\varcalX$ is the preimage of an open set $\varcalX_0\subseteq\bfP_{K_0/k}$,
so $\varcalX$ is affine if $\varcalX_0$ is affine.
The natural morphism $\varcalX\to\varcalX_0$ is proper, so it induces a proper
morphism $\varcalY\to\varcalX_0$.
We can therefore replace $\varcalX$ with
$\varcalX_0$, achieving that $K/k$ is finitely generated. At this stage we
may and do rechoose the affine
sets $\varcalX_i$'s so that $K$ is the graded fraction field of each $A_i$.
(This is done by choosing a finite set of elements $t_1, \dots, t_n \in K^{\times}$
such that $K = {\rm{Frac}}_G(k[t_1,\dots,t_n])$ and adjoining to the $A_i$'s
various homogeneous elements $t_1^{\epsilon_1}, \dots, t_n^{\epsilon_n}$ with
$\epsilon_j = \pm 1$.) Moreover, we can replace each $A_i$
with its graded integral closure $\oA_i$ in $K$ because this procedure does not affect $\varcalX_i$
and $\oA_i$ is
finitely generated over $k$ by Lemma \ref{finiteint}.

Similarly, $\varcalY$ is the preimage in $\bfP_{L/l}$ of an affine subset $\varcalY_0\subseteq\bfP_{L_0/l}$ for a
finitely generated subextension $L_0/l$,
so by replacing $L_0$ with the composite $L_0K$ (which is also finitely
generated over $l$) we achieve that the morphism $\varcalY\to\varcalX$ factors through $\varcalY_0$,
and by Lemma \ref{composlem}(4) the resulting morphism $\varcalY_0 \rightarrow
\varcalX$ is necessarily proper.
Thus, we can assume that $L/l$ is finitely generated as well, and then
by finiteness of graded integral closures (Lemma \ref{finiteint}) we have
$\varcalY=\bfP_{L/l}\{B\}$ for an integrally closed finitely generated graded
$l$-subalgebra $B$ of $L$ (but $\Frac_G(B)$ can be smaller than $L$).

Next, we choose any transcendence basis $\{T_j\}_{j\in J}$ of $l$ over $k$, and let $l_0$ be the
graded subfield of $l$
generated by $k$ and the $T_j$'s. Since $l$ is algebraic over $l_0$, the map
$\psi_{L/l,L/l_0}\:\bfP_{L/l}\to\bfP_{L/l_0}$ is a bijection by Theorem \ref{psilem}(3). Moreover, if
$B=l[b_1,\dots,b_m]$ with homogeneous $b_1, \dots, b_m \in L^{\times}$
then the image of $\varcalY$ in $\bfP_{L/l_0}$ is the affine set
$\varcalY_0=\bfP_{L/l_0}\{b_1,\dots,b_m\}$. We
again get a natural proper morphism $\varcalY_0\to\varcalX$, so once
again we can replace $\varcalY$ with $\varcalY_0$, this time achieving that $l$ is purely transcendental over
$k$.

Note that
the $T_j$'s are algebraically independent over $K$ because of Theorem
\ref{psilem}(1), so
Corollary \ref{ointcor} applies to the $A_i$'s and $l$, giving that the graded rings $lA_i$ are integrally
closed and have a common graded fraction field $\Frac_G(lK)$. If $B_i$ denotes the integral closure of $lA_i$ in
$L$ then $\cap B_i$ is the integral closure of $\cap lA_i$ in $L$ by Corollary \ref{intercor}. Now we make a few
observations: $\varcalY_i=\bfP_{L/l}\{B_i\}$ is the preimage of $\varcalX_i$ in $\bfP_{L/l}$; $B=\cap B_i$
because $B$ is an integrally closed graded $l$-subalgebra
and $\varcalY = \bfP_{L/l}(B)$ is equal to $\cup\varcalY_i =
\cup \bfP_{L/k}(B_i) \subseteq \bfP_{L/l}(\cap B_i)$
(so $B = \cap_{\calO \in \varcalY} \calO \subseteq \cap_{\calO \in \varcalY_i} \calO = B_i$
for all $i$, and the containment $B \subseteq \cap B_i$ is an
equality due to integrality of $\cap B_i$ over $B$ that follows from
$\bfP_{L/l}(B)$ lying in $\bfP_{L/l}(\cap B_i)$);
by Lemma \ref{interlem}, $\cap lA_i=lA$ for $A := \cap A_i$. Summarizing this,
we obtain that $B$ is integral over $lA$.

We have to be careful when working with $A = \cap A_i$:
it could a priori happen (without taking $\varcalY$ into account)
that $A$ is not finitely generated over $k$ (e.g., for ungraded $k$ one can construct such an example using that
there exists a $k$-variety $X$ with ${\rm{H}}^0(X,\calO_X)$ not a finitely generated $k$-algebra). However, since the
graded $l$-subalgebra $B$ in $L$ is finitely generated, we can find a finitely generated graded $k$-subalgebra
$A_0\subseteq A$ such that the integral closure of $A_0$ in $L$ contains $B$, hence coincides with it. Thus,
$\varcalY=\bfP_{L/l}\{B\}$ is the preimage in $\bfP_{L/l}$ of $\bfP_{K/k}\{A_0\}$, so we must have
$\varcalX=\bfP_{K/k}\{A_0\}$ by the surjectivity of the map $\bfP_{L/l}\to\bfP_{K/k}$ (in particular $A$, which
is the integral closure of $A_0$ in $K$, is finitely generated over $k$).
\end{proof}

\section{$H$-strict analytic spaces}\label{Hstrsec}
In principle, the material of this section should have been given in \cite{temkin2}, in particular, it is logically independent of \S\ref{birsec1}--\S\ref{descsec}. {Note also that $H$-strict spaces were first introduced in the first version of this paper, so we decided to keep this section despite the fact that nowadays most of its results are also proved in \cite[Chapter~3]{flat}.} From now on and until the end of the paper, we consider only $\R^{\times}_{>0}$-gradings and $H$ denotes a subgroup of $\R^{\times}_{>0}$ that contains $|k^{\times}|$.

\subsection*{Separated $H$-strict spaces}
We say that a $k$-affinoid algebra $\calA$ is {\em
$H$-strict} if the spectral radius of any its element either vanishes or belongs to the group $\surd{H}$. This
is equivalent to either of the following conditions: (1) there exists an admissible epimorphism
$k\{r_1^{-1}T_1,\dots ,r_n^{-1}T_n\}\twoheadrightarrow \calA$ with $r_1,\dots ,r_n\in H$; (2) there exists an admissible
epimorphism $k\{r_1^{-1}T_1,\dots ,r_n^{-1}T_n\}\twoheadrightarrow \calA$ with $r_1,\dots ,r_n\in\surd{H}$. Obviously (1)
implies (2), and that (2) implies $H$-strictness is well-known for strictly $k$-affinoid algebras (i.e.,
$H\subseteq\surd{|k^\times|}$). The general case is reduced to this one by making a ground field extension $K/k$
so that $|K^\times|$ contains all $r_i$'s.  Conversely, $H$-strictness implies (2) because any admissible
epimorphism $k\{r^{-1}T\} \twoheadrightarrow \calA$ with $T_i \mapsto a_i \in \calA$ where $a_i$ has spectral
radius $s_i$ factors through an admissible epimorphism $k\{s^{-1}T\} \twoheadrightarrow \calA$.  Finally, to see
that (2) implies (1), given (2) with $r_i^N = h_i \in H$ and $T_i \mapsto a_i \in \calA$ we get a finite
admissible map $k\{h^{-1}X\} \twoheadrightarrow \calA$ with $X_i \mapsto a_i^N$. Since proper affinoid maps are finite
(admissible) maps, we then easily deduce (1) via integrality and properness considerations.

A $k$-affinoid space $X=\calM(\calA)$ is called {\em $H$-strict} if the $k$-affinoid algebra $\calA$ is $H$-strict. Then for any
point $x\in X$, $H$-strict affinoid neighborhoods of $x$ form a basis of its neighborhoods
provided that $H \ne 1$. Note that
$H$-strictness is inherited by direct products and closed subspaces, so the intersection of finitely many
$H$-strict affinoid domains in any separated $k$-analytic space is $H$-strict. Also, an affinoid space $X$ is $H$-strict if and only if it admits a finite covering by $H$-strict affinoids $X_i$ because $\rho_X=\max_i(\rho_{X_i})$. Thus, the following definition makes sense: a separated $k$-analytic space is {\em
$H$-strict} if it admits a covering (for the $G$-topology on $k$-analytic spaces) by $H$-strict affinoid
domains. As in the affinoid case, a finite intersection of $H$-strict
analytic domains in a separated $k$-analytic space is $H$-strict.

\begin{lemma} \label{Hredlem}
Let $H \subseteq \R^{\times}_{>0}$ be a non-trivial subgroup containing $|k^{\times}|$.
A point $x$ in a separated $k$-analytic space $W$ has an $H$-strict (resp. $H$-strict affinoid) neighborhood $W'$ if and only if the reduction
$\tilW_x$ is an $H$-strict (resp. $H$-strict affine) birational space.
\end{lemma}

\begin{remark}\label{trivh}
{The inverse implication of the} lemma is false if $H$ is trivial, but the following weaker version
still holds:  if $U$ is separated and $H$-strict and $x\in U$ is a point, then any $H$-strict
open birational subspace $\tilW\subseteq\tilU_x$ can be obtained as the reduction at $x$ of an $H$-strict
$k$-analytic domain
$W\subseteq U$ containing $x$.
{This is proved similarly to the lemma, and one has only to modify
the direct analysis in the affine case.}
\end{remark}

\begin{proof}
We start with the direct implications. Assume first that $x$ has an $H$-strict affinoid neighborhood $W'$. So,
$W'=\calM(\calA)$ and there exists an
admissible epimorphism $k\{r_1^{-1}T_1,\dots,r_n^{-1}T_n\}\twoheadrightarrow \calA$ with $r_i\in H$. We have that
$\tilW_x=\tilW'_x=\bfP_{\wt{\calH(x)}/\tilk}\{\tilf_1,\dots,\tilf_n\}$ by \cite[\S4]{temkin2}, where $\tilf_i$ is the image of
$f_i$ in $\wt{\calH(x)}$ in degree $r_i$,
so for each $i$ either $\tilf_i$ vanishes or $\rho(\tilf_i)=r_i\in H$.  We conclude that
$\tilW_x$ is $H$-strict affine, as stated.

In general, if $x$ possesses an $H$-strict neighborhood $W'$ then we cover $W'$ by $H$-strict affinoid domains $W_i$ and note that the reduction of each $W_i$ at $x$
is $H$-strict by the above case. These reductions provide an $H$-strict covering of the separated graded birational space $\tilW_x$, hence we obtain that $\tilW_x$ is $H$-strict.

To prove the converse implications we, again, start with the affine case, so assume that $\tilW_x$ is $H$-strict and affine. By \cite[5.1]{temkin2},
$x$ possesses an affinoid neighborhood $W''=\calM(\calA)$. Fix an
admissible epimorphism $k\{r_1^{-1}T_1,\dots ,r_n^{-1}T_n\}\twoheadrightarrow \calA$ satisfying $T_i \mapsto f_i \in \calA$, and
without loss of generality assume that $r_i\in\surd{H}$ if and only if $i>m$ for some $m$. Then
$\tilW_x=\tilW''_x=\bfP_{\wt{\calH(x)}/\tilk}\{\tilf_1,\dots,\tilf_n\}$, where $\tilf_i \in \wt{\calH(x)}$ is the image of
$f_i$ in degree $r_i$, so each $\tilf_i$ either vanishes or satisfies $\rho(\tilf_i)=r_i\in\surd{H}$ by
$H$-strictness of $\tilW_x$. {(We use that if $\tilf_i \ne 0$ then the
$H$-strictness implies that $\tilf_i$ is integral over the $H$-graded field $\widetilde{\calH(x)}_H$,
and if $n$ is the degree of the minimal homogeneous polynomial for $\tilf_i$ over $\widetilde{\calH(x)}_H$
then the constant term of this polynomial is nonzero with grading $r_i^n$, and hence $r_i^n \in H$.)} It follows
that $\tilf_i=0$ for $1\le i\le m$; i.e. $r_i>|f_i(x)|$ for those $i$. Since $H \ne 1$, $\surd{H}$ is dense in
$\R^{\times}_{>0}$.  Thus, for each $1 \le i \le m$ there exist $s_i\in\surd{H}$ such that $r_i>s_i>|f_i(x)|$.
Then $W'=W''\{s_1^{-1}f_1,\dots,s_m^{-1}f_m\}$ is an $H$-strict affinoid neighborhood of $x$, as required.

Finally, if $\tilW_x$ is $H$-strict then we cover it by $H$-strict affine subspaces $\tilW_1,\dots,\tilW_n\subseteq\tilW_x$.
By the affine case, each $\tilW_i$ is the reduction at $x$ of an $H$-strict affinoid domain $W_i$, and it remains to note
that the $H$-strict domain $W'=\cup_{i=1}^nW_i$ is a neighborhood of $x$ by \cite[4.5]{temkin2}.
\end{proof}

\subsection*{$H$-strict structures}
Recall that $k$-analytic spaces are defined using a topological notion of {\em nets}, see \cite[\S1.1]{berihes}.
In general (for possibly non-separated $k$-analytic spaces),
$H$-strictness may not be preserved by intersections of separated $H$-strict
$k$-analytic domains, so we are led to
the following definition in case $H \ne 1$:
by an {\em $H$-strict structure} on an arbitrary $k$-analytic space $X$ we mean a net
$\{X_i\}$ of compact separated $H$-strict $k$-analytic domains. (The stronger condition that $X_i$'s are
$k$-affinoid leads to an equivalent definition.) We say that two $H$-strict structures $\{X_i\}_{i\in I}$ and
$\{X'_j\}_{j\in J}$ are {\em equivalent} if their union is an $H$-strict structure. This condition is equivalent
to all intersections $X_i\cap X'_j$ (which are separated but possibly non-compact
$k$-analytic domains in $X$) being $H$-strict.
This really is an equivalence relation:  if $\{X''_l\}_{l \in L}$ is a third $H$-strict
structure on $X$ with each $X'_j \cap X''_l$ also $H$-strict then for each pair $(i,l)$
the separated $k$-analytic space $X_i \cap X''_l$ is covered by the
$H$-strict overlaps $X_i \cap X'_j \cap X''_l = (X_i \cap X'_j) \cap (X'_j \cap X''_l)$
in the $H$-strict spaces $X'_j$ for varying $j$.

\begin{remark}\label{hstrrem}  Let $H \subseteq \R^{\times}_{>0}$ be a non-trivial subgroup containing
$|k^{\times}|$.

(1) The notion of $H$-strictness depends only on the group $\surd{H}$.

(2) If $H\subseteq\surd{|k^{\times}|}$
 then $H$-strictness is the usual $k$-analytic strictness.

(3) Berkovich defined in \cite[\S1.2]{berihes}
a general notion of $\Phi$-analytic space, where $\Phi$ is a (suitable)
family of $k$-affinoid spaces. His definition was mainly motivated by the case of strictly
$k$-analytic spaces, but
one checks immediately that, more generally, the class $\Phi_H$ of all $H$-strict $k$-affinoid spaces satisfies
the conditions (1)--(5) of loc.cit., and the corresponding $\Phi_H$-analytic spaces are exactly the
$k$-analytic
spaces with an $H$-strict structure.
\end{remark}

\begin{remark}\label{hstrrem2}
(1) We exclude the case $H=1$, which can only happen for trivially-valued $k$,
because $1$-strict affinoids do not satisfy the density
condition from \cite[\S1]{berihes}, so they do not form a net
in the sense of \cite[\S1.1]{berihes}. (Briefly, the trivial group
$\surd{H}$ is too small to provide a sufficiently large
collection of positive real numbers in the definition of
$H$-strict $k$-analytic subdomains.)  However, one can weaken our definition by
removing the density condition in the definition of a net.
The resulting definition of $H$-strictness then makes sense and
becomes the usual notion of strict $k$-analyticity from \cite[1.2]{berihes} in the case of
a trivially-valued field $k$.

(2) The case of $H=1$ seems to be much less important. For completeness, we will indicate in remarks which results hold
true in this case, and which modifications in the proofs are required.

(3) One can, more generally, define $H$-strictness for any submonoid $H\subseteq\R^{\times}_{>0}$ containing
$|k^{\times}|$. We do not study this case because some basic results are proved in \cite{temkin2} only when $H$
is a group. However, it seems very probable that every result stated for a group $H$ (resp. a non-trivial group
$H$) holds true for submonoids (resp. submonoids with an element $r<1$).
\end{remark}

\begin{theorem} \label{Hstrth}
Let $H \subseteq \R^{\times}_{>0}$ be a non-trivial subgroup containing $|k^{\times}|$.
If a $k$-analytic space $X$ admits an $H$-strict structure then the intersection of any two separated
$H$-strict $k$-analytic domains $($not assumed to be  compatible with the structure
on $X$$)$
is $H$-strict. In particular, all
$H$-strict structures on $X$ are equivalent and the maximal such structure consists of all compact separated
$H$-strict $k$-analytic domains.
\end{theorem}

Using Remark \ref{trivh}, this theorem and its proof can be extended to the case $H = 1$.

\begin{proof}
Given separated $H$-strict $k$-analytic domains $U$ and $V$ in $X$, we have to prove that $W=U\cap V$ is $H$-strict, and by Lemma \ref{Hredlem} this reduces to showing that $\tilW_x$ is $H$-strict for a point $x\in W$. By the same lemma, the reductions $\tilU_x$ and $\tilV_x$ are $H$-strict and the $H$-strict structure $\{X_j\}_{j\in J}$ on $X$ induces an $H$-strict structure $\{\wt{(X_j)_x}\}_{j\in J}$ on $\tilX_x$. Since $\tilX_x$ is $H$-strict, \cite[2.7]{temkin2} asserts that $H$-strictness of its subspaces is preserved by intersections, and hence $\tilW_x=\tilU_x\cap\tilV_x$ is $H$-strict.
\end{proof}

Theorem \ref{Hstrth} excludes any ambiguity from the following definition
for any (possibly trivial) subgroup $H \subseteq \R^{\times}_{>0}$ containing
$|k^{\times}|$: a $k$-analytic space is {\em
$H$-strict} if it admits an $H$-strict structure.

\subsection*{$H$-strict germs}
If $H \ne 1$ then a germ $(X,x)$ is called {\em $H$-strict} if $x$ admits an
$H$-strict neighborhood in $X$. (The latter definition makes no sense for trivial $H$ since
in that case the existence of such a neighborhood does not imply the existence of a base
of such neighborhoods, so the concept is not intrinsic to the germ $(X,x)$.) Now, since the
notion of an $H$-strict germ is defined, we can generalize Lemma \ref{Hredlem} as follows.

\begin{theorem} \label{Hredth}
Fix a non-trivial subgroup $H \subseteq \R^{\times}_{>0}$ containing $|k^{\times}|$, then

(1) A $k$-analytic space $X$ is $H$-strict if and only if the germ $(X,x)$ is $H$-strict for any $x\in X$.

(2) A germ $(X,x)$ is $H$-strict if and only if its reduction $\tilX_x$ is $H$-strict.
\end{theorem}

The theorem is of local nature, so it does not make sense for the trivial $H$.

\begin{proof}
We start with (2). We saw in the proof of Theorem \ref{Hstrth} that an $H$-strict germ has an $H$-strict reduction, so
now assume that $\tilX_x$ is $H$-strict. Find a finite covering of $(X,x)$ by separated germs $(X_i,x)$
such that their reductions $\tilX_i\subseteq\tilX_x$ are $H$-strict (one finds such a covering of $\tilX_x$ and
then lifts it to the germ). By \cite[4.8]{temkin2} we can shrink each $X_i$
to make them separated. Set
$X_{ij}=X_i\cap X_j$, so each germ $(X_{ij},x)$ has $H$-strict reduction $\tilX_i\cap\tilX_j$, and by Lemma
\ref{Hredlem} we can find an open $X'_{ij} \subseteq X_{ij}$
around $x$ that is $H$-strict.  Again using Lemma \ref{Hredlem}, we can shrink
the $X_i$'s once again so that all
$X_i$'s are $H$-strict and $X_i\cap X_j\subseteq X'_{ij}$ for any choice of $i,j$.

We claim that $X'=\cup X_i$ is an $H$-strict neighborhood
of $x$. Only $H$-strictness needs a proof, and
by definition it suffices to check that each $X_i\cap X_j$ is $H$-strict. Notice that by the
construction $X_i\cup X'_{ij}$ is separated, so $H$-strictness is inherited by the intersection $X_i\cap
X'_{ij}$. By similar reasoning,
$X_i\cap X'_{ij}\cap X_j$ is $H$-strict, but we have chosen $X_i$'s so that the
latter intersection is just $X_i\cap X_j$.

Now, let us prove (1). The direct implication is obvious, so assume that all germs are $H$-strict. Then $X$ can be covered by $H$-strict domains, so we should only prove that if $U,V\subseteq X$ are separated $H$-strict domains then the intersection $W=U\cap V$ is $H$-strict. For any point $x\in W$, the reduction $\tilW_x=\tilU_x\cap\tilV_x$ is $H$-strict by the $H$-strictness of $\tilX_x$, $\tilU_x$ and $\tilV_x$. So, (2) implies that the germ $(W,x)$ is $H$-strict, that is, $x$ possesses an $H$-strict neighborhood inside $W$. Since $W$ is separated, this implies that it is $H$-strict.
\end{proof}

Finally, for any (possibly trivial) subgroup $H \subseteq \R^{\times}_{>0}$
containing $|k^{\times}|$, one can define a suitable notion of morphism to
make a category of $H$-strict $k$-analytic spaces similar to the category of strictly
$k$-analytic spaces. One possibility for $H \ne 1$ is to apply Berkovich's definition of
morphism of $\Phi$-analytic spaces with
$\Phi=\Phi_H$ being the class of all $H$-strict $k$-affinoid spaces. We prefer a more ad hoc equivalent
definition (which has the merit of ``working'' for $H = 1$ as well):
an {\em $H$-strict morphism} $Y\to X$ between $H$-strict $k$-analytic spaces is a $k$-analytic morphism
for which the preimage of any
$H$-strict $k$-analytic subdomain of $X$ is $H$-strict in $Y$.
Note that in case $H = 1$ this recovers the notion of a strictly $k$-analytic morphism
for a trivially-valued field.

\begin{theorem} \label{Hfulth}
If $H \subseteq \R^{\times}_{>0}$ is a subgroup containing $|k^{\times}|$ then
the subcategory of $H$-strict $k$-analytic spaces with $H$-strict analytic morphisms is full in the category of
all $k$-analytic spaces.
\end{theorem}

The particular case $H=|k^{\times}|$ (including the case $H = 1$)
was proved in \cite[4.10]{temkin2}. That proof applies verbatim to the more
general situation in Theorem \ref{Hfulth}
as soon as one replaces strict analyticity (i.e., $|k^{\times}|$-strictness) with $H$-strictness.
The special case $H=1$ for good $k$-analytic spaces is a part of GAGA over a trivially-valued field
\cite[3.5.1(v)]{berbook}.

{
\subsection*{$H$-graded reduction and $H$-strict $G$-topology}
In view of Theorem \ref{Hfulth}, when working within the category of $H$-strict analytic spaces it is natural to provide each $H$-strict analytic space $X$ with the $G$-topology of $H$-strict analytic subdomains. We call the latter the {\em $H$-strict $G$-topology} and denote the associated $G$-topological space by $X_H$. In particular, the usual $G$-topological space $X_G$ coincides with $X_H$ for $H={\bfR_{>0}^\times}$. We will freely view $X_H$ as a Grothendieck category. By a {\em point} of $X_H$ we mean a prime filter of $H$-strict subdomains of $X$ as defined in \cite[p. 83, (p1)--(p3),(p4)']{points}. Intuitively, these are all domains "containing" the point. Conditions (p1)--(p3) define a filter; they are standard. The last condition means that if $V$ is in the filter and $V=\cup_i V_i$ is a {\em Tate covering} then at least one $V_i$ is in the filter. This is a subtle condition making use of the $G$-topology. We will not need this, but it is easy to see that this definition agrees with the classical one from \cite{sga4}, i.e., there is a one-to-one correspondence between the set of points of $X_H$ and the set of isomorphism classes of points of the topos associated with $X_H$.

By $|X_H|$ we denote the set of all points of $X_H$ topologized as follows. For any $H$-strict analytic subdomain $V\subseteq X$ one can identify $|V_H|$ with the subset of $|X_H|$ consisting of points $x$ such that $V\in x$. In particular, if $U$ is another $H$-strict subdomain then $|(U\cap V)_H|=|U_H|\cap|V_H|$ and we obtain that the sets $|W_H|$, where $W$ runs through all $H$-strict subdomains, form a topology base for $|X_H|$. We call $|X_H|$ the topological realization of $X_H$, and our next aim is to describe it in terms of germ reductions.

\begin{remark}
In fact, $|X_H|$ is the topological realization of the topos of $X_H$ as defined in \cite{sga4}. In addition, it follows easily from Deligne's theorem (see \cite[\S2.2]{temkin3} for details) that the topoi of $X_H$ and $|X_H|$ are canonically equivalent. Thus, it is safe and often natural to work with $|X_H|$ instead of $X_H$. It is deduced from Deligne's theorem in \cite[\S2.2]{temkin3} that $|X_H|=\coprod_x \wt{(X,x)}_H$, but we will give below an elementary direct proof of this fact, not using Deligne's theorem. The fact that $X_H$ has enough points will be obtained as a corollary.
\end{remark}

If $X$ is an $H$-strict analytic space and $x\in X$ is a point then $\tilX_x$ is an $H$-strict birational space with associated graded field $\wt{\calH(x)}$, hence it is the pullback of a uniquely defined birational space with associated graded field $\wt{\calH(x)}_H=\oplus_{g\in H}\wt{\calH(x)}_g$. We denote this latter space $\wt{(X,x)}_H$ and call it the {\em $H$-graded reduction} of the germ $(X,x)$. Note that this reduction can be constructed directly from $(X,x)$ as follows: if $(X,x)$ is good we can assume that $X$ is an $H$-strict $k$-affinoid space, say $X=\calM(\calA)$, and we define $\wt{(X,x)}_H$ to be the affine birational space associated with the image of the $H$-graded reduction $\tilcalA_H=\oplus_{g\in H}\tilcalA_g$ in $\wt{\calH(x)}_H$. In general, the reduction is defined through gluing.

\begin{theorem}\label{Hgermth}
Assume that $X$ is an $H$-strict $k$-analytic space and $x\in X$ is a point. Then the $H$-graded reduction functor establishes a one-to-one correspondence between the $H$-strict subdomains of $(X,x)$ and $H$-graded birational subspaces of $\wt{(X,x)}_H$.
\end{theorem}
\begin{proof}
By \cite[Theorem~4.5]{temkin2}, this is true when $H=\bfR_{>0}^\times$. The general case follows since a subdomain of $(X,x)$ is $H$-strict if and only if its reduction is $H$-strict, see Theorem~\ref{Hredth}(2).
\end{proof}

For any $H$-strict analytic space $X$ consider the set $\hatX_H:=\coprod_{x\in X}\wt{(X,x)}_H$, and topologize $\hatX_H$ as follows. For any $H$-strict analytic subdomain $V\subseteq X$ we can identify $\hatV_H$ with a subset of $\hatX_H$, and if $U$ is another $H$-strict subdomain and $W=U\cap V$ then $\hatW_H=\hatU_H\cap\hatV_H$ because the germ reduction functor is compatible with finite intersections. In particular, the sets $\hatW_H$ with $W$ running through all $H$-strict subdomains of $X$ form a topology base.

Given a point $(x\in X,y\in\wt{(X,x)}_H)$ of $\hatX_H$ consider the set $F_x$ of all subdomains $W\subseteq X$ such that $x\in W$ and $y\in\wt{(W,x)}_H$. It is easy to see that $F_x$ is a prime filter and hence we obtain a map $\phi_X\:\hatX_H\to|X_H|$.

\begin{theorem}\label{Htop}
Assume that $X$ is an $H$-strict $k$-analytic space, then

(1) For any point $x\in X$, the map $\wt{(X,x)}_H\into\hatX_H$ is a topological embeddings.

(2) The topological space $\hatX_H$ is sober.

(3) The map $\phi_X\:\hatX_H\to|X_H|$ is a homeomorphism.

(4) The inclusion $X\into\hatX_H$ induces equivalence of the topoi associated with $X_H$ and $\hatX_H$. In particular, $X_H$ has enough points.
\end{theorem}
\begin{proof}
(1) The topology $\hatX_H$ induces on $\wt{(X,x)}_H$ is generated by the sets $\hatW_x\cap\wt{(X,x)}_H=\tilW_x$, where $W$ runs through the $H$-strict domains containing $x$. By Theorem~\ref{Hgermth}, it coincides with the topology of $\wt{(X,x)}_H$.

(2) Assume that $Z$ is an irreducible closed subset in $\hatX_H$. We should prove that it has a single generic point. Assume first that $Z$ contains points $(x,y)$ and $(x',y')$ with $x\neq x'$. Let $U$ and $V$ be disjoint open neighborhoods of $x$ and $x'$ in $X$. Then $T=\hatX_H\setminus\hatV_H$ and $T'=\hatX_H\setminus\hatU_H$ are closed subsets of $\hatX_H$ such that $T\cup T'=\hatX_H$, $(x,y)\in T$ and $(x',y')\in T'$. Therefore, $Z\cap T$ and $Z\cap T'$ are non-empty closed sets whose union is $Z$, and this contradicts the irreducibility of $Z$. This proves that $Z\subset\wt{(X,x)}_H$ for a single $x\in X$, and it remains to use that the space $\wt{(X,x)}_H$ is sober as was recalled in the Terminology and notation.

(3) First, we claim that a family $\{X_i\}_{i\in I}$ of $H$-strict subdomains of $X$ form a Tate covering of $X$ if and only if $\cup_{i\in I}\widehat{(X_i)}_H=\hatX_H$. Indeed, if they form a Tate covering then for any $x\in X$ there exists a finite subset $J\subseteq I$ such that $x\in\cap_{i\in J}X_i$ and $\cup_{i\in J}X_i$ is a neighborhood of $x$. In this case, $\cup_{i\in J}\wt{(X_i,x)}_H=\wt{(X,x)}_H$ and therefore the spaces $\widehat{(X_i)}_H$ with $i\in I$ cover $\hatX_H$. Conversely, assume that $\widehat{(X_i)}_H$ cover $\hatX_H$. For any point $x\in X$ we have that $\cup_{i\in I}\wt{(X_i,x)}_H=\wt{(X,x)}_H$ and by the quasi-compactness of $\wt{(X,x)}_H$, we already have that $\cup_{i\in J}\wt{(X_i,x)}_H=\wt{(X,x)}_H$ for a finite subset $J\subseteq I$. But then $\cup_{i\in J}X_i$ form a neighborhood of $x$ by Theorem~\ref{Hgermth}, and we obtain that $X_i$ form a Tate covering.

Next, we claim that $\phi_X$ is a bijection. Since $\hatX_H$ is sober, its points are the prime filters of open subsets of $\hatX_H$. So, we should show that prime filters $\calF$ on $X_H$ correspond to prime filters $\hatcalF$ on $\hatX_H$. The sets of the form $\hatW_H$, where $W$ runs through $H$-strict analytic domains, form a base of topology of $\hatX_H$. Therefore, each $\hatcalF$ is determined by all sets $\hatW_H$ it contains, and using that Tate coverings in $X_H$ correspond to the usual coverings in $\hatX_H$, we obtain the bijective correspondence of filters.

Finally, the topologies of $\hatX_H$ and $|X_H|$ are generated by the bases $\{\hatW_H\}$ and $\{|W_H|\}$, where $W$ runs through $H$-strict analytic domains, hence the bijection $\phi_X$ is a homeomorphism.

(4) By the definition, the category $X_H$ is cofinal in the category of open subsets of $\hatX_H$. We have shown above that the coverings in $X_H$ and $\hatX_H$ are the same, i.e. the Grothendieck topologies agree. It follows immediately, that the topoi are naturally equivalent.
\end{proof}

\begin{remark}\label{adicrem}
In the extreme cases one obtains the following spaces: if $H=\sqrt{|k^\times}|$ then $X$ is strictly analytic and $|X_H|$ is (the topological space of) the Huber's adic space, see \cite{hub}, corresponding to $X$, if $G=\bfR_{>0}^\times$ then $|X_G|$ is the reified adic space of Kedlaya, see \cite{reified}.
\end{remark}

\section{Applications to boundaryless descent}
As a first application of the theory of graded birational spaces we can now handle descent of goodness and $H$-strictness (allowing the case $H=1$) through morphisms with surjective interior.

\begin{theorem}\label{goodcase}
Let $f\:X\to Y$ be a surjective $k$-analytic morphism without boundary , and let $H \subseteq \R^{\times}_{>0}$ be a subgroup containing $|k^{\times}|$. The following properties hold for $Y$ if and only if they hold for $X$: (i) good, (ii) locally separated, (iii) locally separated and $H$-strict, (iv) locally separated and strictly $k$-analytic.
\end{theorem}
}
\begin{proof}
We start with (i). The very definition of a  morphism being without boundary includes the requirement
that the fiber product $X \times_Y Z$ is good whenever $Z$ is good, so
in particular if $Y$ is good and $f$ is without boundary then it is a tautology that $X$ is good.
For the more interesting descent claim, we can assume that $X$ is good.
To prove that $Y$ admits a $k$-affinoid neighborhood around an arbitrary
$y \in Y$, first choose $x\in X$ over $y$.  By \cite[5.2]{temkin2} the reduction morphism
$\tilX_{x}\to\tilY_y$ is proper in $\bir_\tilk$, so by Theorem \ref{birdescent} the birational space
$\tilX_{x}$ is affine if and only if $\tilY_y$ is so. But goodness for the germ $(X,x)$ is
equivalent to affineness for the birational space $\tilX_{x}$
by \cite[5.1]{temkin2}, and similarly for $(Y,y)$
and $\widetilde{Y}_y$, so we are done.

{It now suffices to deal with (iii), since (ii) and (iv) are the extreme special cases $H=\bfR^\times_{>0}$
and $H = |k^{\times}|$, respectively.} First we assume that $H$ is non-trivial, so Theorem
\ref{Hredth} may be invoked.  For each $x \in X$ the map $\widetilde{X}_x
\rightarrow \widetilde{Y}_{f(x)}$ in ${\rm{bir}}_{\widetilde{k}}$ is proper
by \cite[5.2]{temkin2}, so in particular it is separated.
Hence, by \cite[4.8(iii)]{temkin2} the map $f$ is separated near $x$,
so if $Y$ is locally separated then so is $X$.  If in addition
$Y$ is $H$-strict then $\widetilde{Y}_{f(x)}$ is $H$-strict by Theorem \ref{Hredth},
yet $\widetilde{X}_x = \widetilde{Y}_{f(x)} \times_{\P_{\widetilde{\calH(f(x))}/
\widetilde{k}}} \P_{\widetilde{\calH(x)}/\widetilde{k}}$ (by properness)
so $\widetilde{X}_x$ is $H$-strict too.  Hence, by Theorem \ref{Hredth}(2) we deduce
that $(X,x)$ is $H$-strict, and then $X$ itself is $H$-strict by part (1) of the same theorem.

For the converse when $H \ne 1$, we assume that $X$ is locally separated and $H$-strict,
and we wish to deduce the same two properties for $Y$.
{As above, Theorem \ref{Hredth}(1) implies that} our problem is intrinsic to each germ $(Y,y)$ for
$y\in Y$.  Pick $x\in X$ over such a $y$.  Once again the reduction morphism
$\tilX_x\to\tilY_y$ is proper in $\bir_\tilk$ and $\tilX_x$ is $H$-strict
and separated.    The $H$-strictness and separatedness of the germ $(Y,y)$ is
equivalent to $H$-strictness and separatedness of the birational space
$\widetilde{Y}_y$ (again using Theorem \ref{Hredth} for the $H$-strictness),
and this pair of properties is inherited from
$\tilX_x$ by Theorem \ref{birdescent}.   Once again, taking $H = |k^{\times}|$ settles
the case of strict $k$-analyticity when $k$ is not trivially-valued.

To handle the case $H = 1$ (so $k$ is trivially-valued and $H$-strictness means
strict $k$-analyticity), the above arguments permit us to restrict attention to
the case when $X$ and $Y$ are separated.
To move the property of strict $k$-analyticity between $Y$ and $X$,
we can use the preceding arguments by replacing
Theorem \ref{Hredth} with Remark \ref{trivh}.
\end{proof}

Theorem \ref{goodcase} immediately implies the following slightly more general claim.

\begin{corollary}\label{goodcor}
The direct implications in Theorem \ref{goodcase} hold whenever $f$ is only assumed to be a morphism with a surjective interior.
\end{corollary}
\begin{proof}
If $X$ satisfies one of the properties (i)--(iv), then its open subspace $\Int(X/Y)$ satisfies them too and we can apply Theorem~\ref{goodcase} to the surjective boundaryless morphism $\Int(X/Y)\to Y$.
\end{proof}

\section{Descent with respect to extensions of the ground field}\label{lastsec}
The main problem with descent through a ground field extension $K/k$ are issues related to $G$-surjectivity, and the main obstacle can be easily pointed at: it is not clear at all that the maps $X_K=X\wtimes_k K\to X$ are $G$-surjective. We will first prove a key lemma that settles the problem, and then apply it to construct the descent.

\subsection*{Test points}
The most straightforward attempt to prove $G$-surjectivity of $h\:Y=X_K\to X$ is for any point $x\in X$ to choose a lift $y\in Y$ and consider the map $\tilh_y\:\tilY_y\to\tilX_x$. Surprisingly, it often happens that $\tilh_y$ is not proper: although $\tilY_y\to\bfP_{\wt{\calH(y)}/\tilK}$ is the set-theoretic base change of $\tilX_x\to\bfP_{\wt{\calH(x)}/\tilk}$ by \cite[5.3]{temkin2}, the map $\psi_y=\psi_{\wHy/\tilK,\wHx/\tilk}$ does not have to be surjective. So, our main task is to show that there always exists a lift $y=x_K$ of $x$ such that $\psi_y$ is onto, and by Theorem~\ref{psilem}(1) this amounts to taking $y$ so that $\wHy$ is ``generic enough''. Naturally, such points $x_K$ will later be used to descent properties: we will see that $X$ satisfies a property $\bfP$ at $x$ if and only if $X_K$ satisfies $\bfP$ at $x_K$, while the descent may fail for other points of $h^{-1}(x)$. Unlike the descent results of \S\ref{setsec} based on a family of quasi-Cartesian diagrams, this time it is critical  to use a single test point $x_K$, but its existence follows easily from what we have already proved.

\begin{lemma}\label{genpointcor2}
Let $X$ be a $k$-analytic space with a point $x$, $K/k$ an analytic field extension, and $Y=X\wtimes_k K$. Then,

(1) There exists a point $y\in Y$ over $x$ such that the graded subfields $\wt{\calH(x)}$ and $\tilK$ of $\widetilde{\calH(y)}$ are in general position over $\tilk$.

(2) For any point $y$ as in (1), the induced morphism $\tilY_y\to\tilX_x$ is proper in $\bir_{\R^{\times}_{>0}}$.

(3) The map $h\:Y\to X$ is $G$-surjective.
\end{lemma}
\begin{proof}
(1) By Remark~\ref{genrem}(2) there exists a point $t\in\Spec_G(\wHx\otimes_\tilk\tilK)$ such that the images of $\wHx$ and $\tilK$ in $k(t)$ are in general position over $\tilk$. By Theorem~\ref{surjlem}, $t$ possesses a lift $y$ in $\calM(\calH(x)\wtimes_k K)=h^{-1}(x)$.

(2) The map $\psi_{\widetilde{\calH(y)}/\widetilde{\calH(x)}, \widetilde{K}/\tilk}$ is surjective by Theorem~\ref{psilem}(1) and $\tilY_y\to\bfP_{\wt{\calH(y)}/\tilK}$ is the base change of $\tilX_x\to\bfP_{\wt{\calH(x)}/\tilk}$ by \cite[5.3]{temkin2}.

(3) For any $x\in X$ find $y$ as in (1) and use that $\tilY_y\to\tilX_x$ is onto by (2).
\end{proof}

It sounds plausible that in the strictly analytic case $H=\sqrt{|k^\times|}$ the analogue of claim (iii) for $|Y_H|\to|X_H|$ can be easily established using adic spaces, but we could not find such a result in the literature. Let us also give two typical examples of the fiber $h^{-1}(x)$, which show that it can be rather complicated.

\begin{example}\label{fiberexam1}
Let $X=\calM(k\{r^{-1}T\})$ be a closed disc of radius $r > 0$, let $x$ be its maximal point (corresponding to the spectral norm on $k\{r^{-1}T\}$), and let $K=\calH(x)$. The relative boundary $\partial(X/\calM(k))$ consists of the single point $x$, and the relative boundary of $X_K=\calM(K\{r^{-1}T\})$ over $\calM(K)$ consists of a single point $x_K$ lying over $x$. The fiber $Z$ of $X_K$ over $x$ is isomorphic to $\calM(K\wtimes_k K)$, so it has many points in general. For example, if $r\notin\surd{|k^{\times}|}$ then $K = k\{r^{-1}T\}$ and $Z$ is isomorphic to a closed disc over $K$, but if $r=1$ then $Z$ is large but not $K$-affinoid: as a subset of the closed unit $K$-disc  it is ``not defined over $k$". In both cases $x_K$ is a point of the fiber over $x$ that is ``as generic as possible'' and one can easily check that $\psi_{x_K}$ is proper. For any other choice of $y\in h^{-1}(x)$, the space $\tilY_y$ is proper while $\tilX_x$ is not. In view of Lemma~\ref{composlem}(4), $\tilh_y$ is not proper and hence $\psi_y$ is not surjective.
\end{example}

\begin{example}\label{fiberexam2}
Let $X=\calM(k\{r^{-1}T\})$ be a disc. Assume for simplicity that $k$ is algebraically closed and let $x\in X$ a point of type $4$ (see \cite[\S1]{berbook}) equal to the intersection of closed discs whose radii tend to some $s$ from above.  (The arithmetic of the field $\calH(x)$ depends on whether or not $s \in \surd{|k^{\times}|}$.) Let $K = \calH(x)$ over $k$. It is easy to see that $\calH(x)\wtimes_k K\simeq K\{s^{-1}S\}$ for $S=T\otimes 1- 1\otimes T$, so the fiber of $X_K$ over $x$ is a disc over $K$. Note that the graded residue fields $\tilk$ and $\tilK = \wt{\calH(x)}$ are isomorphic, but the graded field $(\calH(x)\wtimes_k K)^{\sim}$ is not algebraic over $\tilk$.
\end{example}

\subsection*{Properties of spaces}
Now, we will use the points constructed in Corollary \ref{genpointcor2} to descend properties of analytic spaces.

\begin{theorem}\label{strictK}
Let $K/k$ be an analytic field extension and $H\subseteq\R_{>0}^\times$ any $($possibly trivial$)$ subgroup containing $|K^{\times}|$. A $k$-analytic space $X$ is good $($resp. $H$-strict and locally separated$)$ if and only if the $K$-analytic space $X_K$ is good $($resp. $H$-strict and locally separated$)$. In particular, if $|K^{\times}|/|k^{\times}|$ is a torsion group then $X$ is strictly $k$-analytic and locally separated if and only if $X_K$ is strictly $K$-analytic and locally separated.
\end{theorem}
\begin{proof}
Let $\bfP$ be the property of being good (resp. $H$-strict and locally separated) and $\tilbfP$ be the property of an $\bfR_{>0}^\times$-graded birational space being affine (resp. $H$-strict and separated). If $X$ satisfies $\bfP$ then $X_K$ obviously satisfies $\bfP$ too. Conversely, assume that $X$ does not satisfy $\bfP$ locally at a point $x$. By Corollary~\ref{genpointcor2}, we can find a preimage $y\in Y:=X_K$ of $x$ so that $\tilY_y\to\tilX_x$ is proper. By Theorem \ref{birdescent}, $\tilY_y$ satisfies $\tilbfP$ if and only if $\tilX_x$ does. Also, by \cite[4.8(iii),~5.1]{temkin2} for local separatedness and goodness and Theorem \ref{Hredth} (which requires $H \ne 1$) for $H$-strictness,  $X$ (resp. $Y$) satisfies $\bfP$ locally at $x$ (resp. $y$) if and only if $\tilX_x$ (resp. $\tilY_y$) satisfies $\tilbfP$, at least if we require $H \ne 1$. Hence, assuming $H \ne 1$, $X$ satisfies $\bfP$ at $x$ if and only if $Y$ satisfies $\bfP$ at $y$, and since we assumed that $X$ is non-$\bfP$ at $x$ we conclude that $X_K$ is non-$\bfP$ at $y$, so $X_K$ does not satisfy $\bfP$. Taking $H = |k^{\times}|$ settles the case of strict analyticity when $k$ is not trivially-valued and $|K^{\times}|/|k^{\times}|$ is a torsion group.

It remains to show that if $K$ is trivially-valued and $X_K$ is strictly $K$-analytic and locally separated then $X$ is strictly $k$-analytic and locally separated.  The preceding argument with local separatedness shows that $X$ is locally separated, so we can assume $X$ is separated.  We may then replace Theorem~\ref{Hredth} with Remark~\ref{trivh} to carry over the above argument in the case of trivially-valued $k$ and $K$.
\end{proof}

\subsection*{Properties of morphisms}
{Finally, we use a similar technique to descend properties of morphisms.

\begin{theorem}\label{finited}
Let $K/k$ be an analytic field extension, let $g\:X \rightarrow Y$ be a morphism of $k$-analytic spaces, and let $g_K\:X_K \rightarrow Y_K$ be the induced $K$-analytic morphism. Each of the properties (i)--(xx) from the Introduction holds for $g$ if and only if it holds for $g_K$.
\end{theorem}
\begin{proof}
The argument closely copies the proof of Tate-flat descent in Theorem \ref{Tateflatth} and only the $G$-set-theoretic block (v)--(viii) requires a (substantial) additional effort. Let us first briefly indicate the easier arguments. To avoid subscripts, it will be convenient to denote $g_K$ also by $g'\:X'\to Y'$.

One descends properties (i)--(iv) similarly to the proof of Theorem \ref{surjth}. This time one uses that the map $Y_K\to Y$ is surjective and diagonals and differentials are compatible with the ground field extension: $\Omega_{X_G/Y_G}\wtimes_kK=\Omega_{X'_G/Y'_G}$ and $\delta_{g'}=(\delta_g)_K$.

One descends properties (ix)--(xiv) similarly to the proof of Theorem \ref{compsurth}. This time one uses the obvious fact that the map $X_K\to X$ is properly surjective.

One descends properties (xv)--(xx) similarly to the proofs of Theorems \ref{flatth} and \ref{Tateflatth}. This time one uses that flatness descends through ground field extensions by \cite[4.5.6]{flat}.

Finally, let us consider properties (v)--(viii). As in the proof of Theorem \ref{Gsurth}, it suffices to deal with $G$-surjectivity and being without boundary. The former is clear since the map $Y_K\to Y$ is $G$-surjective by Corollary~\ref{genpointcor2}(3). So, it remains to show that if $g'$ is without boundary then so is $g$, and the latter can be checked locally at a point $x\in X$. Choose a lift $x'\in X'$ of $x$ as in Corollary~\ref{genpointcor2}(1), and set $y=g(x)$ and $y'=g'(x')$. Since the subfields $\tilK$ and $\wt{\calH(x)}$ of $\wt{\calH(x')}$ are in general position over $\tilk$, we also have that the subfields $\tilK$ and $\wt{\calH(y)}$ of $\wt{\calH(y')}\subseteq\wt{\calH(x')}$ are in general position over $\tilk$. Indeed, any subset $T\subset\tilK^\times$ algebraically independent over $\tilk$ is also algebraically independent over $\wt{\calH(x)}$, and hence also over $\wt{\calH(y)}$. This shows that the lift $y'$ of $y$ also satisfies the condition of Corollary~\ref{genpointcor2}(1), and by Corollary~\ref{genpointcor2}(2) we obtain that the reduction maps $\tilX'_{x'}\to\tilX_x$ and $\tilY'_{y'}\to\tilY_y$ are proper. The map $\tilX'_{x'}\to\tilY'_{y'}$ is proper by \cite[5.2]{temkin2}, hence the composition $\tilX'_{x'}\to\tilY'_{y'}\to\tilY_y$ is proper by Lemma~\ref{composlem}(iii), and the map $\tilX_y\to\tilY_y$ is proper by Lemma~\ref{composlem}(iv). Using \cite[5.2]{temkin2} again we obtain that $g$ has no boundary at $x$.
\end{proof}
}
\subsection*{An application}
As an application of Theorem \ref{finited}, we can use the rigid-analytic theory of ampleness \cite{relamp} to
set up a parallel theory in the $k$-analytic case (without imposing goodness requirements).  We begin with a
definition, in which $\mathbf{P}(V)$ for a finite-dimensional $k$-vector space $V$ is the $k$-analytic space
associated to the algebraic projective space ${\rm{Proj}}({\rm{Sym}}(V))$; it represents the functor of
invertible sheaves $\mathscr{L}$ for the $G$-topology (on a varying $k$-analytic space $X$) equipped with a
surjection $V \otimes_k \mathscr{O}_{X_G} \rightarrow \mathscr{L}$.

\begin{definition}
An invertible sheaf $\mathscr{L}$ for the $G$-topology $X_G$ on a proper $k$-analytic space $X$ is {\em ample}
if there exists an $n > 0$ such that the map  $\Gamma(X_G, \mathscr{L}^{\otimes n}) \otimes_k \mathscr{O}_{X_G}
\rightarrow \mathscr{L}^{\otimes n}$ of coherent $\mathscr{O}_{X_G}$-modules is surjective and the resulting
morphism $X \rightarrow \mathbf{P}(\Gamma(X_G, \mathscr{L}^{\otimes n}))$ is a closed immersion.

If $f\:X \rightarrow S$ is a proper map of $k$-analytic spaces then an invertible $\mathscr{O}_{X_G}$-module
$\mathscr{L}$ is {\em relatively ample} with respect to $f$ if $\mathscr{L}_s = \mathscr{L}|_{X_s}$ is ample on
the fibral $\mathscr{H}(s)$-analytic space $X_s$ for every $s \in S$.
\end{definition}

For a $k$-analytic space $S$ and a coherent $\mathscr{O}_{S_G}$-module $\mathscr{E}$,
we will use the $S$-proper $k$-analytic space  $\mathbf{P}(\mathscr{E})$
that classifies invertible $\mathscr{O}_{X_G}$-modules
equipped
with a surjection from $\mathscr{E} \otimes_{\mathscr{O}_{S_G}} \mathscr{O}_{X_G}$
(where $X$ is a varying $k$-analytic space over $S$),
exactly as for schemes. Via the universal property and gluing for the
$G$-topology \cite[1.3.3]{berihes}, to construct $\mathbf{P}(\mathscr{E})$ it suffices to do this for $k$-affinoid
$S$ provided that it is compatible with $k$-affinoid base change.   Relative analytification over affinoid
algebras in the sense of \cite[2.6.1]{berihes} provides such a construction over an affinoid base
$\mathscr{M}(\mathscr{A})$ by using the corresponding algebraic construction  over $\Spec(\mathscr{A})$. Via the
universal property and the behavior of relative analytification with respect to closed immersions, the formation
of $\mathbf{P}(\mathscr{E})$ commutes with any base change on $S$ and surjections $\mathscr{E}' \rightarrow
\mathscr{E}$ on $S_G$ induce closed immersions $\mathbf{P}(\mathscr{E}) \hookrightarrow
\mathbf{P}(\mathscr{E}')$ over $S$. In particular, this shows that $\mathbf{P}(\mathscr{E})$ admits a closed
immersion into a standard projective space locally over $S_G$. Hence, $\mathbf{P}(\mathscr{E})$ is $S$-proper
since this property is clear when $S$ is $k$-affinoid, so it holds locally for the
$G$-topology on $S$ in general, and properness is local for this topology \cite[5.6]{temkin2}.

\begin{corollary}\label{corample}
Let $f\:X \rightarrow S$ be a proper map of $k$-analytic spaces and $\mathscr{L}$ be an invertible
$\mathscr{O}_{X_G}$-module.
\begin{enumerate}
\item The set $U_{\mathscr{L}}$ of $s \in S$ such that $\mathscr{L}_s$ is
ample on the $\mathscr{H}(s)$-analytic space
$X_s$ is open and its formation commutes with $k$-analytic base change on $S$
and with any analytic extension of the ground field.
\item If $\mathscr{L}$ is relatively ample then locally on $S$
there exists $n_0 > 0$ such that $f^{\ast}(f_{\ast}(\mathscr{L}^{\otimes n})) \rightarrow
\mathscr{L}^{\otimes n}$ is surjective and
the natural map $\iota_n\:X \rightarrow {\mathbf{P}}(f_{\ast}(\mathscr{L}^{\otimes n}))$
is a closed immersion for all $n \ge n_0$.
\end{enumerate}
\end{corollary}

\begin{proof}
The crucial fact we have to show is that if $S = \mathscr{M}(k)$ and $K/k$ is an analytic extension field then
$\mathscr{L}$ is ample on $X$ if and only if the associated coherent pullback $\mathscr{L}_K$  is ample on
$X_K$. A ground field extension does not affect whether or not a map between coherent sheaves for the
$G$-topology is surjective, and by Theorem \ref{finited} the property of a morphism being a closed immersion is
likewise unaffected.  Hence, the only problem is to show that for a coherent $\mathscr{O}_{X_G}$-module
$\mathscr{F}$ (such as $\mathscr{L}^{\otimes n}$ for a fixed $n > 0$) the natural map $K \otimes_k
\Gamma(X_G,\mathscr{F}) \rightarrow \Gamma((X_K)_G, \mathscr{F})$ is an isomorphism.   More generally, we claim
that $K \otimes_k {\rm{H}}^i(X_G, \mathscr{F}) \rightarrow {\rm{H}}^i((X_K)_G, \mathscr{F}_K)$ is an isomorphism
for any $i \ge 0$.  Observe that the ordinary tensor products here may be replaced with completed tensor
products, since the cohomology is finite-dimensional (the infinite-dimensional case is considered in \cite[Theorem~A.5]{stein}).  It suffices to prove that the maps in the \v{C}ech
complex associated to a finite affinoid covering of a proper analytic space and a coherent sheaf for the
$G$-topology are admissible (in the sense of having closed images whose subspace and quotient topologies
coincide).    This property is unaffected by a ground field extension (using completed tensor products), so by
the relationship between strictly analytic spaces and rigid spaces \cite[1.6.1]{berihes} we may pass to the
strictly analytic case and hence to coherent sheaves on proper rigid spaces (the equivalence with rigid-analytic
properness is \cite[4.5]{temkin2}). In this case the desired property of the \v{C}ech complex was proved by
Kiehl \cite[2.5]{kiehlproper} in his proof of coherence of higher direct images.

Now we prove the first part of the corollary. It follows from the invariance under a ground field extension that
the formation of the set $U_{\mathscr{L}}$ is compatible with a ground field extension $K/k$ in the sense that
$\pi^{-1}(U_{\mathscr{L}}) = U_{\mathscr{L}_K}$ where $\pi\:S_K \rightarrow S$ is the canonical map.   Since
$\pi$ is also topologically a quotient map (it is even a compact surjection), it therefore suffices to solve the
problem after a ground field extension.  The formation of $U_{\mathscr{L}} \subseteq S$ is certainly local for
the $G$-topology on $S$, so by using a compact $k$-analytic neighborhood of an arbitrary point $s \in S$ we see
that it suffices to treat the case when $S$ is compact (so $X$ is compact).  Hence, by using a ground field
extension we can assume that $|k^{\times}| \ne \{1\}$ and $X$ and $S$ are strictly $k$-analytic.  In this case
there is a proper map of quasi-compact and quasi-separated rigid spaces $f_0\:X_0\rightarrow S_0$ corresponding
to $f$ and an invertible sheaf $\mathscr{L}_0$ on $X_0$ corresponding to $\mathscr{L}$. By \cite[3.2.9]{relamp},
there is a subset $U_{\mathscr{L}_0} \subseteq S_0$ that is a Zariski-open subset in a canonical Zariski-open
subset $W_{\mathscr{L}_0} \subseteq S$ such that the points of $U_{\mathscr{L}_0}$ are exactly the $s \in S$
such that $\mathscr{L}_0$ has ample pullback to $(X_0)_s = (X_s)_0$ (in the sense of rigid geometry) and such
that the formation of $W_{\mathscr{L}_0}$ and $U_{\mathscr{L}_0}$ is compatible with arbitrary ground field
extension $K/k$.    Since ampleness on a fiber is unaffected by passage between the rigid-analytic and
$k$-analytic categories, it follows that if we let $W_{\mathscr{L}} \subseteq S$ be the Zariski-open subset
corresponding to $W_{\mathscr{L}_0} \subseteq S_0$ then the Zariski-open subset of $W_{\mathscr{L}}$
corresponding to $U_{\mathscr{L}_0}$ is equal to $U_{\mathscr{L}}$.  This establishes the openness of
$U_{\mathscr{L}}$, and so finishes the proof of the first part.

To prove the second part we may again reduce to the case when $S$ is compact.  The formation of higher direct
images with respect to $f$ (using the $G$-topology) is compatible with any ground field extension, by
essentially the same argument we used above for cohomology over a field: we may pass to the case of an affinoid
base, and we use that Kiehl's results on \v{C}ech complexes are valid in the relative setting over an affinoid
base (not just over a ground field as base). Thus, once again using Theorem \ref{finited} for the property of being a closed
immersion, we may assume $|k^{\times}| \ne \{1\}$ and that $S$ and $X$ are strictly $k$-analytic. The analogue
of our desired result was proved locally on $S_0$ in the rigid-analytic case in
\cite[3.1.4,~3.2.4,~3.2.7]{relamp}.  (The ability to get the closed immersion property for all large $n$ is
shown in the proof of \cite[3.2.7]{relamp}.)   Since $S_0$ is quasi-compact, we therefore get a single $n_0$
such that $\mathscr{L}_0^{\otimes n}$ is generated by $(f_0)_{\ast}(\mathscr{L}_0^{\otimes n}) =
(f_{\ast}(\mathscr{L}^{\otimes n}))_0$ and the resulting map $(\iota_n)_0$ is a closed immersion for all $n \ge
n_0$. Passing back to the $k$-analytic category gives the desired result over $S$.
\end{proof}

An immediate consequence of the second part of the preceding corollary is that locally on $S_G$ (or locally on
$S$ when $S$ is good) a high power of a relatively ample line bundle is the pullback by $\mathscr{O}(1)$
relative to a closed immersion into a standard projective space over the base.  In particular, locally on $S$ a
sufficiently high power of a relatively ample line bundle satisfies the familiar cohomological vanishing and
generation properties for higher direct images against a fixed coherent sheaf on $X_G$. In the rigid-analytic
case this is \cite[3.2.4]{relamp}, but note that the present approach does not reprove this result in the
rigid-analytic  case since that result is a crucial part of the rigid-analytic ingredients used in the proof of
Corollary \ref{corample}.

\bibliographystyle{amsalpha}
\bibliography{descent}

\end{document}